\documentclass[11pt, reqno]{amsart}

\usepackage{amsmath, amsthm, amssymb, amsfonts}
\usepackage{eucal}
\usepackage{hyperref}
\usepackage{bm}
\usepackage[all,cmtip]{xy}
\usepackage{enumitem}

\usepackage[top=3.75cm, bottom=3cm, left=3cm, right=3cm]{geometry} 
\numberwithin{equation}{section} 

\makeatletter
\renewcommand\part{%
   \if@noskipsec \leavevmode \fi
   \par
   \addvspace{4ex}%
   \@afterindentfalse
   \secdef\@part\@spart}

\def\@part[#1]#2{%
    \ifnum \c@secnumdepth >\m@ne
      \refstepcounter{part}%
      \addcontentsline{toc}{part}{Part \thepart.\hspace{1em}#1}%
    \else
      \addcontentsline{toc}{part}{#1}%
    \fi
    {\parindent \z@ \raggedright
     \interlinepenalty \@M
     \normalfont
     \ifnum \c@secnumdepth >\m@ne
     \centering 
       \Large\bfseries \partname\nobreakspace\thepart
       \nobreak. 
     \fi
     \Large \bfseries { #2}%
     \par}%
    \nobreak
    \vskip 3ex
    \@afterheading}
\def\@spart#1{%
    {\parindent \z@ \raggedright
     \interlinepenalty \@M
     \normalfont
     \huge \bfseries #1\par}%
     \nobreak
     \vskip 3ex
     \@afterheading}
\makeatother

\renewcommand{\thepart}{\Roman{part}} 

\theoremstyle{plain}
\newtheorem{theorem}{Theorem}[section]
\newtheorem{corollary}[theorem]{Corollary}

\newtheorem{lemma}[theorem]{Lemma}
\newtheorem{proposition}[theorem]{Proposition}

\theoremstyle{definition}
\newtheorem{definition}[theorem]{Definition}
\newtheorem{remark}[theorem]{Remark}
\newtheorem{setup}[theorem]{Setup}
\newtheorem{example}[theorem]{Example}

\setlist[itemize]{leftmargin=*, itemsep={2pt}}
\setlist[enumerate]{leftmargin=*, itemsep={2pt}}

\newcommand{\st}{\mid} 
\newcommand{\set}[1]{\left\{ \, #1 \, \right\}}

\DeclareMathOperator{\G}{Gr}
\newcommand{\svee}{\scriptscriptstyle\vee}
\newcommand{\bPv}{\bP^{\svee}}
\newcommand{\bGv}{\bG^{\svee}}
\newcommand{\bLv}{\bL^{\svee}}
\newcommand{\Fl}[2]{\mathbf{Fl}_{#1,#2}}
\newcommand{\Flv}[2]{\mathbf{Fl}^{\svee}_{#1,#2}}
\newcommand{\bLp}{\bL^{+}}
\newcommand{\bLm}{\bL^{-}}
\newcommand{\bLvp}{\bL^{\svee,+}}
\newcommand{\bLvm}{\bL^{\svee,-}}

\newcommand{\bLL}{\mathbf{LL}}

\newcommand{\bMp}{\bM^{+}} 
\newcommand{\bMm}{\bM^{-}} 

\newcommand{\tbM}{\widetilde{\bM}} 

\newcommand{\Vd}{V^{\svee}}

\newcommand{\Spec}{\mathrm{Spec}}

\newcommand{\id}{\mathrm{id}}
\newcommand{\pr}{\mathrm{pr}}
\newcommand{\gp}{g^{+}}
\newcommand{\gm}{g^{-}}
\newcommand{\fp}{f^{+}}
\newcommand{\fm}{f^{-}}
\newcommand{\piv}{\pi^{\svee}}
\newcommand{\pip}{\pi^{+}}
\newcommand{\pivp}{\pi^{\svee,+}}
\newcommand{\pim}{\pi^{-}}
\newcommand{\pivm}{\pi^{\svee,-}}

\newcommand{\pminus}{p^{-}}
\newcommand{\pplus}{p^{+}}
\newcommand{\qminus}{q^{-}}
\newcommand{\qplus}{q^{+}}

\newcommand{\tzeta}{\widetilde{\zeta}}
\newcommand{\zetap}{\zeta^{+}}
\newcommand{\zetam}{\zeta^{-}}

\newcommand{\deltam}{\delta^{-}}
\newcommand{\deltap}{\delta^{+}}

\newcommand{\tcE}{\widetilde{\cE}}
\newcommand{\cEp}{\cE^{+}}
\newcommand{\cEm}{\cE^{-}}
\newcommand{\Phip}{\Phi^{+}}
\newcommand{\Phim}{\Phi^{-}}

\DeclareMathOperator{\Hom}{Hom}
\newcommand{\cHom}{\mathcal{H}\!{\it om}}
\DeclareMathOperator{\Map}{Map}

\newcommand{\cHomS}{\mathcal{H}\!{\it om}_S}
\DeclareMathOperator{\im}{im}

\newcommand{\Tor}{\mathrm{Tor}}

\newcommand{\rank}{\mathrm{rank}}

\DeclareMathOperator{\length}{\mathrm{length}}

\newcommand{\CatSt}{\mathrm{Cat_{st}}}
\newcommand{\PrCatSt}{\mathrm{PrCat_{st}}}
\newcommand{\PrCatStcg}{\mathrm{PrCat^{\omega}_{st}}}
\newcommand{\Cat}{\mathrm{Cat}}
\newcommand{\PrCat}{\mathrm{PrCat}}

\newcommand{\Ind}{\mathrm{Ind}}

\newcommand{\Perf}{\mathrm{Perf}}
\newcommand{\QCoh}{\mathrm{D_{qc}}}
\newcommand{\Db}{\mathrm{D^b_{coh}}}

\newcommand{\Fun}{\mathrm{Fun}}

\newcommand{\Funex}{\mathrm{Fun^{ex}}}
\newcommand{\FunL}{\mathrm{Fun^{L}}}
\newcommand{\FunLcg}{\mathrm{Fun^{L,\omega}}}

\newcommand{\Spaces}{\mathcal{S}}
\newcommand{\op}{\mathrm{op}}

\newcommand{\sotimes}{\otimes} 
\newcommand{\stimes}{\times}

\newcommand{\FM}[7]{\mathrm{FM}{\left(#1/#2, #3/#4, {#6} \times_{#7} {#5} \right)}}
\newcommand{\Func}[3]{\Fun_{#1}{\left(#2 ,#3 \right)}}

\newcommand{\Coh}{\mathrm{coh}}

\newcommand{\Crit}{\mathrm{Crit}}
\newcommand{\CPD}{\mathrm{CPD}}

\newcommand{\Ku}{\mathcal{K}}

\newcommand{\ev}{\mathrm{ev}}
\newcommand{\coev}{\mathrm{coev}}

\newcommand{\bone}{\mathbf{1}}

\newcommand{\Ho}{\mathrm{Ho}}

\newcommand{\tr}{\mathrm{tr}}

\newcommand{\halpha}{\widehat{\alpha}}

\newcommand{\hpr}{\widehat{\pr}}

\newcommand{\hpd}{\natural}

\newcommand{\cA}{\mathcal{A}}
\newcommand{\cAd}{\mathcal{A}^{\natural}}
\newcommand{\dcA}{{^{\natural}}\mathcal{A}}
\newcommand{\cB}{\mathcal{B}}
\newcommand{\cC}{\mathcal{C}}

\newcommand{\cD}{\mathcal{D}}
\newcommand{\cE}{\mathcal{E}}
\newcommand{\cF}{\mathcal{F}}

\newcommand{\cK}{\mathcal{K}}

\newcommand{\cO}{\mathcal{O}}

\newcommand{\cQ}{\mathcal{Q}}
\newcommand{\cR}{\mathcal{R}}
\newcommand{\cS}{\mathcal{S}}
\newcommand{\cU}{\mathcal{U}}

\newcommand{\cW}{\mathcal{W}}

\newcommand{\bG}{\mathbf{G}}
\newcommand{\bH}{\mathbf{H}}
\newcommand{\bL}{\mathbf{L}}
\newcommand{\bM}{\mathbf{M}}

\newcommand{\bP}{\mathbf{P}}

\newcommand{\bZ}{\mathbf{Z}}

\newcommand{\rD}{\mathrm{D}}

\newcommand{\rL}{\mathrm{L}}
\newcommand{\rT}{\mathrm{T}}
\newcommand{\rR}{\mathrm{R}}
\newcommand{\rS}{\mathrm{S}}

\newcommand{\fa}{\mathfrak{a}}
\newcommand{\fatw}{\mathfrak{a}'}
\newcommand{\fad}{\mathfrak{a}^{\hpd}}
\newcommand{\dfa}{{^{\hpd}}\mathfrak{a}}

\newcommand{\llangle}{\left \langle}
\newcommand{\rrangle}{\right \rangle}

\begin{document}

\title{Noncommutative homological projective duality}

\author{Alexander Perry}
\address{Department of Mathematics, Columbia University, New York, NY 10027 \smallskip}
\email{aperry@math.columbia.edu}

\thanks{This work was partially supported by an NSF postdoctoral fellowship, DMS-1606460.}

\begin{abstract}
We generalize Kuznetsov's theory of homological projective duality to 
the setting of noncommutative algebraic geometry. 
Simultaneously, we develop the theory over general base schemes, 
and remove the usual smoothness, properness, and transversality hypotheses. 
\end{abstract}

\maketitle

\tableofcontents


\section{Introduction}
\label{section-intro} 

The purpose of this paper is twofold. 
The first goal is to set up a robust framework for noncommutative algebraic geometry, 
using recent advances in higher category theory. 
The second and main goal, which builds on the first, 
is to generalize Kuznetsov's theory~\cite{kuznetsov-HPD} of 
homological projective duality (HPD) to this setting. 
These results are crucial in our work~\cite{joins} with Kuznetsov on categorical 
joins, which was one of the main motivations for this paper. 

In this introduction, we focus on the HPD part of our work. 

\subsection{Kuznetsov's HPD} 
\label{subsection-kuznetsov-HPD}
The input for HPD is a smooth projective \emph{Lefschetz variety}, which consists of 
the following data: 
\begin{itemize}
\item[--] A smooth projective variety $X$ over an algebraically closed 
field of characteristic $0$, with a morphism $X \to \bP(V)$ to 
a projective space. 
\item[--] A right Lefschetz decomposition of 
the category $\Perf(X)$ of perfect complexes.  
\end{itemize}
Let $\cO_X(H)$ denote the pullback of $\cO_{\bP(V)}(1)$. 
A right Lefschetz decomposition of $\Perf(X)$ is a semiorthogonal decomposition of 
the form 
\begin{equation*}
\Perf(X) = \llangle \cA_0, \cA_1(H), \dots, \cA_{m-1}((m-1)H) \rrangle , 
\end{equation*}
where $\cA_0 \supset \cA_1 \supset \cdots \supset \cA_{m-1}$ is a 
descending chain of categories and $\cA_i(iH)$ denotes the image of 
$\cA_i$ under the autoequivalence of $\Perf(X)$ given by tensoring with $\cO_X(iH)$.  
Many varieties admit interesting decompositions of this form, see \cite{kuznetsov-icm} for a survey. 

The key property of a Lefschetz decomposition --- from which its name derives, by 
analogy with the Lefschetz hyperplane theorem ---
is that it induces 
semiorthogonal decompositions of the base changes of $X$ along 
linear subspaces $L \subset V$. More precisely, let $s$ denote the 
codimension of $L$ in $V$. Then if $X \times_{\bP(V)} \bP(L)$ has 
the expected dimension $\dim(X) - s$, pullback along $X \times_{\bP(V)} \bP(L) \to X$ 
embeds the categories $\cA_i(iH)$ into $\Perf (X \times_{\bP(V)} \bP(L) )$ for 
$i \geq s$, and there is a semiorthogonal decomposition 
\begin{equation}
\label{sod-perf-XL}
\Perf(X \times_{\bP(V)} \bP(L))  = \llangle \Ku_L(X), 
\cA_s(H) , \cA_{s+1}(2H), 
\dots, 
\cA_{m-1}((m-s)H)  \rrangle . 
\end{equation}
The category $\Ku_L(X)$ should be thought of as the ``interesting component'' of the category of 
perfect complexes on $X \times_{\bP(V)} \bP(L)$ --- it is what is left after removing the pieces coming from the 
ambient variety $X$. 

Given the above data, Kuznetsov constructs the \emph{HPD category} $\Perf(X)^{\hpd}$, 
which is a $\bP(V^{\svee})$-linear category that can be thought of as a total space for the 
categories $\Ku_{L}(X)$ as $L$ ranges over the hyperplanes in $V$. 
The main theorem of~\cite{kuznetsov-HPD} shows that if 
$\Perf(X)^{\hpd}$ is geometric, i.e. if there exists a variety $Y$ together 
with a morphism $Y \to \bP(V^{\svee})$ and a $\bP(V^{\svee})$-linear 
equivalence $\Perf(Y) \simeq \Perf(X)^{\hpd}$, then $\Perf(Y)$ has a 
natural Lefschetz decomposition such that the ``interesting components'' 
of orthogonal linear sections of $X$ and $Y$ are equivalent.   
To be precise, we must also assume that the Lefschetz decomposition 
of $X$ is \emph{moderate}, in the sense that its length satisfies 
$m$ satisfies $m < \dim(V)$. 
This condition is quite mild and essentially always satisfied in practice, see Remark~\ref{remark-moderate}. 
Then if $\cO_Y(H')$ denotes the pullback of $\cO_{\bP(V^{\svee})}(1)$, the theorem 
is as follows. 

\begin{theorem}[\protect{\cite{kuznetsov-HPD}}]
\label{theorem-kuznetsov-HPD}
The category $\Perf(Y)$ admits a natural left Lefschetz decomposition, 
i.e. a semiorthogonal decomposition 
\begin{equation*}
\Perf(Y) = \llangle \cB_{1-n}((1-n)H'), \dots, \cB_{-1}(-H'), \cB_{0} \rrangle 
\end{equation*}
where $\cB_{1-n} \subset \cdots \subset \cB_{-1} \subset \cB_0$, 
such that the following holds. 
Let $L \subset V$ be a linear subspace, let $L^{\perp} = \ker(V^{\svee} \to L^{\svee})$ 
be its orthogonal, and let $r = \dim(L)$ and $s = \dim(L^{\perp})$. 
Assume the fiber products $X \times_{\bP(V)} \bP(L)$ and 
$Y \times_{\bP(V^{\svee})} \bP(L^{\perp})$ are of expected dimension. 
Let 
\begin{equation} 
\label{sod-perf-YLperp}
\Perf{(Y \times_{\bP(\Vd)} \bP(L^{\perp}))}  = 
\llangle 
\cB_{1-n}((r-n)H'),
\dots , \cB_{-r-1}(-2H'), \cB_{-r}(-H'), 
\Ku'_{L^{\perp}}(Y) \rrangle 
\end{equation}
be the induced semiorthogonal decomposition  
analogous to~\eqref{sod-perf-XL}. 
Then there is an equivalence of categories 
\begin{equation*}
\Ku_L(X) \simeq \Ku'_{L^{\perp}}(Y). 
\end{equation*} 
\end{theorem}

\begin{remark}
\label{remark-Db-HPD}
Kuznetsov also proves a version of the theorem for the 
bounded derived categories of coherent sheaves on $X \times_{\bP(V)} \bP(L)$ and $Y \times_{\bP(V^{\svee})} \bP(L^{\perp})$. 
Namely, these categories have semiorthogonal decompositions analogous to~\eqref{sod-perf-XL} 
and~\eqref{sod-perf-YLperp}, 
with equivalent ``interesting components''. 
\end{remark}

Theorem~\ref{theorem-kuznetsov-HPD} is the source of many semiorthogonal 
decompositions and derived equivalences in algebraic geometry, see~\cite{kuznetsov-icm, thomas-HPD} 
for surveys. 
 
Finally, as the terminology suggests, HPD is closely related to classical projective duality. 
Namely, define the projective dual $X^{\svee} \subset \bP(V^{\svee})$ 
of $X \to \bP(V)$ to be the locus of hyperplanes 
$H \in \bP(V^{\svee})$ such that the base change $X \times_{\bP} H$ is singular; 
this reduces to the usual notion when $X \to \bP(V^{\svee})$ is a closed 
immersion. 
Then Kuznetsov 
shows~\cite[Theorem~7.9]{kuznetsov-HPD} that for $Y$ as above, 
the locus of critical values of the morphism $Y \to \bP(V^{\svee})$ 
coincides with $X^{\svee}$. 

\subsection{Noncommutative HPD}
In practice, the HPD category $\Perf(X)^{\hpd}$ is 
often \emph{not} geometric in the above sense. 
Instead, usually the best one can hope for is that $\Perf(X)^{\hpd}$ is 
``close to geometric'',  e.g. equivalent to the bounded derived 
category of coherent $\cR$-modules $\Db(Y, \cR)$ on a variety 
$Y$ equipped with a sheaf of finite (noncommutative) $\cO_Y$-algebras $\cR$ 
(see \cite{kuznetsov-icm} for examples), 
or to the derived category of a gauged Landau-Ginzburg 
model (see \cite{katzarkov-GIT} for examples). 

In general, the HPD category $\Perf(X)^{\hpd}$ only has the 
structure of a ``noncommutative scheme'' over $\bP(V^{\svee})$,  
in the sense of Kontsevich. 
More precisely, using Lurie's foundational work~\cite{lurie-HA}, we define a 
``noncommutative scheme'' over a scheme $S$ to be a small 
idempotent-complete stable $\infty$-category $\cC$, equipped 
with a $\Perf(S)$-module structure; for short, we say $\cC$ is an \emph{$S$-linear category}. 
To orient the reader, we note that the category of perfect complexes 
on an $S$-scheme $X$ is an $S$-linear category, where $\Perf(S)$ acts on 
$\Perf(X)$ via pullback followed by tensor product. 

The first goal of this paper is to develop some foundational material about linear categories. 
We leave the detailed discussion of our results in this direction to the main text. 

The second and main goal, which builds on the first, is to generalize HPD to the case of categories linear 
over a fixed (quasi-compact and separated) base scheme $S$. 
Namely, for a vector bundle $V$ on $S$, we develop a theory of HPD 
where:
\begin{enumerate}
\item The input variety $X$ itself is replaced with a $\bP(V)$-linear category. 
\item There are no geometricity assumptions on the $\bP(V^{\svee})$-linear HPD category. 
\end{enumerate}
The existence of such a version of HPD is essential for our work~\cite{joins} with Kuznetsov, 
which was one of the main motivations for this paper. 

The input for our theory is a \emph{Lefschetz category over $\bP(V)$}, 
which is a $\bP(V)$-linear category $\cA$ equipped with a Lefschetz structure. 
By definition, a Lefschetz structure consists of a pair of right and left Lefschetz 
decompositions 
\begin{align*}
\cA & = \langle \cA_0, \cA_1(H), \dots, \cA_{m-1}((m-1)H) \rangle, \\ 
\cA & = \langle \cA_{1-m}((1-m)H), \dots, \cA_{-1}(-H), \cA_0 \rangle, 
\end{align*}
where 
\begin{equation*}
0 \subset \cA_{1-m} \subset \dots \subset \cA_{-1} \subset \cA_0 \supset \cA_1 \supset \dots \supset \cA_{m-1} \supset 0. 
\end{equation*} 

\begin{remark}
Technically, we also require the subcategory $\cA_i \subset \cA$ to be right admissible for $i \geq 0$ and 
left admissible for $i \leq 0$, see Definitions~\ref{definition-lefschetz-center} and \ref{definition-lc}. 
If $\cA = \Perf(X)$ for a smooth projective variety $X$ over $\bP(V)$, the above data is in fact equivalent to giving only one of 
the above Lefschetz decompositions, so a Lefschetz variety in the sense of \S\ref{subsection-kuznetsov-HPD} 
gives rise to a Lefschetz category. 
The two Lefschetz decompositions in the data of a Lefschetz structure 
give rise to ``right'' and ``left'' versions of HPD. We focus on the ``right'' one below. 
\end{remark}

Given any subbundle $L \subset V$, 
Lurie's work~\cite{lurie-HTT, lurie-HA} can be used to make sense of 
the base change category 
\begin{equation*}
\cA \otimes_{\Perf(\bP(V))} \Perf(\bP(L)) , 
\end{equation*}
which plays the role of the linear section of $\cA$ by $\bP(L)$.  
Indeed, the results of~\cite{bzfn} imply that if there exists a 
morphism of schemes $X \to \bP(V)$ and a $\bP(V)$-linear 
equivalence $\Perf(X) \simeq \cA$, 
then the above base changed category recovers $\Perf{(X \times_{\bP(V)} \bP(L))}$.\footnote{Here, 
$X \times_{\bP(V)} \bP(L)$ denotes the derived fiber product, which agrees with the 
usual scheme-theoretic fiber product if $X$ and $\bP(L)$ are $\Tor$-independent over $\bP(V)$.}
Further, there is an induced semiorthogonal decomposition 
\begin{equation}
\label{sod-CL}
\cA \otimes_{\Perf(\bP(V))} \Perf(\bP(L)) = 
\llangle \Ku_L(\cA), 
\cA_s(H) , \cA_{s+1}(2H), 
\dots, 
\cA_{m-1}((m-s)H)  \rrangle 
\end{equation}
where $s$ denotes the corank of $L \subset V$. 

We construct a $\bP(V^{\svee})$-linear HPD category $\cAd$ in the above setup, 
and prove the following version of Theorem~\ref{theorem-kuznetsov-HPD}. 
The ``right strong'' and ``left strong'' conditions appearing in the theorem are mild assumptions on 
the existence of certain adjoints (see Definitions~\ref{definition-strong} and~\ref{definition-lc}), 
which are automatic if $\cA = \Perf(X)$ for a smooth projective variety over $\bP(V)$. 
As in \S\ref{subsection-kuznetsov-HPD}, we say $\cA$ is moderate if its Lefschetz decompositions have 
length less than $\rank(V)$. 

\begin{theorem}
\label{theorem-intro-HPD}
Assume $\cA$ is right strong and moderate. 
Then the category $\cAd$ admits a left strong, moderate Lefschetz structure, such that if 
\begin{equation*}
\cAd = \llangle \cAd_{1-n}((1-n)H'), \dots, \cAd_{-1}(-H'), \cAd_{0} \rrangle 
\end{equation*}
is its left Lefschetz decomposition, then the following holds. 
Let $L \subset V$ be a subbundle, 
let $L^{\perp} = \ker(V^{\svee} \to L^{\svee})$ 
be its orthogonal, and 
let $r = \rank(L)$ and $s = \rank(L^{\perp})$. 
Let 
\begin{equation*} 
\cAd \otimes_{\Perf(\bP(V^{\svee}))} \Perf(\bP(L^{\perp})) = 
\llangle 
\cAd_{1-n}((r-n)H'), \dots, \cAd_{-r-1}(-2H'), \cAd_{-r}(-H'), 
\Ku'_{L^{\perp}}(\cAd) 
\rrangle 
\end{equation*}
be the induced semiorthogonal decomposition analogous to~\eqref{sod-CL}. 
Then there is an equivalence of categories 
\begin{equation*}
\cK_L(\cA) \simeq \cK'_{L^{\perp}}(\cAd) . 
\end{equation*}
\end{theorem}

If the categories $\cA$ and $\cAd$ are geometric, i.e. if $\cA \simeq \Perf(X)$ for 
a $\bP(V)$-scheme $X$ and $\cAd \simeq \Perf(Y)$ for a $\bP(V^{\svee})$-scheme $Y$, 
then Theorem~\ref{theorem-intro-HPD} recovers Theorem~\ref{theorem-kuznetsov-HPD}. 
In fact, even in this situation, our result is more general than Theorem~\ref{theorem-kuznetsov-HPD} 
in several respects. 
First, we work over a general base $S$. 
Second, we do not need any smoothness or projectivity hypotheses on~$X$. 
Third, using~\cite{bzfn}, 
Theorem~\ref{theorem-intro-HPD} 
implies Theorem~\ref{theorem-kuznetsov-HPD} holds without 
any expected dimension assumptions, provided the fiber products 
$X \times_{\bP(V)} \bP(L)$ and $Y \times_{\bP(V^{\svee})} \bP(L^{\perp})$ are taken 
in the derived sense; this answers a question left open in~\cite{kuznetsov-HPD}.
Further, using~\cite{bznp} we show that when the base scheme $S$ is noetherian 
and defined over a field of characteristic $0$, 
Theorem~\ref{theorem-intro-HPD} also implies 
a similar result for bounded derived categories of coherent sheaves, 
recovering the result mentioned in Remark~\ref{remark-Db-HPD}. 

As in~\cite{kuznetsov-HPD}, 
we deduce Theorem~\ref{theorem-intro-HPD} from a  
stronger result (Theorem~\ref{main-theorem}), which roughly says 
Theorem~\ref{theorem-intro-HPD} holds in families 
as $L$ varies in $\G(r, V)$. 
As a consequence, we also prove that HPD is a duality. 

Our method of proof is closely modeled on~\cite{kuznetsov-HPD}, but 
there are several difficulties to overcome in our setting. 
First, it takes some work to even formulate the objects appearing 
in the proof. 
For instance, in the framework of~\cite{kuznetsov-HPD} where 
$\cA \simeq \Perf(X)$ is geometric, a key role is played by the 
spaces of universal linear sections of $X \to \bP(V)$. 
Our basic observation here is that all of the necessary constructions can be 
made in the case where $\cA = \Perf(\bP(V))$, and then transported to 
general $\cA$ by base change. 
A more serious obstacle is that Kuznetsov uses Fourier--Mukai 
kernels in an essential way. 
Namely, all of the functors arising in his proof are of the form 
\mbox{$\Phi_{\cE} \colon \Perf(Z_1) \to \Perf(Z_2)$}, where $Z_1$ and $Z_2$ are 
schemes over a base scheme $T$ and $\Phi_{\cE}$ is the functor given   
by a Fourier--Mukai kernel $\cE \in \Perf(Z_2 \times_T Z_1)$. 
There is a precise relationship between various operations on the kernel $\cE$ 
(e.g. pullback along a morphism $f \times \id \colon Z'_2 \times_T Z_1 \to Z_2 \times_T Z_1$ 
where $f \colon Z'_2 \to Z_2$ is a morphism)  
and operations on the associated functor $\Phi_{\cE}$ 
(e.g. composition with pullback along a morphism $f \colon Z'_2 \to Z_2$). 
Kuznetsov uses this to deduce facts about functors $\Phi_{\cE}$ via the 
geometry of the kernel spaces $Z_2 \times_T Z_1$. 
To import these ideas into our setting, we develop a robust formalism of 
Fourier--Mukai kernels for certain categories arising from base change. 
Throughout, we crucially use recent advances in higher category 
theory and derived algebraic geometry, especially~\cite{lurie-HA, bzfn}.

Finally, we show that our version of HPD is closely related to classical 
projective duality. 
Namely, for any category $\cC$ linear over a base $T$, we introduce 
the notion of its \emph{critical locus} $\Crit_T(\cC)$, which is   
the locus of points in $T$ parameterizing the singular fibers of $\cC$. 
For any $\bP(V)$-linear category $\cA$ which is smooth and proper over $S$, 
we use this notion to define the \emph{classical projective dual} 
$\CPD(\cA) \subset \bP(\Vd)$. 
If $\cA$ is a Lefschetz category which is smooth and proper, we prove 
$\CPD(\cA) = \Crit_{\bP(\Vd)}(\cAd)$. 

\subsection{Further directions} 
We regard our results as part of a larger program of \emph{homological projective geometry},  
whose goal is to find categorical analogs of results from classical 
projective \mbox{geometry}, 
and to bring them to bear on the structure of derived categories. 
In this theory, Lefschetz categories over $\bP(V)$ play the role of 
projective varieties embedded in $\bP(V)$. 
The prototype of homological projective geometry is HPD. 
In~\cite{joins} we show 
the classical notion of the join of projective varieties also fits into this framework. 
Namely, we introduce categorical joins, and show they  
are related to HPD in the same way classical joins are related to classical 
projective duality. 
An interesting feature of homological projective geometry is that all 
of the operations which are known so far (i.e. HPD and categorical joins) 
preserve smoothness of the objects involved, 
whereas in classical projective geometry this is far from true 
(i.e. projective duals and joins of smooth varieties are usually singular). 

In this paper, we show that HPD can be formulated over quite general 
base schemes, which need not be defined over a field. 
Working over a field, there are a number of examples where the HPD 
category admits a close-to-geometric model~\cite{kuznetsov-icm}; 
we believe many of these results should extend to more general base schemes. 

Finally, we note that the theory developed here may hold over even more 
general bases, namely over spectral schemes in the sense of~\cite{lurie-SAG}, 
e.g. over the sphere spectrum. 

\subsection{Related work}
In his unpublished habilitation thesis~\cite{kuznetsov-hab}, 
Kuznetsov developed a version of HPD which works for input categories 
$\cC$ that are realizable (compatibly with their $\bP(V)$-linear structure) 
as an admissible subcategory of the derived category of a variety. 
Due to my inadequate Russian I was not able to read~\cite{kuznetsov-hab}, 
but its existence served as an inspiration for this paper. 

More recently, HPD has been revisited from several points of view. 
In~\cite{katzarkov-GIT}, Ballard--Deliu--Favero--Isik--Katzarkov 
use variation of GIT to realize in certain cases the HPD category 
as the derived category of a gauged Landau--Ginzburg model, 
and to give a new proof of Theorem~\ref{theorem-kuznetsov-HPD} 
in these cases. 

In~\cite{thomas-HPD}, Thomas reproved Theorem~\ref{theorem-kuznetsov-HPD} using a reinterpretation of Kuznetsov's original proof. 
He handles the case where $\Perf(X)^{\hpd}$ need 
not be geometric, but $X$ is required to be a genuine variety (not ``noncommutative'') and he works 
with a special class of Lefschetz decompositions (the ``rectangular'' ones).  

During the preparation of this paper, two other works on HPD appeared. 
In~\cite{jiang}, Jiang--Leung--Xie build on the argument 
from~\cite{thomas-HPD} to prove a generalization of 
Theorem~\ref{theorem-kuznetsov-HPD}, where the the pair $(\bP(L), \bP(L^{\perp}))$ 
is replaced by any HPD pair.  
We independently discovered (a more general form of) this result, 
which appears as an application in~\cite{joins}. 

Finally, in~\cite{jorgen} Rennemo builds on~\cite{katzarkov-GIT} to 
develop a version of HPD for categories $\cC$ that are 
admissible subcategories in the derived category of a smooth quotient 
stack. 

Our results are stronger than the above in several respects: 
we handle linear categories without any geometricity assumptions; 
we work over general base schemes and make no smoothness or properness assumptions on our categories; and 
we work at the enhanced level of stable $\infty$-categories 
(which allows us to give natural derived algebraic geometry 
interpretations in non-transverse situations). 

\subsection{Organization of the paper} 
Part~\ref{part-I} of this paper is dedicated to foundational material on 
noncommutative algebraic geometry. 
We begin in \S\ref{section-linear-categories} by introducing our framework 
for linear categories. 
In \S\ref{section-sod} we discuss semiorthogonal decompositions of such categories. 
In \S\ref{section-smooth-proper} we define the notions of 
smoothness and properness of a linear category, and among other things discuss their 
behavior under base change and semiorthogonal decompositions. 
We also introduce the notion of the critical locus of a linear category. 
In \S\ref{section-kernels} we develop our formalism of Fourier--Mukai kernels.

In Part~\ref{part-II} of the paper, we use the material from Part~\ref{part-I} to 
formulate and prove our results on HPD. 
First, in \S\ref{section-lefschetz-categories} we define 
Lefschetz categories and describe their behavior under 
passage to linear sections.  
In \S\ref{section-HPD-category} we define the HPD category, 
show that it has a natural Lefschetz structure modulo a generation statement, 
and explain its relation to classical projective duality.  
Finally, in \S\ref{section-main-theorem} 
we prove the main theorem of HPD, i.e. Theorem~\ref{theorem-intro-HPD} from above.

\subsection{Conventions} 
\label{subsection-conventions} 
All schemes are assumed quasi-compact and separated, 
and $S$ denotes a base scheme which is fixed throughout the paper. 
A vector bundle on a scheme $X$ is a finite locally free $\cO_X$-module. 
Given a vector bundle $V$ on $X$, we denote by 
\begin{equation*}
\G(r, V) \to X 
\end{equation*} 
the Grassmannian parameterizing rank $r$ \emph{subbundles} of $V$, 
and we set $\bP(V) = \G(1, V)$. In particular, 
the pushforward of $\cO_{\bP(V)}(1)$ to $X$ is $V^{\svee}$. 
The categorical conventions used in this paper are discussed in 
\S\ref{subsection-DAG}-\ref{subsection-linear-category}. 
In particular, we note that we always consider derived functors 
(pullbacks, pushforward, tensor products, etc.), but we write them with 
underived notation. 

A remark on notation: we tend to denote general linear categories by 
the letters $\cC$ or $\cD$, and Lefschetz categories or categories which are 
linear over a projective bundle by the letters $\cA$ or $\cB$. 

\subsection{Acknowledgements} 
I thank Sasha Kuznetsov for introducing me to the theory of homological projective duality, 
for many useful conversations surrounding this topic, 
and for comments on this paper. 
Part~\ref{part-II} of this paper owes a great intellectual debt to his work~\cite{kuznetsov-HPD}. 
I am also grateful to Daniel Halpern-Leistner for helpful discussions 
about $\infty$-categories and derived algebraic geometry. 

 
\newpage 
\part{Noncommutative algebraic geometry} 
\label{part-I}


\section{Linear categories}
\label{section-linear-categories} 

In this section, we introduce the notion of an $S$-linear category. 
These categories are discussed in \S\ref{subsection-linear-category}, 
after some remarks on derived algebraic geometry in 
\S\ref{subsection-DAG} and a review of some relevant aspects 
of the theory of stable $\infty$-categories in \S\ref{subsection-stable-category}. 
We also discuss ``large'' (i.e. presentable) versions of $S$-linear categories, 
which are useful for technical reasons.

\subsection{Derived algebraic geometry}
\label{subsection-DAG}
We work in the setting of derived algebraic geometry, see~\cite{lurie-SAG, gaitsgory-DAG}. 
This means we regard schemes as objects of the $\infty$-category of derived schemes.  
In particular given morphisms of schemes $X \to S$ and $Y \to S$, the symbol 
\begin{equation*}
X \times_S Y 
\end{equation*}
denotes their \emph{derived} fiber product. 
We note that this agrees with the usual fiber product of schemes whenever the 
morphisms $X \to S$ and $Y \to S$ are $\Tor$-independent over $S$. 
In fact, the only time we need to leave the category of ordinary schemes  
is in \S\ref{section-kernels} and Remark~\ref{remark-geometric-linear-section}, 
and even there derived schemes can be avoided at the cost of requiring $\Tor$-independence 
of all fiber products.

\subsection{Stable $\infty$-categories} 
\label{subsection-stable-category}
We work with stable $\infty$-categories, as developed in~\cite{lurie-HA}. 
Here we briefly review some key facts. 
We consider several classes of stable $\infty$-categories:  
\begin{itemize}
\item[--] The $\infty$-category $\CatSt$ of small idempotent-complete 
stable $\infty$-categories, with morphisms the exact functors. 
For $\cC , \cD \in \CatSt$, we denote by $\Funex(\cC, \cD)$ the 
$\infty$-category of exact functors from $\cC$ to $\cD$. 

\item[--] The $\infty$-category $\PrCatSt$ of presentable 
stable $\infty$-categories, with morphisms 
the cocontinuous functors (i.e. the functors that preserve small colimits). 
For $\cC , \cD \in \PrCatSt$, we denote by $\FunL(\cC, \cD)$ the 
$\infty$-category of 
cocontinuous 
functors from $\cC$ to $\cD$. 

\item[--] The $\infty$-category $\PrCatStcg$ of compactly generated presentable 
stable $\infty$-categories, with morphisms the cocontinuous functors that preserve 
compact objects. 
For $\cC , \cD \in \PrCatSt$, we denote by $\FunLcg(\cC, \cD)$ the 
$\infty$-category of cocontinuous functors from $\cC$ to $\cD$ 
that preserve compact objects. 
\end{itemize}

\begin{remark}
By definition, $\PrCatStcg$ is a (non-full) subcategory of $\PrCatSt$. 
\end{remark}

We are primarily interested in $\CatSt$, but for technical reasons 
it is often convenient to work with $\PrCatSt$; the category 
$\PrCatStcg$ mediates between these two. 
Namely, there is an \emph{Ind-completion} functor \cite[Section 5.5.7]{lurie-HTT}
\begin{equation*}
\Ind \colon \CatSt \to \PrCatSt 
\end{equation*}
which by construction factors through 
the inclusion $\PrCatStcg \to \PrCatSt$, and induces an equivalence 
\begin{equation*}
\Ind \colon \CatSt \to \PrCatStcg 
\end{equation*}
with inverse the functor 
\begin{equation*}
(-)^c \colon \PrCatStcg \to \CatSt 
\end{equation*}
taking $\cC \in \PrCatStcg$ to its subcategory $\cC^c \subset \cC$ of 
compact objects. 

There is a natural way to form the ``tensor product'' of 
categories in $\CatSt, \PrCatSt,$ or $\PrCatStcg$~\cite[Sections 4.8.1-4.8.2]{lurie-HA}. 
More precisely, $\PrCatSt$ is a closed symmetric monoidal $\infty$-category, 
with product denoted $\cC \otimes \cD$ and 
internal mapping objects given by $\FunL(\cC, \cD)$. 
The tensor product $\cC \otimes \cD$ is characterized by the universal property 
\begin{equation*}
\FunL(\cC \otimes \cD, \cE) \simeq \Fun'(\cC \times \cD, \cE), 
\end{equation*}
where $\Fun'(\cC \times \cD, \cE)$ is the full subcategory of 
$\Fun(\cC \times \cD, \cE)$ (the category of all functors) spanned by  
functors which preserve small colimits separately in each variable. 
The category $\PrCatStcg$ inherits a closed 
symmetric monoidal structure from $\PrCatSt$, with internal mapping objects 
given by $\FunLcg(\cC, \cD)$. 
Via the equivalence $\CatSt \simeq \PrCatSt$, we use this to equip $\CatSt$ with 
a closed symmetric monoidal structure, with internal mapping objects given by 
$\Funex(\cC, \cD)$. 
Explicitly, for $\cC , \cD \in \CatSt$, the tensor product is given by the formula 
\begin{equation*}
\cC \otimes \cD \simeq (\Ind(\cC) \otimes \Ind(\cD))^c 
\end{equation*}
and is characterized by the universal property 
\begin{equation*}
\Funex(\cC \otimes \cD, \cE) \simeq \Fun'(\cC \times \cD, \cE), 
\end{equation*}
where $\Fun'(\cC \times \cD, \cE)$ is the full subcategory of 
$\Fun(\cC \times \cD , \cE)$ spanned by functors that are exact 
separately in each variable. 

\subsubsection{Derived categories of schemes} 
Derived categories of derived schemes give examples of categories 
in $\PrCatSt$ and $\CatSt$. 
Namely, for $X$ a derived scheme, there is a category $\QCoh(X) \in \PrCatSt$ 
of unbounded complexes of quasi-coherent sheaves, and a full subcategory $\Perf(X) \in \CatSt$ 
of perfect complexes (see~\cite{bzfn}); 
if $X$ is a classical scheme, 
the homotopy categories of $\QCoh(X)$ and 
$\Perf(X)$ agree with their classical triangulated versions.
Further, $\QCoh(X)$ and $\Perf(X)$ have natural symmetric monoidal 
structures corresponding to tensor products of sheaves, by which 
$\QCoh(X)$ has the structure of a commutative algebra object in $\PrCatSt$ 
and $\Perf(X)$ has the structure of a commutative algebra object in $\CatSt$. 
We note that if $X$ is quasi-compact and separated, 
then by \cite[Proposition 3.19]{bzfn} the category $\QCoh(X)$ is compactly generated 
and there are equivalences  
\begin{equation*}
\Perf(X) \simeq \QCoh(X)^c \qquad \text{and} \qquad \Ind(\Perf(X)) \simeq \QCoh(X), 
\end{equation*} 
induced by the inclusion $\Perf(X) \subset \QCoh(X)$. 

By convention, all functors between derived categories 
(e.g. pushforward, pullback, tensor product) are derived, but will be written with 
underived notation. 
For example, for a morphism of schemes~$f \colon X \to Y$ 
we write $f^* \colon \Perf(Y) \to \Perf(X)$ for the derived pullback functor, and similarly for the functors $f_*$ and $\otimes$. 
We always work with functors defined between categories of perfect complexes. 
Note that in general, $f_*$ may not preserve perfect complexes, but it does if 
$f \colon X \to Y $ is a perfect (i.e. pseudo-coherent of finite $\Tor$-dimension) proper morphism \cite[Example 2.2(a)]{lipman}. 
This assumption will be satisfied in all of the cases where we use $f_*$ in Part~\ref{part-II} of the paper. 

Recall that $f_*$ is right adjoint to $f^*$. 
Sometimes, we need other adjoint functors as well.
Provided they exist, we denote by $f^!$ the right adjoint of $f_* \colon \Perf(X) \to \Perf(Y)$ and by 
$f_!$ the left adjoint of $f^* \colon \Perf(Y) \to \Perf(X)$, so that $(f_!,f^*,f_*,f^!)$ is an adjoint sequence. 
For instance, if $f \colon X \to Y$ is a perfect proper morphism and 
a relative dualizing complex $\omega_f$ exists and is a perfect complex on $X$, 
then $f^!$ and $f_!$ exist and are given by 
\begin{equation}
\label{eq:shriek-adjoints}
f^!(-) \simeq f^*(-) \otimes \omega_f 
\qquad \text{and} \qquad 
f_!(-) \simeq f_*(- \otimes \omega_f). 
\end{equation}
Indeed, the functor $f_* \colon \QCoh(X) \to \QCoh(Y)$ between the unbounded 
derived categories of quasi-coherent sheaves admits a right adjoint 
$f^! \colon \QCoh(Y) \to \QCoh(X)$, the relative dualizing complex of $f$ is by definition 
$\omega_f = f^!(\cO_Y)$, and the above formulas for $f^!$ and $f_!$ hold on quasi-coherent 
categories (the first holds by \cite[Proposition 2.1]{lipman} and implies the second). 
Hence if $f$ is a perfect proper morphism and $\omega_f$ is perfect, it 
follows that all of these functors and adjunctions restrict to categories of perfect complexes. 
In all of the cases where we need $f^!$ and $f_!$ in the paper, the following stronger 
assumptions will be satisfied. 

\begin{remark}
\label{remark:good-morphism}
Suppose $f \colon X \to Y$ is a morphism between schemes which 
are smooth, projective, and of constant relative dimension over $S$. 
Then $f$ is perfect, projective, and has a relative dualizing complex, which is a shift of a line bundle: 
\begin{equation}
\label{eq:intro-omega-f}
\omega_{f} = \omega_{X/S} \otimes f^*(\omega_{Y/S})^{\svee} , 
\end{equation}
where $\omega_{X/S}$ and $\omega_{Y/S}$ denote the relative dualizing \emph{complexes}. 
In particular, for such an $f$, all of the functors $f_!$, $f^*$, $f_*$, $f^!$ are defined 
and adjoint between categories of perfect complexes.
\end{remark}

\subsection{Categories linear over a base}
\label{subsection-linear-category}
In this paper, the main objects of study are categories which 
are linear over a base scheme $S$ in the following sense.  

\begin{definition}
Let $\Cat_S$ be the $\infty$-category of 
modules over the commutative algebra object $\Perf(S) \in \CatSt$, i.e 
\begin{equation*}
\Cat_S = \textrm{Mod}_{\Perf(S)}(\CatSt) 
\end{equation*}
in the notation of \cite[Section 4.5]{lurie-HA}. 
An object $\cC \in \Cat_S$ is called an \emph{$S$-linear category}. 
A morphism $\cC \to \cD$ in $\Cat_S$ is an \emph{$S$-linear functor}; 
these morphisms form the objects of an $\infty$-category denoted  
$\Fun_{\Perf(S)}(\cC, \cD)$. 
\end{definition}

Given an $S$-linear category, we use the notation  
\begin{align*}
\cC \times \Perf(S) & \to \cC  \\ 
(C, F) & \mapsto C \otimes F
\end{align*}
for the action functor of $\Perf(S)$ on $C$. 

\begin{remark}
If $\pi \colon X \to S$ is a morphism of schemes, 
then $\cC = \Perf(X)$ is naturally an $S$-linear category, 
with action functor given by $(C, F) \mapsto C \otimes \pi^*(F)$. 
This is the primordial example of an $S$-linear category. 
The philosophy of noncommutative algebraic geometry is to think of 
an arbitrary $S$-linear category $\cC$ as a ``noncommutative 
scheme with a morphism to $S$'' (although we will not use that terminology).  
Throughout the paper, 
we will see that a number of operations associated to $X \to S$ 
have analogues for an arbitrary ``noncommutative scheme'' $\cC$. 
\end{remark}

As above, we can also consider presentable versions of $S$-linear categories.   
\begin{definition}
Let $\PrCat_S$ be the category of modules over the commutative 
algebra object $\QCoh(S) \in \PrCatSt$, i.e. 
\begin{equation*}
\PrCat_S = \textrm{Mod}_{\QCoh(S)}(\PrCatSt) . 
\end{equation*}
An object $\cC \in \PrCat_S$ is called a \emph{presentable $S$-linear category}. 
A morphism $\cC \to \cD$ in $\PrCat_S$ is a \emph{cocontinuous $S$-linear functor}; 
these morphisms form the objects of an $\infty$-category denoted $\Fun_{\QCoh(S)}(\cC, \cD)$.  
\end{definition}

Recall that $\QCoh(S)$ is compactly generated because $S$ is quasi-compact 
and separated by our standing assumptions. Hence we can also make the following definition.  
\begin{definition}
Let $\PrCat_S^{\omega}$ be the category of modules over the commutative 
algebra object $\QCoh(S) \in \PrCatStcg$, i.e. 
\begin{equation*}
\PrCat_S^{\omega} = \textrm{Mod}_{\QCoh(S)}(\PrCatStcg) . 
\end{equation*}
An object $\cC \in \PrCat_S^{\omega}$ is called a \emph{compactly generated 
presentable $S$-linear category}. 
A morphism $\cC \to \cD$ in $\PrCat_S^{\omega}$ is a 
\emph{cocontinuous $S$-linear functor preserving compact objects}; 
these morphisms form the objects of an $\infty$-category denoted 
$\Fun^{\omega}_{\QCoh(S)}(\cC, \cD)$. 
\end{definition}

The categories $\Cat_S, \PrCat_S$, and $\PrCat_S^{\omega}$ are related in the 
same way as the categories $\CatSt, \PrCatSt,$ and $\PrCatStcg$. 
Namely, $\PrCat_S^{\omega}$ is a non-full subcategory of $\PrCat_S$, 
and Ind-completion induces an equivalence 
\begin{equation*}
\Ind \colon \Cat_{S} \to \PrCat_S^{\omega} 
\end{equation*}
with inverse the functor 
\begin{equation*}
(-)^c \colon \PrCat_S^{\omega} \to \Cat_{S} 
\end{equation*}
induced by taking compact objects. 

Further, the categories $\Cat_S, \PrCat_S,$ and $\PrCat_S^{\omega}$ 
admit symmetric monoidal structures (see \cite[Theorem 4.5.2.1]{lurie-HA}), 
with units respectively given by $\Perf(S)$, $\QCoh(S)$, and $\QCoh(S)$. 
These monoidal structures are closed with internal mapping objects 
given by the functor categories introduced above. 
For $\cC, \cD \in \PrCat_S$, we denote their tensor product by  
\begin{equation*}
\cC \otimes_{\QCoh(S)} \cD \in \PrCat_S. 
\end{equation*}
For $\cC, \cD \in \Cat_S$, we denote their tensor product by 
\begin{equation*}
\cC \otimes_{\Perf(S)} \cD \in \Cat_S, 
\end{equation*}
which can be described by the formula 
\begin{equation*}
\cC \otimes_{\Perf(S)} \cD \simeq (\Ind(\cC) \otimes_{\QCoh(S)} \Ind(\cD))^c. 
\end{equation*}
The tensor product of categories in $\Cat_S$ or $\PrCat_S$ can be 
characterized by a universal property.  
Namely, if $\cC, \cD, \cE \in \Cat_S$, then 
$S$-linear functors $\cC \otimes_{\Perf(S)} \cD \to \cE$ classify 
the bilinear maps $\cC \times \cD \to \cE$ (as defined in~\cite[Section 4.4.1]{lurie-HA}); 
a similar statement holds for $\PrCat_S$. 
In particular, there is a canonical functor 
\begin{equation*}
\cC \times \cD \to \cC \otimes_{\Perf(S)} \cD. 
\end{equation*}
Given objects $C \in \cC$ and $D \in \cD$, we denote by 
$C \boxtimes D \in \cC \otimes_{\Perf(S)} \cD$ their image under 
this functor. 

The following result gives generating objects for a tensor product of 
$S$-linear categories. 
We will use the following terminology. 
A \emph{stable subcategory} of a stable $\infty$-category is a 
full subcategory which is also stable. 
If $\cC$ is a category in $\CatSt$ and $\Sigma$ is a set of objects, 
then we say $\Sigma$ \emph{thickly generates} $\cC$ if the smallest 
idempotent-complete stable subcategory containing $\Sigma$ is $\cC$ itself.  

\begin{lemma}
\label{lemma-box-tensor-generation} 
Let $\cC$ and $\cD$ be $S$-linear categories. 
Then $\cC \otimes_{\Perf(S)} \cD$ is thickly generated 
by objects of the form $C \boxtimes D$ for $C \in \cC$ and $D \in \cD$. 
\end{lemma}

\begin{proof}
Equivalently,  
$\Ind(\cC) \otimes_{\QCoh(S)} \Ind(\cD)$ is compactly generated by objects of 
the form $C \boxtimes D$ for $C \in \cC$ and $D \in \cD$. 
This can be proved as in~\cite[Chapter I.1, Proposition 7.4.2]{gaitsgory-DAG}, which 
treats the analogous statement for stable $\infty$-categories without an 
$S$-linear structure. 
\end{proof}

Finally, we observe that by using tensor products, 
we can make sense of base changes of linear categories. 
Namely, if $\cC$ is an $S$-linear category and $S' \to S$ is a morphism of schemes, 
then the \emph{base change of $\cC$ along $S' \to S$} is the $S'$-linear 
category
\begin{equation*}
\cC \otimes_{\Perf(S)} \Perf(S') \in \Cat_{S'}. 
\end{equation*}  

\subsubsection{Mapping objects} 
For objects $C,D \in \cC$ of an $\infty$-category $\cC$, we 
denote by $\Map_{\cC}(C,D)$ the space of maps from $C$ to $D$. 
If $\cC$ is a presentable $S$-linear category, 
then there is a mapping object 
\begin{equation*}
\cHomS(C,D) \in \QCoh(S)
\end{equation*} 
characterized by equivalences 
\begin{equation*}
\Map_{\QCoh(S)}(F, \cHomS(C,D)) \simeq 
\Map_{\cC}(C \otimes F, D) 
\end{equation*}
for $F \in \QCoh(S)$. More precisely, the functor 
\begin{equation*}
\Map_{\cC}(C \otimes (-), D) \colon \QCoh(S)^{\op} \to \Spaces, 
\end{equation*}
where $\Spaces$ denotes the $\infty$-category of spaces, 
is representable by~\cite[Proposition 5.5.2.2]{lurie-HTT}, and 
by definition $\cHom_S(C, D)$ is the representing object. 
If $\cC$ is an $S$-linear category, then $\cC$ is a full subcategory 
of the presentable $S$-linear category $\Ind(\cC)$; 
for objects $C, D \in \cC$ we denote by $\cHom_S(C,D) \in \QCoh(S)$ 
the mapping object between $C$ and $D$ regarded as objects of $\Ind(\cC)$. 

\begin{remark}
\label{remark-mapping-object-scheme} 
Let $X$ be a scheme 
with a morphism $\pi \colon X \to S$, so that $\Perf(X)$ is $S$-linear. 
Then for $C, D \in \Perf(X)$, we have 
\begin{equation*}
\cHomS(C,D) \simeq \pi_* \cHom_X(C,D),
\end{equation*}
where $\cHom_X(C,D)$ denotes the derived sheaf $\Hom$ on $X$. 
\end{remark}

We have the following K\"{u}nneth formula for mapping objects in 
tensor products of categories. 
\begin{lemma}
\label{lemma-kunneth}
Let $\cC$ and $\cD$ be $S$-linear categories. 
If $C_1, C_2 \in \cC$ and $D_1, D_2 \in \cD$, 
then the $\QCoh(S)$-valued mapping 
object between $C_1 \boxtimes D_1$ and $C_2 \boxtimes D_2$ in $\cC \otimes_{\Perf(S)} \cD$ satisfies
\begin{equation*}
\cHom_S(C_1 \boxtimes D_1, C_2 \boxtimes D_2) 
\simeq 
\cHom_S(C_1,C_2) \otimes \cHom_S(D_1, D_2) . 
\end{equation*}
\end{lemma}

Given an $S$-linear category $\cC$, its base change along a morphism 
$S' \to S$ is an $S'$-linear category, and hence has $\QCoh(S')$-valued 
mapping objects. They satisfy the following K\"{u}nneth 
formula. 
\begin{lemma}
\label{lemma-homs-base-change}
Let $\cC$ be an $S$-linear category. 
Let $S' \to S$ be a morphism of schemes.  
If $C_1, C_2 \in \cC$ and $F_1,F_2 \in \Perf(S')$, 
then the $\QCoh(S')$-valued mapping 
object between $C_1 \boxtimes F_1$ and $C_2 \boxtimes F_2$ 
in $\cC \otimes_{\Perf(S)} \Perf(S')$ satisfies 
\begin{equation*}
\cHom_{S'}(C_1 \boxtimes F_1, C_2 \boxtimes F_2) \simeq \cHom_{S}(C_1, C_2) \otimes \cHom_{S'}(F_1,F_2) 
\in \QCoh(S'), 
\end{equation*}
where the product on the right is taken with respect to the $\QCoh(S)$-module 
structure on $\QCoh(S')$ induced by pullback. 
\end{lemma}

\subsubsection{Adjoints}
Given a functor of stable $\infty$-categories $\phi \colon \cC \to \cD$, we typically denote by 
$\phi^*$ the left adjoint whenever it exists, and by $\phi^!$ the right adjoint whenever it exists. 
Further, if $\cC$ is an $S$-linear category and $\phi \colon \cD_1 \to \cD_2$ 
is an $S$-linear exact functor of $S$-linear categories, by abuse of notation 
we denote also by
\begin{equation*}
\phi \colon \cC \otimes_{\Perf(S)} \cD_1 \to \cC \otimes_{\Perf(S)} \cD_2 
\end{equation*}
the induced functor. 

\begin{lemma}
\label{lemma-adjoint-S-linear}
Let $\phi \colon \cC \to \cD$ be an $S$-linear functor between $S$-linear categories. 
If $\phi$ has a left adjoint $\phi^* \colon \cD \to \cC$ (or right adjoint $\phi^!$) 
when regarded as a functor of plain stable $\infty$-categories, then $\phi^*$ (or $\phi^!$) 
is naturally an $S$-linear functor. 
\end{lemma}

\begin{proof}
In other words, we must show that $\phi^* \colon \cD \to \cC$ canonically commutes with 
the $\Perf(S)$-actions on $\cD$ and $\cC$. 
For objects $C \in \cC, D \in \cD, F \in \Perf(S)$, using adjointness and the $S$-linearity of $\phi$, 
we find equivalences  
\begin{align*}
\cHom_S(\phi^*(D \otimes F), C) & \simeq \cHom_S(D \otimes F, \phi(C)) \\ 
& \simeq \cHom_S(D , \phi(C) \otimes F^{\svee}) \\
& \simeq \cHom_S(D, \phi(C \otimes F^{\svee})) \\ 
& \simeq \cHom_S(\phi^*(D), C \otimes F^{\svee}) \\ 
& \simeq \cHom_S(\phi^*(D) \otimes F, C) . 
\end{align*}
Hence by Yoneda we obtain a canonical equivalence $\phi^*(D \otimes F) \simeq \phi^*(D) \otimes F$, 
which corresponds to an $S$-linear structure on $\phi^*$. 
A similar argument works for $\phi^!$. 
\end{proof}

The proofs of the following lemmas are formal and left to the reader. 

\begin{lemma}
\label{lemma-bc-adjoints}
Let $\phi_1 \colon \cC_1 \to \cD_1$ and $\phi_2 \colon \cC_2 \to \cD_2$ be $S$-linear functors. 
\begin{enumerate}
\item If $\phi_1$ and $\phi_2$ both admit left adjoints $\phi_1^*$ and $\phi_2^*$ \textup(or right adjoints $\phi_1^!$ and $\phi_2^!$\textup), 
then the functor $\phi_1 \otimes \phi_2 \colon \cC_1 \otimes_{\Perf(S)} \cC_2 \to \cD_1 \otimes_{\Perf(S)} \cD_2$ has a 
left adjoint given by $\phi_1^* \otimes \phi_2^*$ \textup(or right adjoint given by $\phi_1^! \otimes \phi_2^!$\textup). 
\item If $\phi_1$ and $\phi_2$ both admit left or right adjoints and are fully faithful, then so is 
$\phi_1 \otimes \phi_2$. 
\end{enumerate}
\end{lemma}

\begin{lemma}
\label{lemma-adjoints-triangle}
Let $\phi_i \colon \cC \to \cD$ be an $S$-linear functor between $S$-linear 
stable $\infty$-categories for $i =1,2,3$, and let $\phi_1 \to \phi_2 \to \phi_3$ 
be an exact triangle in $\Fun_{\Perf(S)}(\cC, \cD)$. 
\begin{enumerate}
\item If each $\phi_i$ admits a left adjoint $\phi_i^*$, there is 
an induced exact triangle 
$\phi_3^* \to \phi_2^* \to \phi_1^*$. 

\item If each $\phi_i$ admits a right adjoint $\phi_i^!$, there 
is an induced exact triangle 
$\phi_3^! \to \phi_2^! \to \phi_1^!$. 
\end{enumerate}
\end{lemma}


\section{Semiorthogonal decompositions} 
\label{section-sod}
In this section, we consider semiorthogonal decompositions of $S$-linear categories. 
In \S\ref{subsection-sod-definitions} we give the basic definitions, 
and in \S\ref{subsection-admissible-category} we discuss the notions of admissible 
subcategories and mutation functors; the results here are standard in the 
triangulated setting (see \cite{bondal, bondal-kapranov}), so we freely omit 
proofs when the usual arguments work without modification. 
In \S\ref{subsection-sod-small-large} we describe the relation between 
semiorthogonal decompositions of linear and presentable linear 
categories. 
Then we discuss induced semiorthogonal decompositions of tensor products 
of linear categories in \S\ref{subsection-sod-tensor}, and of functor categories 
between linear categories in \S\ref{subsection-sod-functor}. 
Finally, in \S\ref{subsection-splitting-functors} we review the notion of a splitting 
functor from \cite{kuznetsov-HPD}. 

\subsection{Basic definitions} 
\label{subsection-sod-definitions}
Given an $S$-linear category $\cC$, an \emph{$S$-linear stable subcategory} $\cA \subset \cC$ 
is a stable subcategory of $\cC$ which is preserved by the $\Perf(S)$-action on $\cC$; 
this is equivalent to the data of an $S$-linear category $\cA$ equipped with a fully 
faithful $S$-linear functor $\cA \to \cC$. 
Similarly, given a presentable $S$-linear category $\cC$, a \emph{presentable $S$-linear stable subcategory} 
$\cA \subset \cC$ is a stable subcategory of $\cC$ which is closed under colimits and 
preserved by the $\QCoh(S)$-action on~$\cC$; 
this is equivalent to the data of a presentable $S$-linear category $\cA$ equipped 
with a fully faithful cocontinuous $S$-linear functor $\cA \to \cC$.

\begin{definition}
\label{definition-sod}
Let $\cC$ be an $S$-linear (resp. presentable $S$-linear) category. 
An \emph{$S$-linear (resp. presentable $S$-linear) semiorthogonal decomposition} of $\cC$
is a sequence $\cA_1, \cA_2, \dots, \cA_n$ of $S$-linear (resp. presentable $S$-linear) 
stable subcategories of $\cC$ 
--- called the \emph{components} of the decomposition --- such that: 
\begin{enumerate}
\item \label{sod-1}
$\cHomS(C,D) \simeq 0$ for all $C \in \cA_i, D \in \cA_j$, and $i > j$. 
\item \label{sod-2}
For any $C \in \cC$, 
there exists a diagram 
\begin{equation}
\label{Ci-filtration} 
\xymatrix{
0 \ar@{=}[r] & C_n \ar[rr] && C_{n-1} \ar[dl] \ar[r] & \cdots \ar[r]  & 
C_1 \ar[rr] && C_0 \ar[dl] \ar@{=}[r] & C \\
&  & A_n \ar@{-->}[ul] &   & & & A_1 \ar@{-->}[ul] & & 
}
\end{equation}
where $A_i \in \cA_i$ and the triangles are exact. 
\end{enumerate}
If only condition~\eqref{sod-1} is satisfied, we say the sequence $\cA_1, \dots, \cA_n$ 
is \emph{semiorthogonal}. 
\end{definition}

\begin{remark}
By $S$-linearity of the categories $\cA_i$, in~\eqref{sod-1} it is equivalent to 
require that the space $\Map_{\cC}(C,D)$ is contractible. 
\end{remark} 

Given a stable $\infty$-category $\cC$ and a collection of subcategories $\cA_i \subset \cC, i = 1, \dots, n$, 
we denote by 
\begin{equation*}
\llangle \cA_1, \dots, \cA_n \rrangle \subset \cC 
\end{equation*}
the stable subcategory of $\cC$ generated by the $\cA_i$. 
In particular, for a semiorthogonal decomposition as in Definition~\ref{definition-sod}, we have 
\begin{equation*}
\cC = \llangle \cA_1, \dots, \cA_n \rrangle. 
\end{equation*}

\begin{lemma}
\label{lemma-sod}
Let $\cC$ be an $S$-linear (resp. presentable $S$-linear) category, with an $S$-linear 
(resp. presentable $S$-linear) semiorthogonal decomposition
\begin{equation*}
\cC = \llangle \cA_1, \dots, \cA_n \rrangle . 
\end{equation*}
Let $\alpha_i \colon \cA_i \to \cC$ denote the inclusion. 
Then there are $S$-linear (resp. cocontinuous $S$-linear) functors 
\begin{alignat*}{2}
\tr_i & \colon \cC \to \cC , \quad & 0 \leq i \leq n, \\
\pr_i & \colon \cC \to \cA_i , \quad & 1 \leq i \leq n, 
\end{alignat*} 
such that:
\begin{enumerate}
\item \label{sod-projection-functors}
There is a diagram in the category of functors $\Fun_{\Perf(S)}(\cC, \cC)$ 
(resp. $\Fun_{\QCoh(S)}(\cC, \cC)$) 
\begin{equation*}
\xymatrix{
0 \ar@{=}[r] & \tr_n \ar[rr] && \tr_{n-1} \ar[dl] \ar[r] & \cdots \ar[r]  & 
\tr_1 \ar[rr] && \tr_0 \ar[dl] \ar@{=}[r] & \id_{\cC} \\
&  & \alpha_n \circ \pr_n \ar@{-->}[ul] &   & & & \alpha_1 \circ \pr_1 \ar@{-->}[ul] & & 
}
\end{equation*}
where the triangles are exact, which recovers~\eqref{Ci-filtration} 
when applied to any $C \in \cC$. 
\item The functor $\pr_i \colon \cC \to \cA_i$ is a retraction, i.e. 
$\pr_i \circ \alpha_i \simeq \id_{\cA_i}$. 
\item The restriction of $\pr_i \colon \cC \to \cA_i$ to 
$\langle \cA_{i}, \cA_{i+1}, \dots, \cA_{n} \rangle \subset \cC$ is left adjoint 
to the inclusion $\cA_i \to \langle \cA_{i}, \dots, \cA_{n} \rangle$. 
In particular, $\pr_1 = \alpha_1^*$ is left adjoint to $\alpha_1 \colon \cA_1 \to \cC$. 
\item The restriction of $\pr_i \colon \cC \to \cA_i$ to 
$\langle \cA_{1}, \cA_2 , \dots, \cA_{i} \rangle \subset \cC$ is right adjoint 
to the inclusion $\cA_i \to \langle \cA_{1}, \cA_2 , \dots, \cA_{i} \rangle$. 
In particular, $\pr_n = \alpha_n^!$ is right adjoint to $\alpha_n \colon \cA_n \to \cC$. 
\end{enumerate}
\end{lemma}

\begin{proof}
This result is well-known in the triangulated setting; the same proof 
works in our setup. 
The only point which deserves explanation is that in the presentable 
case, the functors $\tr_i$ and $\pr_i$ are indeed cocontinuous. 
This claim reduces to the case where the length of the semiorthogonal 
decomposition is $n=2$. 
Then the diagram from~\eqref{sod-projection-functors} 
amounts to a distinguished triangle 
\begin{equation*}
\alpha_2 \circ \pr_2 \to \id_{\cC} \to \alpha_1 \circ \pr_1, 
\end{equation*}
where $\pr_1 = \alpha_1^*$ and $\pr_2 = \alpha_2^!$. 
Being a left adjoint, the functor $\pr_1$ is cocontinuous. 
Hence the above triangle implies $\alpha_2 \circ \pr_2$ is 
cocontinuous. 
Since $\alpha_2$ is fully faithful and cocontinuous, we 
conclude $\pr_2$ is cocontinuous. 
\end{proof}

\begin{definition}
Given a semiorthogonal decomposition as in Lemma~\ref{lemma-sod}, 
the functors $\tr_i$ and $\pr_i$ are called the \emph{truncation} and 
\emph{projection} functors. 
\end{definition}

\subsection{Admissible subcategories and mutation functors} 
\label{subsection-admissible-category}

\begin{definition}
Let $\cC$ be an $S$-linear category, 
and let $\cA \subset \cC$ be an $S$-linear stable subcategory. 
Let $\alpha \colon \cA \to \cC$ denote the inclusion. Then $\cA$ is called:  
\begin{itemize}
\item[--] \emph{left admissible} in $\cC$ if $\alpha$ admits a left adjoint $\alpha^* \colon \cC \to \cA$; 
\item[--] \emph{right admissible} in $\cC$ if $\alpha$ admits a right adjoint $\alpha^! \colon \cC \to \cA$;  
\item[--] \emph{admissible} in $\cC$ if it is both left and right admissible. 
\end{itemize}
\end{definition}

Admissibility of a subcategory is related to semiorthogonal decompositions as follows. 
Given a subcategory $\cA$ of an $\infty$-category $\cC$, 
consider the full subcategories of $\cC$ defined by  
\begin{align*}
\cA^{\perp} & = \{ \, C \in \cC  \mid \Map_{\cC}(D,C) \text{ is contractible for all } D \in \cA \, \} ,  \\ 
^{\perp}\cA & = \{ \, C \in \cC  \mid \Map_{\cC}(C, D) \text{ is contractible for all } D \in \cA \, \} . 
\end{align*}
We call $\cA^\perp$ the \emph{right orthogonal} to $\cA \subset \cC$, and 
$^{\perp}\cA$ the \emph{left orthogonal} to $\cA \subset \cC$. 
If $\cC$ is an $S$-linear category and $\cA \subset \cC$ is an $S$-linear stable subcategory, 
then the orthogonal categories $\cA^{\perp}$ and $^{\perp}\cA$ are also $S$-linear stable 
subcategories of $\cC$. 
Clearly, $\cA^{\perp}, \cA$ and $ \cA, {^{\perp}}\cA$ are semiorthogonal pairs in $\cC$. 

\begin{lemma}
\label{lemma-admissible-sod}
Let $\cC$ be an $S$-linear category, 
and let $\cA, \cB$ be a pair of $S$-linear 
stable subcategories. The following are equivalent: 
\begin{enumerate}
\item $\cC = \llangle \cA, \cB \rrangle$ is a semiorthogonal decomposition. 
\item $\cA \subset \cC$ is left admissible and $\cB =  {^\perp} \cA$. 
\item $\cB \subset \cC$ is right admissible and $\cA = \cB^{\perp}$. 
\end{enumerate}
\end{lemma}

\begin{definition}
Let $\cA \subset \cC$ be an $S$-linear stable subcategory of an $S$-linear category. 
If $\cA \subset \cC$ is left admissible, then by Lemma~\ref{lemma-admissible-sod} the 
inclusion $i \colon {^\perp}\cA \to \cC$ admits a right adjoint $i^!$; in this case the 
functor 
\begin{equation*}
\rR_{\cA} = i \circ i^! \colon \cC \to \cC 
\end{equation*} 
is called the \emph{right mutation functor} through $\cA$. 
Similarly, if $\cA \subset \cC$ is right admissible, then the inclusion $j \colon \cA^{\perp} \to \cC$ 
admits a left adjoint $j^*$, and the functor 
\begin{equation*}
\rL_{\cA} = j \circ j^* \colon \cC \to \cC 
\end{equation*}
is called the \emph{left mutation functor} through $\cA$.  
\end{definition}

\begin{lemma}
\label{lemma-mutation-functors}
Let $\cA \subset \cC$ be an $S$-linear stable subcategory of an $S$-linear category. 
\begin{enumerate}
\item If $\cA \subset \cC$ is left admissible, then the right mutation functor $\rR_{\cA}$ kills $\cA$ and is 
fully faithful on the subcategory $\cA^{\perp} \subset \cC$. 
\item If $\cA \subset \cC$ is right admissible, then the left mutation functor $\rL_{\cA}$ kills $\cA$ and 
is fully faithful on the subcategory ${^\perp}\cA \subset \cC$. 
\item If $\cA \subset \cC$ is admissible, then $\rR_{\cA}$ and $\rL_{\cA}$ induce mutually inverse equivalences 
\begin{equation*}
\rR_{\cA}|_{\cA^{\perp}}   \colon \cA^{\perp} \xrightarrow{\sim} {^\perp}\cA 
\quad \text{and} \quad 
\rL_{\cA}|_{{^\perp}\cA}   \colon {^\perp}\cA \xrightarrow{\sim} \cA^{\perp}. 
\end{equation*}
\end{enumerate} 
\end{lemma}

\begin{lemma}
\label{lemma-adjoint-adjoints}
Let $\cC$ be an $S$-linear category, and let 
$\alpha \colon \cA \to \cC$ be the inclusion of an 
$S$-linear stable subcategory. 
\begin{enumerate} 
\item \label{left-left-adjoint}
If $\cA \subset \cC$ is left admissible and ${^{\perp}}\cA \subset \cC$ is admissible, 
then $\alpha^* \colon \cC \to \cA$ has a left adjoint given by 
\begin{equation*}
\rR_{{^{\perp}}\cA} \circ \alpha \colon \cA \to \cC . 
\end{equation*}

\item \label{right-right-adjoint} 
If $\cA \subset \cC$ is right admissible and $\cA^{\perp} \subset \cC$ is admissible, 
then $\alpha^! \colon \cC \to \cA$ has a right adjoint given by 
\begin{equation*}
\rL_{\cA^{\perp}} \circ \alpha \colon \cA \to \cC. 
\end{equation*}
\end{enumerate}
\end{lemma}

\begin{proof}
We prove \eqref{left-left-adjoint}; the proof of \eqref{right-right-adjoint} is similar. 
Let $\cB = {^{\perp}}\cA$, so that $\cC = \llangle \cA, \cB \rrangle$, and let 
$\beta \colon \cB \to \cC$ be the inclusion. 
Note that for $C \in \cA$ and $D \in \cC$ we have exact triangles 
\begin{align*}
& \beta \beta^!(D) \to  D \to \alpha \alpha^*(D) , \\ 
& \rR_{\cB} \alpha(C) \to  \alpha(C) \to \beta \beta^*\alpha(C) , 
\end{align*}
where $\cHom_S(\beta \beta^*\alpha(C), \alpha \alpha^*(D)) \simeq 0$
and $\cHom_S(\rR_{\cB} \alpha(C),  \beta \beta^!(D)) \simeq 0$. 
Using this, we obtain equivalences 
\begin{align*}
\cHom_S(C, \alpha^*(D)) & \simeq 
\cHom_S(\alpha(C), \alpha\alpha^*(D))  \\ 
& \simeq \cHom_S( \rR_{{^{\perp}}\cA} \alpha(C), \alpha \alpha^*(D)) \\ 
& \simeq \cHom_S( \rR_{{^{\perp}}\cA} \alpha(C), D) , 
\end{align*} 
which proves the adjunction. 
\end{proof}

\begin{lemma}
\label{lemma-admissible-sequence}
Let $\cC$ be an $S$-linear category, and let 
$\cA_1, \dots, \cA_n$ be a semiorthogonal sequence of $S$-linear stable subcategories.  
\begin{enumerate}
\item If the $\cA_i \subset \cC$ are all right admissible, then 
\begin{equation*}
\llangle \cA_1, \dots, \cA_n \rrangle \subset \cC 
\end{equation*}
is also a right admissible $S$-linear stable subcategory, with left mutation functor given by 
\begin{equation*}
\rL_{\llangle \cA_1, \dots, \cA_n \rrangle}  \simeq \rL_{\cA_1} \circ \rL_{\cA_2} \circ \cdots \circ \rL_{\cA_n}. 
\end{equation*}

\item If the $\cA_i \subset \cC$ are all left admissible, then 
\begin{equation*}
\llangle \cA_1, \dots, \cA_n \rrangle \subset \cC 
\end{equation*}
is also a left admissible $S$-linear stable subcategory, with right mutation functor given by 
\begin{equation*}
\rR_{\llangle \cA_1, \dots, \cA_n \rrangle}  \simeq \rR_{\cA_n} \circ \rR_{\cA_{n-1}} \circ \cdots \circ \rR_{\cA_1}. 
\end{equation*}
\end{enumerate}
\end{lemma}

\begin{lemma}
\label{lemma-mutation-sod}
Let $\cC$ be an $S$-linear category with an $S$-linear semiorthogonal decomposition 
\begin{equation*}
\cC = \llangle \cA_1, \dots, \cA_n \rrangle . 
\end{equation*}
\begin{enumerate}
\item For $1 \leq i \leq n-1$, if $\cA_i \subset \cC$ is right admissible 
there is a semiorthogonal decomposition
\begin{equation*}
\cC = \llangle \cA_1, \dots, \cA_{i-1}, \rL_{\cA_i}(\cA_{i+1}), \cA_i, \cA_{i+2}, \dots, \cA_n \rrangle. 
\end{equation*}
\item For $2 \leq i \leq n$, if $\cA_i \subset \cC$ is left admissible 
there is a semiorthogonal decomposition
\begin{equation*}
\cC =  \llangle \cA_1, \dots, \cA_{i-2}, \cA_i, \rR_{\cA_i}(\cA_{i-1}), \cA_{i+1}, \dots, \cA_n \rrangle. 
\end{equation*}
\end{enumerate}
\end{lemma}

\subsection{Linear versus presentable linear semiorthogonal decompositions} 
\label{subsection-sod-small-large}

\begin{lemma}
\label{lemma-sod-small-to-large}
Let $\cC$ be an $S$-linear category with an $S$-linear semiorthogonal decomposition 
\begin{equation*}
\cC = \llangle \cA_1, \dots, \cA_n \rrangle . 
\end{equation*}
Then there is an induced presentable $S$-linear semiorthogonal decomposition 
\begin{equation*}
\Ind(\cC) = \llangle \Ind(\cA_1), \dots, \Ind(\cA_n) \rrangle,  
\end{equation*}
where the embedding functors $\Ind(\cA_i) \to \Ind(\cC)$ preserve compact objects. 
\end{lemma}

\begin{proof}
It is easy to see that $\Ind(\cA_1), \dots, \Ind(\cA_n)$ are 
semiorthogonal in $\Ind(\cC)$. 
Further, the projection functors for the given semiorthogonal decomposition of $\cC$ induce 
projection functors for $\Ind(\cC)$, 
so that Definition~\ref{definition-sod}\eqref{sod-2} holds. 
Finally, since $\Ind(\cA_i)^c = \cA_i$ and $\Ind(\cC)^c = \cC$, the embeddings 
$\Ind(\cA_i) \to \Ind(\cC)$ preserve compact objects. 
\end{proof}

In general, given a presentable $S$-linear semiorthogonal decomposition, passing to 
categories of compact objects does not necessarily induce an $S$-linear semiorthogonal decomposition. 
However: 
\begin{lemma}
\label{lemma-sod-large-to-small}
Let $\cC$ be a presentable $S$-linear category with a presentable $S$-linear 
semiorthogonal decomposition  
\begin{equation*}
\cC = \llangle \cA_1, \dots, \cA_n \rrangle . 
\end{equation*}
\begin{enumerate}
\item \label{pr1-compact-objects}
The projection functor $\pr_1 \colon \cC \to \cA_1$ preserves 
compact objects. 
\item The embedding functor $\alpha_n \colon \cA_n \to \cC$ 
preserves compact objects. 
\end{enumerate}
If for all $i$ the embedding functors $\alpha_i \colon \cA_i \to \cC$ preserve compact objects, then furthermore: 
\begin{enumerate}
\item For all $i$ the projection functors $\pr_i \colon \cC \to \cA_i$ preserve compact objects. 
\item There is an induced semiorthogonal decomposition 
\begin{equation*}
\cC^c = \llangle \cA_1^c , \dots, \cA_n^c \rrangle. 
\end{equation*}
\end{enumerate}
\end{lemma}

\begin{proof}
By induction we reduce to the case where $n=2$. 
The functor $\pr_1 = \alpha_1^*$ preserves compact objects because 
its right adjoint $\alpha_1$ is cocontinuous. 
Similarly, $\alpha_2$ preserves compact objects because its right 
adjoint $\pr_2 = \alpha_2^!$ is cocontinuous by Lemma~\ref{lemma-sod}. 

Now assume that $\alpha_1$ preserves compact objects. 
Then there is a triangle of functors
\begin{equation*}
\alpha_2 \circ \pr_2  \to \id_{\cC} \to \alpha_1 \circ \pr_1 , 
\end{equation*}
where the second and third vertices, and hence also the first, preserve compact objects. 
Since $\alpha_2$ is fully faithful and cocontinuous, it follows that 
$\pr_2$ preserves compact objects. 
Since for all $i$ the embeddings $\alpha_i$ and the projection functors $\pr_i$ preserve 
compact objects, it follows formally that there is an induced semiorthogonal decomposition 
$\cC^c = \llangle \cA_1^c , \cA_2^c \rrangle$. 
\end{proof}

\begin{remark}
Lemmas~\ref{lemma-sod-small-to-large} and~\ref{lemma-sod-large-to-small} can be 
summarized as follows: 
Under the equivalence $\Ind \colon \Cat_{S} \xrightarrow{\sim} \PrCat_S^{\omega}$, 
$S$-linear semiorthogonal decompostions correspond to presentable $S$-linear semiorthogonal 
decompositions such that the embedding functors of the components preserve compact objects. 
\end{remark}

\subsection{Semiorthogonal decompositions of tensor products} 
\label{subsection-sod-tensor} 

\begin{lemma}
\label{lemma-sod-base-change}
Let $\cC = \llangle \cA_1, \dots, \cA_m \rrangle$ and $\cD = \llangle \cB_1, \dots, \cB_{n} \rrangle$  
be $S$-linear semiorthogonal decompositions. 
Then the functor $\cA_i \otimes_{\Perf(S)} \cB_j \to \cC \otimes_{\Perf(S)} \cD$ is fully faithful 
for all $i,j$. Moreover, there is an $S$-linear semiorthogonal decomposition 
\begin{equation*}
\cC \otimes_{\Perf(S)} \cD = \llangle \cA_i \otimes_{\Perf(S)} \cB_j \rrangle_{1 \leq i \leq m,\, 1 \leq j \leq n} 
\end{equation*} 
where the ordering on the set $\set{ (i,j) \st  1 \leq i \leq m , 1 \leq j \leq n }$ is any 
one which extends the coordinate-wise partial order. 
The projection functor onto the $(i,j)$-component of this decomposition is given by 
\begin{equation*}
\pr_{\cA_i} \otimes \pr_{\cB_j}  \colon \cC \otimes_{\Perf(S)} \cD \to  \cA_i \otimes_{\Perf(S)} \cB_j, 
\end{equation*} 
where $\pr_{\cA_i} \colon \cC \to \cA_i$ and $\pr_{\cB_j} \colon \cD \to \cB_j$ are the projection 
functors for the given decompositions.
\end{lemma}

\begin{proof} 
The result reduces to the case where $m = 2$ and $n = 1$. 
It follows from Lemmas~\ref{lemma-admissible-sod} and \ref{lemma-bc-adjoints} 
that the functor 
\begin{equation*}
\cA_i \otimes_{\Perf(S)} \cD \to \cC \otimes_{\Perf(S)} \cD
\end{equation*}
is fully faithful for $i=1,2$.  
It follows from Lemmas~\ref{lemma-box-tensor-generation} and \ref{lemma-kunneth} 
that the categories 
\begin{equation*}
\cA_1 \otimes_{\Perf(S)} \cD, ~ \cA_2 \otimes_{\Perf(S)} \cD 
\end{equation*} 
are semiorthogonal in $\cC \otimes_{\Perf(S)} \cD$. 
The projection functors for the original semiorthogonal decomposition induce 
projection functors for $\cC \otimes_{\Perf(S)} \cD$, 
so that Definition~\ref{definition-sod}\eqref{sod-2} holds. 
\end{proof}

\begin{remark} 
\label{remark-sod-bc}
If $S' \to S$ is a morphism of schemes, then taking $\cD = \Perf(S')$ in 
Lemma~\ref{lemma-sod-base-change} gives a base change result for 
semiorthogonal decompositions. 
\end{remark}

\begin{lemma}
\label{lemma-base-change-admissible}
Let $\cC$ be an $S$-linear category, and let $\cA \subset \cC$ be 
an $S$-linear stable subcategory. 
Let $\cD$ be another $S$-linear category. 
Then if $\cA \subset \cC$ is left (or right) admissible, so is 
\begin{equation*}
\cA \otimes_{\Perf(S)} \cD \subset \cC \otimes_{\Perf(S)} \cD. 
\end{equation*} 
\end{lemma}

\begin{proof}
Follows from Lemma~\ref{lemma-bc-adjoints}. 
\end{proof}

\begin{lemma}
Let $\cC$ be an $S$-linear category, and let $\cA \subset \cC$ be 
a left admissible $S$-linear stable subcategory. 
Let $f \colon S' \to S$ be a morphism of schemes such that $f_*$ preserves perfect complexes. 
Let 
\begin{equation*}
\cA' = \cA \otimes_{\Perf(S)} \Perf(S') \subset \cC' = \cC \otimes_{\Perf(S)} \Perf(S')  . 
\end{equation*}
Then in terms of the functor 
\begin{equation*}
f_* \colon \cC' = \cC \otimes_{\Perf(S)} \Perf(S')  \to \cC \otimes_{\Perf(S)} \Perf(S) \simeq \cC, 
\end{equation*}
the subcategory $\cA' \subset \cC'$ is given by 
\begin{equation}
\label{Aprime-description}
\cA' = \set{ C' \in \cC'  \st f_*(C' \otimes F) \in \cA \text{ for all } F \in \Perf(S') } . 
\end{equation}
\end{lemma}

\begin{proof} 
Clearly $\cA'$ is contained in the right side of~\eqref{Aprime-description}.  
For the reverse inclusion, 
we consider the left orthogonal $\cB \subset \cC$ to $\cA$ and its 
base change 
\begin{equation*}
\cB' = \cB \otimes_{\Perf(S)} \Perf(S'). 
\end{equation*} 
We have $\cC = \llangle \cA , \cB \rrangle$ by Lemma~\ref{lemma-admissible-sod}, so 
$\cC' = \llangle \cA', \cB' \rrangle$ by Lemma~\ref{lemma-sod-base-change}. 
Hence given $C' \in \cC'$ such that $f_*(C' \otimes F) \in \cA$ for all $F \in \Perf(S')$, 
we must show that $C'$ is right orthogonal to~$\cB'$. 
By Lemma~\ref{lemma-box-tensor-generation} it suffices to show that 
if $B \in \cB$ and $G \in \Perf(S')$, then 
\begin{equation*}
\cHom_S(B \boxtimes G, C') \simeq 0. 
\end{equation*}
Note that we can write $B \boxtimes G = f^*(B) \otimes G$, and hence 
\begin{equation*}
\cHom_S(B \boxtimes G, C') \simeq \cHom_S(f^*(B), C' \otimes G^{\svee}) 
\simeq \cHom_S(B, f_*(C' \otimes G^{\svee})). 
\end{equation*}
This vanishes since $f_*(C' \otimes G^{\svee}) \in \cA$ by assumption. 
\end{proof}

\subsection{Semiorthogonal decompositions of functor categories}  
\label{subsection-sod-functor} 
Recall that given $S$-linear (resp. presentable $S$-linear) 
categories $\cC$ and $\cD$, the $S$-linear (resp. cocontinuous $S$-linear) functors 
form the objects of an $S$-linear (resp. presentable $S$-linear) 
category $\Fun_{\Perf(S)}(\cC, \cD)$ (resp. $\Fun_{\QCoh(S)}(\cC,\cD)$). 
In the following lemma, we use the uniform notation $\Fun_S(\cC, \cD)$ to denote 
$\Fun_{\Perf(S)}(\cC, \cD)$ if $\cC, \cD \in \Cat_S$ or $\Fun_{\QCoh(S)}(\cC,\cD)$ if $\cC, \cD \in \PrCat_S$. 

\begin{lemma}
\label{lemma-functor-category-sod}
Let $\cC$ and $\cD$ be $S$-linear (resp. presentable $S$-linear) categories. 
\begin{enumerate}
\item \label{functor-category-sod-1}
Let $\cC = \langle \cA_1, \dots, \cA_n \rangle$ be an 
$S$-linear (resp. presentable $S$-linear) semiorthogonal decomposition. 
Then for every $i$ the functor 
\begin{equation}
\label{functor-sod-embedding}
\Func{S}{\cA_i}{\cD} \to \Func{S}{\cC}{\cD} 
\end{equation}
induced by the projection $\pr_i \colon \cC \to \cA_i$ is fully faithful.  
Further, there is an $S$-linear (resp. presentable $S$-linear) semiorthogonal decomposition 
\begin{equation*}
\Func{S}{\cC}{\cD} = 
\llangle \Func{S}{\cA_1}{\cD}, \dots, 
\Func{S}{\cA_n}{\cD} 
\rrangle , 
\end{equation*}
whose projection functors 
\begin{equation}
\label{functor-sod-projection}
\Func{S}{\cC}{\cD} \to \Func{S}{\cA_i}{\cD} 
\end{equation}
are induced by the embeddings $\alpha_i \colon \cA_i \to \cD$. 

\item \label{functor-category-sod-2}
Let $\cD = \langle \cB_1, \dots, \cB_n \rangle$ be an 
$S$-linear (resp. presentable $S$-linear) semiorthogonal decomposition. 
Then for every $i$ the functor 
\begin{equation*}
\Func{S}{\cC}{\cB_i} \to \Func{S}{\cC}{\cD} 
\end{equation*}
induced by the emedding $\beta_i \colon \cB_i \to \cD$ is fully faithful.  
Further, there is an $S$-linear (resp. presentable $S$-linear) semiorthogonal decomposition 
\begin{equation*}
\Func{S}{\cC}{\cD} = 
\llangle 
\Func{S}{\cC}{\cB_1}, \dots, 
\Func{S}{\cC}{\cB_n} 
\rrangle , 
\end{equation*}
whose projection functors 
\begin{equation*}
\Func{S}{\cC}{\cD} \to \Func{S}{\cC}{\cB_i} 
\end{equation*}
are induced by the projections $\pr_i \colon \cD \to \cB_i$. 
\end{enumerate}
\end{lemma}

\begin{proof}
We prove~\eqref{functor-category-sod-1}. 
First assume we are in the $S$-linear (not presentable) situation. 
The result reduces to the case $n=2$. 
One checks that the functor~\eqref{functor-sod-projection} is left adjoint 
to~\eqref{functor-sod-embedding} for $i=1$,
right adjoint to~\eqref{functor-sod-embedding} for $i=2$, and in both  
cases the composition 
\begin{equation*}
\Func{S}{\cA_i}{\cD} \to \Func{S}{\cC}{\cD} \to \Func{S}{\cA_i}{\cD} 
\end{equation*}
is equivalent to the identity. This proves the claim that $\Func{S}{\cA_i}{\cD} \to \Func{S}{\cC}{\cD}$ 
is fully faithful, and the adjointness of~\eqref{functor-sod-projection} 
and~\eqref{functor-sod-embedding} immediately implies that 
\begin{equation*}
\Func{S}{\cA_1}{\cD}, ~ \Func{S}{\cA_2}{\cD} 
\end{equation*}
are semiorthogonal in $\Func{S}{\cC}{\cD}$. 
Finally, given any $F \in \Func{S}{\cC}{\cD}$, the triangle 
\begin{equation*}
\alpha_2 \circ \pr_2 \to \id_{\cC} \to \alpha_1 \circ \pr_1 
\end{equation*}
induces a triangle 
\begin{equation*}
F \circ \alpha_2 \circ \pr_2 \to F \to F \circ \alpha_1 \circ \pr_1, 
\end{equation*}
which shows that Definition~\ref{definition-sod}\eqref{sod-2} holds. 
In the presentable $S$-linear case the argument above works verbatim, 
with the additional remark that the functors~\eqref{functor-sod-embedding} 
and~\eqref{functor-sod-projection} are cocontinuous, because colimits of 
functors are computed objectwise (see~\cite[Section 5.1.2]{lurie-HTT}).

Part~\eqref{functor-category-sod-2} is proved similarly. 
\end{proof}

\subsection{Splitting functors}
\label{subsection-splitting-functors} 
Here we review the notion of a splitting functor from~\cite{kuznetsov-HPD}  
in the context of $S$-linear categories. 

\begin{definition}
\label{definition-ker-im}
Let $\phi \colon \cC \to \cD$ be an $S$-linear functor between $S$-linear categories. 
The \emph{kernel} of $\phi$ is the full subcategory 
$\ker \phi \subset \cC$ spanned by objects $C \in \cC$ such that $\phi(C) \simeq 0$, 
and the \emph{image} of $\phi$ is the full subcategory $\im \phi \subset \cD$ spanned 
by the objects $\phi(C) \in \cD$ for $C \in \cC$. 
\end{definition}

\begin{remark}
The kernel of a functor $\phi$ as in Definition~\ref{definition-ker-im} is automatically an 
$S$-linear stable subcategory of $\cC$. 
This is true of the image of $\phi$ if for instance $\phi$ is fully faithful, but 
is not true in general. 
\end{remark}

The following result is \cite[Theorem 3.3]{kuznetsov-HPD} in our setting, which holds by the same proof. 

\begin{theorem}
\label{theorem-splitting-functors}
Let $\phi \colon \cC \to \cD$ be an $S$-linear functor between $S$-linear categories. 
The following are equivalent: 
\begin{enumerate}[leftmargin=.8cm, label=(\text{\arabic*r})]
\item
\label{splitting-functor-1-r}
$\ker \phi$ is a right admissible subcategory 
of $\cC$, the restriction of $\phi$ to $(\ker \phi){^\perp}$ is fully faithful, and 
$\im \phi$ is a right admissible subcategory of $\cD$.
\item
\label{splitting-functor-2-r}
$\phi$ admits a right adjoint $\phi^!$, and the composition of the canonical 
morphism $\id \to \phi^!\phi$ with $\phi$ gives an equivalence $\phi \simeq \phi \phi^! \phi$. 
\item
\label{splitting-functor-3-r}
$\phi$ admits a right adjoint $\phi^!$, there are semiorthogonal decompositions
\begin{align*}
\cC & = \llangle \im \phi^!, \ker \phi \rrangle,  \\
\cD & = \llangle \ker \phi^!, \im \phi \rrangle, 
\end{align*}
and the functors $\phi$ and $\phi^!$ induce mutually inverse equivalences 
$\im \phi^! \simeq \im \phi$. 
\end{enumerate}
Similarly, the following are equivalent: 
\begin{enumerate}[leftmargin=.8cm, label=(\text{\arabic*l})]
\item
\label{splitting-functor-1}
$\ker \phi$ is a left admissible subcategory 
of $\cC$, the restriction of $\phi$ to ${^\perp}(\ker \phi)$ is fully faithful, and 
$\im \phi$ is a left admissible subcategory of $\cD$.
\item
\label{splitting-functor-2}
$\phi$ admits a left adjoint $\phi^*$, and the composition of the canonical 
morphism $\phi^*\phi \to \id$ with $\phi$ gives an equivalence $\phi \phi^* \phi \simeq \phi$. 
\item
\label{splitting-functor-3}
$\phi$ admits a left adjoint $\phi^*$, there are semiorthogonal decompositions
\begin{align*}
\cC & = \llangle \ker \phi , \im \phi^* \rrangle,  \\
\cD & = \llangle \im \phi, \ker \phi^* \rrangle, 
\end{align*}
and the functors $\phi$ and $\phi^*$ induce mutually inverse equivalences 
$\im \phi^* \simeq \im \phi$. 
\end{enumerate}
\end{theorem}

\begin{definition}
Let $\phi \colon \cC \to \cD$ be an $S$-linear functor between $S$-linear categories. 
We say $\phi$ is \emph{right splitting} if the equivalent conditions \ref{splitting-functor-1-r}-\ref{splitting-functor-3-r} of 
Theorem~\ref{theorem-splitting-functors} hold, and 
\emph{left splitting} if the equivalent conditions \ref{splitting-functor-1}-\ref{splitting-functor-3} hold, and 
\emph{splitting} if it is right and left splitting. 
\end{definition}


\section{Smooth and proper categories} 
\label{section-smooth-proper}
In this section, we study the notions of smoothness and properness of $S$-linear categories,  
which are the analogues in noncommutative algebraic geometry of the usual geometric 
notions of smoothness and properness of a scheme over $S$.  
In \S\ref{subsection-smooth-proper-definitions} we define smoothness and properness 
of linear categories, and discuss the closely related notion of dualizability. 
In \S\ref{subsection-smooth-proper-geometry} we explain how smoothness or properness 
of an $S$-scheme relates to the corresponding property  
of its category of perfect complexes. 
In \S\ref{subsection-smooth-proper-base-change} we show that smoothness and 
properness of linear categories behave well under base change. 
In \S\ref{subsection-smooth-proper-adjoints} we show that a functor between 
linear categories admits adjoints if the source is smooth and proper and the target is proper. 
In \S\ref{subsection-smooth-proper-sod} we analyze the behavior of smoothness 
and properness under semiorthogonal decompositions. 
In \S\ref{subsection-serre-functor} we define relative Serre functors, give some of their 
properties, and prove they exist in the smooth and proper case. 
In \S\ref{subsection-critical-loci} we introduce the notion of the critical locus 
of a linear category. 
Finally, in \S\ref{subsection-Db} we discuss the bounded coherent category 
of an $S$-linear category and its behavior under semiorthogonal decompositions; 
under reasonable assumptions, this construction recovers the bounded derived category of coherent sheaves on an $S$-scheme 
from its category of perfect complexes. 

We note that a number of the results in this section are folklore, 
or appear in some form in the literature, 
cf. \cite{antieau-gepner, orlov-ncschemes, lunts-smoothness, kuznetsov-lunts, toen-moduli-objects}. 
However, we could not find an adequate reference for 
the point of view taken in this work. 

\subsection{Basic definitions} 
\label{subsection-smooth-proper-definitions}
\begin{definition}
Let $\cC$ be a symmetric monoidal $\infty$-category with unit $\bone$. 
An object $C \in \cC$ is called \emph{dualizable} if there exists 
an object $D \in \cC$ and morphisms 
\begin{align*}
\ev & \colon D \otimes C \to \bone, \\
\coev & \colon \bone \to C \otimes D, 
\end{align*}
such that the compositions 
\begin{align*}
 C \xrightarrow{\coev \otimes \id} & C \otimes D \otimes C \xrightarrow{\id \otimes \ev} C , \\
 D \xrightarrow{\id \otimes \coev} & D \otimes C \otimes D \xrightarrow{\ev \otimes \id} D , 
\end{align*}
are equivalent to the identity. 
The object $D$ is called the \emph{dual} of $C$, and 
$\ev$ and $\coev$ are called the \emph{evaluation} and 
\emph{coevaluation} morphisms. 
\end{definition}

\begin{remark}
\label{remark-duality-data-unique}
Equivalently, an object $C \in \cC$ is dualizable if it is 
dualizable as an object of the symmetric monoidal homotopy category $\Ho(\cC)$. 
Moreover, if $C \in \cC$ is dualizable, then the dual $D$ and the 
evaluation and coevaluation morphisms are uniquely determined in $\Ho(\cC)$. 
\end{remark}

In particular, given a category $\cC$ which is in $\Cat_S$ or $\PrCat_S$, 
it makes sense to ask whether $\cC$ is dualizable. If so, we denote by 
$\rD_S(\cC)$ the dual. 

\begin{lemma}
\label{lemma-presentable-category-dualizable}
Let $\cC$ be a compactly generated presentable $S$-linear category. 
Then $\cC$ is dualizable as an object of $\PrCat_S$, with dual 
given by 
\begin{equation*}
\rD_S(\cC) \simeq \Ind((\cC^{c})^{\op}), 
\end{equation*}
where $(\cC^{c})^{\op}$ denotes the opposite category of $\cC^{c}$. 
\end{lemma}

\begin{proof}
By~\cite[Chapter I.1, Proposition 7.3.2]{gaitsgory-DAG}, the underlying 
stable $\infty$-category of $\cC$ is dualizable with dual 
$\Ind((\cC^{c})^{\op})$. 
Since by assumption $S$ is quasi-compact and separated, 
an object of $\QCoh(S)$ is compact if and only if it is 
dualizable if and only if it is a perfect complex. 
Hence by~\cite[Chapter I.1, Lemma~9.1.5]{gaitsgory-DAG}, $\QCoh(S)$ 
is a rigid symmetric monoidal $\infty$-category. 
Now the result follows from~\cite[Chapter I.1, Proposition 9.4.4]{gaitsgory-DAG}. 
\end{proof}

The following is an easy consequence of the definitions.
\begin{lemma}
\label{lemma-dual-functor-category}
Let $\cC$ be a presentable $S$-linear category. 
If $\cC$ is dualizable, then for any $\cD_1, \cD_2 \in \PrCat_S$ there 
is an equivalence 
\begin{equation*}
\Fun_{\QCoh(S)}(\cD_1 \otimes_{\QCoh(S)} \cC, \cD_2) \simeq 
\Fun_{\QCoh(S)}(\cD_1, \rD_S(\cC) \otimes_{\QCoh(S)} \cD_2).
\end{equation*} 
\end{lemma}

Now let $\cC$ be an $S$-linear category. 
Then by Lemma~\ref{lemma-presentable-category-dualizable}, 
the presentable $S$-linear category $\Ind(\cC)$ is dualizable 
with dual $\Ind(\cC^{\op})$. More explicitly, the evaluation 
morphism 
\begin{equation}
\label{ev-Ind-C}
\ev \colon \Ind(\cC^{\op}) \otimes_{\QCoh(S)} \Ind(\cC) 
\to \QCoh(S) 
\end{equation}
is induced by the functor 
\begin{equation*}
\cHom_S(-,-) \colon \cC^{\op} \times \cC \to \QCoh(S). 
\end{equation*}
Under the equivalence 
\begin{equation*}
\Ind(\cC) \otimes_{\QCoh(S)} \Ind(\cC^{\op})  \simeq \Fun_{\QCoh(S)}(\Ind(\cC), \Ind(\cC)) 
\end{equation*}
deduced from Lemma~\ref{lemma-dual-functor-category}, 
the coevaluation morphism
\begin{equation}
\label{coev-Ind-C} 
\coev \colon \QCoh(S) \to \Ind(\cC) \otimes_{\QCoh(S)} \Ind(\cC^{\op}) 
\end{equation}
is identified with the canonical functor 
\begin{equation*} 
\QCoh(S) \to \Fun_{\QCoh(S)}(\Ind(\cC), \Ind(\cC)) 
\end{equation*}
which sends $\cO_S \in \QCoh(S)$ to $\id \in \Fun_{\QCoh(S)}(\Ind(\cC), \Ind(\cC))$. 

\begin{definition}
Let $\cC$ be an $S$-linear category. 
We say: 
\begin{enumerate}
\item $\cC$ is \emph{proper} if the evaluation morphism~\eqref{ev-Ind-C} 
is a morphism in the category $\PrCat_S^{\omega}$, i.e. if this functor 
preserves compact objects. 
\item $\cC$ is \emph{smooth} if the coevaluation morphism~\eqref{coev-Ind-C} 
is a morphism in the category $\PrCat_S^{\omega}$, i.e. if this functor 
preserves compact objects. 
\end{enumerate}
\end{definition}

\begin{remark}
The condition that $\cC$ is smooth or proper depends on its 
$S$-linear structure. For emphasis, we shall sometimes say 
$\cC$ is smooth \emph{over $S$} or proper \emph{over $S$}. 
For instance, if $T \to S$ is a morphism of schemes and $\cC$ is a 
$T$-linear category, then we say $\cC$ is smooth and proper 
over $S$ to mean that $\cC$ is smooth and proper with its induced 
$S$-linear structure. 
\end{remark}

\begin{lemma}
\label{lemma-smooth-proper-criteria}
Let $\cC$ be an $S$-linear category. 
\begin{enumerate}
\item \label{properness-criterion}
$\cC$ is proper if and only if for every $C, D \in \cC$ the 
mapping object $\cHom_S(C, D)$ lies in $\Perf(S) \subset \QCoh(S)$. 

\item \label{smoothness-criterion}
$\cC$ is smooth if and only if $\id_{\Ind(\cC)} \in \Fun_{\QCoh(S)}(\Ind(\cC), \Ind(\cC))$ 
is a compact object. 
\end{enumerate}
\end{lemma}

\begin{proof}
Follows from the descriptions of the functors~\eqref{ev-Ind-C} and~\eqref{coev-Ind-C} above. 
\end{proof}

\begin{lemma}
\label{lemma-smooth-proper-dualizable}
Let $\cC$ be an $S$-linear category. 
Then $\cC$ is smooth and proper over $S$ if and only if $\cC$ is 
dualizable as an object of $\Cat_S$. In this case, the dual of $\cC$ 
is $\cC^{\op}$. 
\end{lemma}

\begin{proof}
Since $\Ind$ induces a symmetric monoidal 
equivalence $\Cat_S \simeq \PrCat_S^{\omega}$, 
$\cC$ is dualizable as an object of $\Cat_S$ if and only if $\Ind(\cC)$ is 
dualizable as an object of $\PrCat_S^{\omega}$. 
By definition, if $\cC$ is smooth and proper over $S$, then 
$\Ind(\cC)$ is dualizable as an object of $\PrCat_S^{\omega}$. 
Conversely, if $\Ind(\cC)$ is dualizable as an object of $\PrCat_S^{\omega}$, 
by Remark~\ref{remark-duality-data-unique} 
its duality data must be given by~\eqref{ev-Ind-C} and~\eqref{coev-Ind-C}, 
which are hence morphisms in $\PrCat_S^{\omega}$. 
\end{proof}

\subsection{Relation to geometry}
\label{subsection-smooth-proper-geometry}
If $X$ is an $S$-scheme, smoothness and properness of $\Perf(X)$ 
are related to the corresponding properties of $X$ as follows. 
Recall that a morphism $X \to S$ is called \emph{perfect} if it is 
pseudo-coherent and of finite $\Tor$-dimension, 
see \cite[\href{http://stacks.math.columbia.edu/tag/0685}{Tag 0685}]{stacks-project}. 

\begin{lemma}
\label{lemma-smooth-proper-morphisms}
Let $X \to S$ be a morphism of schemes. 
\begin{enumerate} 
\item \label{proper-morphism-proper-category}
If $X \to S$ is a perfect proper morphism, then $\Perf(X)$ is proper over $S$. 

\item \label{proper-category-proper-morphism}
If $X$ and $S$ are noetherian, $X \to S$ is separated and of finite type, 
and $\Perf(X)$ is proper over $S$, then $X \to S$ is a perfect proper 
morphism. 

\item \label{geometric-smoothness-criterion}
$\Perf(X)$ is smooth over $S$ if and only if 
$\Delta_{*}(\cO_X) \in \QCoh(X \times_S X)$  
is a perfect complex, where $\Delta \colon X \to X \times_S X$ is the diagonal morphism. 
\item \label{smooth-morphism-smooth-category}
If $X$ is smooth over $S$, then $\Perf(X)$ is smooth over $S$. 
\item \label{smooth-category-smooth-morphism} 
If $X \to S$ is flat and locally of finite presentation and $\Perf(X)$ 
is smooth over $S$, then $X$ is smooth over $S$. 
\item \label{smooth-proper-morphism-smooth-proper-category}
If $X$ is smooth and proper over $S$, then $\Perf(X)$ is smooth and proper over $S$. 
\item \label{smooth-proper-category-smooth-proper-morphism} 
If $X$ and $S$ are noetherian, $X \to S$ is a flat, separated, and of finite type, 
and $\Perf(X)$ is smooth and proper over $S$, then $X$ is smooth and proper 
over $S$. 
\end{enumerate}
\end{lemma}

\begin{proof}
By the criterion of Lemma~\ref{lemma-smooth-proper-criteria}\eqref{properness-criterion} 
and Remark~\ref{remark-mapping-object-scheme}, part~\eqref{proper-morphism-proper-category} 
follows from the fact that pushforward along a perfect proper morphism preserves perfect 
complexes (see~\cite[Example~2.2(a)]{lipman}). 

Part~\eqref{proper-category-proper-morphism} holds by
Remark~\ref{remark-mapping-object-scheme} combined 
with~\cite[Lemma 0.20 and Proposition 0.21]{neeman}. 

By the criterion of Lemma~\ref{lemma-smooth-proper-criteria}\eqref{smoothness-criterion}, 
$\Perf(X)$ is smooth over $S$ if and only if 
\begin{equation*}
\id \in \Func{\QCoh(S)}{\QCoh(X)}{\QCoh(X)}
\end{equation*}
is a compact object. 
But by~\cite[Theorem 1.2(2)]{bzfn} there is an equivalence 
\begin{equation*}
\Func{\QCoh(S)}{\QCoh(X)}{\QCoh(X)}  \simeq \QCoh(X \times_S X) 
\end{equation*} 
under which $\id$ corresponds to $\Delta_{*}(\cO_X)$.  
By \cite[Proposition 3.24]{bzfn} the fiber product $X \times_S X$ is a perfect derived 
scheme; in particular, $\QCoh(X \times_S X)^c = \Perf(X \times_S X)$. 
Hence~\eqref{geometric-smoothness-criterion} holds. 

If $X \to S$ is smooth, then the derived fiber product $X \times_S X$ agrees 
with the usual fiber product of schemes, and 
$\Delta$ is a section of the smooth morphism $X \times_S X \to X$. 
Hence $\Delta$ is a regular immersion 
by~\cite[\href{http://stacks.math.columbia.edu/tag/067R}{Tag 067R}]{stacks-project}, 
and hence a regular closed immersion by our standing separatedness assumptions. 
So $\Delta_*(\cO_X)$ is a perfect complex, which by~\eqref{geometric-smoothness-criterion} 
proves~\eqref{smooth-morphism-smooth-category}. 

By Lemma~\ref{lemma-smoothness-properness-base-change}\eqref{smoothness-base-change} 
below, part~\eqref{smooth-category-smooth-morphism} reduces to the case 
where $S = \Spec(k)$ for a field $k$. 
Moreover, it follows from~\eqref{geometric-smoothness-criterion} that the smoothness of 
$\Perf(X)$ over $S$ is local on $X$, so we may assume $X$ is affine. 
In this case~\eqref{smooth-category-smooth-morphism} holds by~\cite[Proposition 4.17]{lunts-smoothness}. 

Finally, \eqref{smooth-proper-morphism-smooth-proper-category} follows by combining  
\eqref{proper-morphism-proper-category} and \eqref{smooth-morphism-smooth-category}, 
and \eqref{smooth-proper-category-smooth-proper-morphism} by combining 
\eqref{proper-category-proper-morphism} and \eqref{smooth-category-smooth-morphism}.  
\end{proof}

\subsection{Behavior under base change}
\label{subsection-smooth-proper-base-change}

Smoothness and properness of linear categories are stable under base change: 
\begin{lemma} 
\label{lemma-smoothness-properness-base-change}
Let $\cC$ be an $S$-linear category, 
and let $S' \to S$ be a morphism of schemes. 
\begin{enumerate}
\item \label{properness-base-change}
If $\cC$ is proper over $S$, then $\cC \otimes_{\Perf(S)} \Perf(S')$ is 
proper over $S'$. 
\item \label{smoothness-base-change} 
If $\cC$ is smooth over $S$, then $\cC \otimes_{\Perf(S)} \Perf(S')$ is 
smooth over $S'$. 
\end{enumerate}
\end{lemma}

\begin{proof}
Part~\eqref{properness-base-change} follows from Lemmas~\ref{lemma-smooth-proper-criteria}\eqref{properness-criterion}, \ref{lemma-box-tensor-generation}, and \ref{lemma-homs-base-change}. 
For part~\eqref{smoothness-base-change}, we note that if 
$\cC' = \cC \otimes_{\Perf(S)} \Perf(S')$, then there is an equivalence 
\begin{equation*} 
\Func{\QCoh(S')}{\Ind(\cC')}{\Ind(\cC')} \\
\simeq 
\Func{\QCoh(S)}{\Ind(\cC)}{\Ind(\cC)} \otimes_{\QCoh(S)} \QCoh(S'), 
\end{equation*}
under which $\id_{\Ind(\cC')}$ corresponds to $\id_{\Ind(\cC)} \boxtimes \cO_{S'}$. 
Now the result follows from the criterion of Lemma~\ref{lemma-smooth-proper-criteria}\eqref{smoothness-criterion}. 
\end{proof} 

\begin{lemma}
\label{lemma-base-change-smooth-proper}
Let $T \to S$ be a morphism of schemes. 
Let $\cC$ be a $T$-linear category which is smooth and proper over $S$. 
Let $T' \to T$ be a smooth and proper morphism. 
Then the base change 
\begin{equation*}
\cC \otimes_{\Perf(T)} \Perf(T')
\end{equation*} 
is smooth and proper over $S$. 
\end{lemma}

\begin{proof}
By Lemma~\ref{lemma-smooth-proper-morphisms} the category $\Perf(T')$ is 
smooth and proper over $T$. 
Now the result follows from Lemma~\ref{lemma-smooth-proper-dualizable} 
combined with~\cite[Chapter I.1, Corollary 9.5.4]{gaitsgory-DAG}. 
\end{proof}

\begin{remark}
Lemma~\ref{lemma-base-change-smooth-proper} is the analog of the following 
simple geometric fact: if $T$ is an $S$-scheme, $X$ is a $T$-scheme which is 
smooth and proper over $S$, and $T' \to T$ is a smooth and proper morphism, then 
the base change $X \times_{T} T'$ is smooth and proper over $S$. 
Indeed, $X \times_{T} T' \to S$ is the composition of the smooth and proper 
morphisms $X \times_{T} T' \to X$ and $X \to S$. 
\end{remark}

\subsection{Existence of adjoints}  
\label{subsection-smooth-proper-adjoints} 

\begin{lemma}
\label{lemma-smooth-proper-adjoints}
Let $\phi \colon \cC \to \cD$ be an $S$-linear functor between $S$-linear categories, where 
$\cC$ is smooth and proper over $S$ 
and $\cD$ is proper over $S$. 
Then $\phi$ admits left and right adjoints. 
\end{lemma}

\begin{proof}
Since $\cD$ is proper, we have a Yoneda functor 
\begin{equation}
\label{Hom-phi-adjoint}
\cHom_S(\phi(-), -) \colon \cC^{\op} \times \cD \to \Perf(S) . 
\end{equation}
This induces an $S$-linear functor 
\begin{equation*}
\cC^{\op} \otimes_{\Perf(S)} \cD \to \Perf(S). 
\end{equation*}
Since $\cC$ is dualizable, there is an equivalence  
\begin{equation*}
\Func{\Perf(S)}{\cC^{\op} \otimes_{\Perf(S)} \cD}{\Perf(S)} \simeq 
\Func{\Perf(S)}{\cD}{\cC} .
\end{equation*}
Under this equivalence, the above functor corresponds to a functor 
$\phi^! \colon \cD \to \cC$, which is the right adjoint to $\phi$; namely, 
by construction we have an equivalence of functors
\begin{equation*}
\cHom_S(\phi(-), -) \simeq \cHom_S(- , \phi^!(-)) \colon \cC^{\op} \times \cD \to \Perf(S) . 
\end{equation*}
A similar argument shows the existence of a left adjoint to $\phi$. 
\end{proof}

\subsection{Behavior under semiorthogonal decompositions} 
\label{subsection-smooth-proper-sod} 
We will need the following technical result below. 
\begin{lemma}
\label{lemma-functor-category-sod-cpt}
Let $\cC$ and $\cD$ be $S$-linear categories. 
\begin{enumerate}
\item \label{functor-sod-cpt-1}
Let $\cC = \llangle \cA_1, \dots, \cA_n \rrangle$ be an 
$S$-linear semiorthogonal decomposition whose components are admissible. 
Then the embedding and projection functors for the components of the semiorthogonal decomposition
\begin{equation*}
\Func{\QCoh(S)}{\Ind(\cC)}{\Ind(\cD)} = 
\llangle \Func{\QCoh(S)}{\Ind(\cA_1)}{\Ind(\cD)}, \dots, \Func{\QCoh(S)}{\Ind(\cA_n)}{\Ind(\cD)} \rrangle
\end{equation*}
induced by Lemmas~\ref{lemma-sod-small-to-large} and~\ref{lemma-functor-category-sod} 
all preserve compact objects. 

\item \label{functor-sod-cpt-2}
Let $\cD = \llangle \cB_1, \dots, \cB_n \rrangle$ be an 
$S$-linear semiorthogonal decomposition. 
Then the embedding and projection functors for the components of the semiorthogonal decomposition
\begin{equation*}
\Func{\QCoh(S)}{\Ind(\cC)}{\Ind(\cD)} = 
\llangle \Func{\QCoh(S)}{\Ind(\cC)}{\Ind(\cB_1)}, \dots, \Func{\QCoh(S)}{\Ind(\cC)}{\Ind(\cB_n)} \rrangle
\end{equation*}
induced by Lemmas~\ref{lemma-sod-small-to-large} and~\ref{lemma-functor-category-sod} 
all preserve compact objects. 
\end{enumerate}
\end{lemma}

\begin{proof}
We prove~\eqref{functor-sod-cpt-1}. We may assume $n=2$. 
Let $\widehat{\pr}_i \colon \Ind(\cC) \to \Ind(\cA_i)$ be the functor 
induced by $\pr_i \colon \cC \to \cA_i$. 
Then by Lemma~\ref{lemma-sod-large-to-small}, all we need to show is 
that the embedding functor 
\begin{equation}
\label{functor-cpt-1}
\begin{aligned}
\Func{\QCoh(S)}{\Ind(\cA_1)}{\Ind(\cD)} & \to \Func{\QCoh(S)}{\Ind(\cC)}{\Ind(\cD)}  \\
F & \mapsto F \circ \widehat{\pr}_1 
\end{aligned}
\end{equation}
preserves compact objects. 
Note that we have an equivalence 
\begin{equation*}
\Ind(\cA_1^{\op}) \otimes_{\QCoh(S)} \Ind(\cD) \simeq \Func{\QCoh(S)}{\Ind(\cA_1)}{\Ind(\cD)} 
\end{equation*}
induced by the functor 
\begin{equation*}
\begin{aligned}
\cA_1^{\op} \times \cD & \to \Func{\QCoh(S)}{\Ind(\cA_1)}{\Ind(\cD)} \\
(C,D) & \mapsto \cHom_S(C,-) \otimes D, 
\end{aligned}
\end{equation*}
and similarly an equivalence 
\begin{equation*}
\Ind(\cC^{\op}) \otimes_{\QCoh(S)} \Ind(\cD) \simeq \Func{\QCoh(S)}{\Ind(\cC)}{\Ind(\cD)} 
\end{equation*}
induced by the functor 
\begin{equation*}
\begin{aligned}
\cC^{\op} \times \cD & \to \Func{\QCoh(S)}{\Ind(\cC)}{\Ind(\cD)} \\
(C,D) & \mapsto \cHom_S(C,-) \otimes D .
\end{aligned}
\end{equation*}
By Lemma~\ref{lemma-adjoint-adjoints} the functor $\pr_1 = \alpha_1^* \colon \cC \to \cA_1$ 
admits a left adjoint, namely $\rR_{{\cA_2}} \circ \alpha_1 \colon \cA_1 \to \cC$. 
It follows that under the above equivalences, the functor~\eqref{functor-cpt-1} 
is identified with the functor 
\begin{equation*}
\Ind(\cA_1^{\op}) \otimes_{\QCoh(S)} \Ind(\cD) \to \Ind(\cC^{\op}) \otimes_{\QCoh(S)} \Ind(\cD) 
\end{equation*}
given by $\Ind$ of the functor 
\begin{equation*}
\cA_1^{\op} \otimes_{\Perf(S)} \cD \to \cC^{\op} \otimes_{\Perf(S)} \cD 
\end{equation*}
induced by 
\begin{equation*}
(\rR_{{\cA_2}} \circ \alpha_1) \times \id_{\cD} \colon \cA_1^{\op} \times \cD \to \cC^{\op} \times \cD . 
\end{equation*}
Hence the functor~\eqref{functor-cpt-1} preserves compact objects. 

The proof of~\eqref{functor-sod-cpt-2} is similar. Let $\widehat{\beta}_i \colon \Ind(\cB_i) \to \Ind(\cD)$ 
be the functor induced by $\beta_i \colon \cB_i \to \cD$. 
Then the functor 
 \begin{equation}
\label{functor-cpt-2}
\begin{aligned}
\Func{\QCoh(S)}{\Ind(\cC)}{\Ind(\cB_1)} & \to \Func{\QCoh(S)}{\Ind(\cC)}{\Ind(\cD)}  \\
F & \mapsto \widehat{\beta}_1 \circ F  
\end{aligned}
\end{equation}
is identified with $\Ind$ of the functor 
\begin{equation*}
\cC^{\op} \otimes_{\Perf(S)} \cB_1 \to \cC^{\op} \otimes_{\Perf(S)} \cD 
\end{equation*}
induced by 
\begin{equation*}
\id_{\cC^{\op}} \times \beta_1 \colon \cC^{\op} \times \cB_1 \to \cC^{\op} \times \cD. 
\end{equation*}
Now we conclude by the same argument as above. Note that unlike~\eqref{functor-sod-cpt-1}, 
we did not need the components $\cB_i \subset \cD$ to be admissible. 
\end{proof}

\begin{lemma}
\label{lemma-smooth-proper-sod}
Let $\cC$ be an $S$-linear category with an $S$-linear semiorthogonal decomposition 
\begin{equation*}
\cC = \llangle \cA_1, \dots, \cA_n \rrangle. 
\end{equation*}
\begin{enumerate}
\item \label{proper-sod}
If $\cC$ is proper over $S$, then $\cA_i$ is proper over $S$ for all $i$. 
\item \label{smoothness-sod} 
If $\cC$ is smooth over $S$, then $\cA_1$ is smooth over $S$. 
\item \label{smoothness-sod-admissible} 
If all the components $\cA_i \subset \cC$ are admissible, 
then $\cC$ is smooth over $S$ if and only if $\cA_i$ is smooth over $S$ for all $i$. 
\item \label{smooth-proper-sod} 
If $\cC$ is smooth and proper over $S$, then $\cA_i \subset \cC$ is admissible and 
$\cA_i$ is smooth and proper over $S$ for all $i$. 
\end{enumerate}
\end{lemma}

\begin{proof}
Part~\eqref{proper-sod} is immediate from the definitions. 
Parts~\eqref{smoothness-sod}-\eqref{smooth-proper-sod} reduce to the case $n=2$, 
so we assume this for the rest of the proof. 

Let $\halpha_i \colon \Ind(\cA_i) \to \Ind(\cC)$ be the functor induced by 
the embedding $\alpha_i \colon \cA_i \to \cC$, and let $\hpr_i \colon \Ind(\cC) \to \Ind(\cA_i)$ 
be the functor induced by the projection $\pr_i \colon \cC \to \cA_i$. 
Then by Lemmas~\ref{lemma-sod-small-to-large} and~\ref{lemma-functor-category-sod}  
there is a semiorthogonal decomposition 
\begin{align*}
\Fun_{\QCoh(S)}(\Ind(\cC), \Ind(\cC)) = \langle & \Fun_{\QCoh(S)}(\Ind(\cA_1), \Ind(\cA_1)), \Fun_{\QCoh(S)}(\Ind(\cA_1), \Ind(\cA_2)), \\ 
& \Fun_{\QCoh(S)}(\Ind(\cA_2), \Ind(\cA_1)), \Fun_{\QCoh(S)}(\Ind(\cA_2), \Ind(\cA_2 )) \rangle 
\end{align*}
with embedding functors 
\begin{equation*}
\begin{aligned}
f_{ij} \colon \Func{\QCoh(S)}{\Ind(\cA_i)}{\Ind(\cA_j)} & \to \Func{\QCoh(S)}{\Ind(\cC)}{\Ind(\cC)} \\
 F & \mapsto \halpha_j \circ F \circ \hpr_i
\end{aligned}
\end{equation*}
and projection functors
\begin{equation*}
\begin{aligned}
p_{ij} \colon \Func{\QCoh(S)}{\Ind(\cC)}{\Ind(\cC)} & \to \Func{\QCoh(S)}{\Ind(\cA_i)}{\Ind(\cA_j)}  \\
G & \mapsto \hpr_j \circ G \circ \halpha_i .
\end{aligned}
\end{equation*}
Observe that $p_{11}(\id_{\Ind(\cC)}) \simeq \id_{\Ind(\cA_1)}$, $p_{22}(\id_{\Ind(\cC)}) \simeq \id_{\Ind(\cA_2)}$, 
and there is an exact triangle 
\begin{equation*}
f_{22}(\id_{\Ind(\cA_2)}) \to \id_{\Ind(\cC)} \to f_{11}(\id_{\Ind(\cA_1)}) . 
\end{equation*}

By Lemma~\ref{lemma-sod-large-to-small}
the projection $p_{11}$ preserves compact objects. 
Hence if 
\begin{equation*}
\id_{\Ind(\cC)} \in \Func{\QCoh(S)}{\Ind(\cC)}{\Ind(\cC)}
\end{equation*}
is compact then so is
\begin{equation*}
\id_{\Ind(\cA_1)} \in \Func{\QCoh(S)}{\Ind(\cA_1)}{\Ind(\cA_1)}. 
\end{equation*} 
By Lemma~\ref{lemma-smooth-proper-criteria}\eqref{smoothness-criterion} this proves~\eqref{smoothness-sod}.  

If all $\cA_i \subset \cC$ are admissible, then by Lemma~\ref{lemma-functor-category-sod-cpt} 
all of the embedding and projection functors $f_{ij}, p_{ij}$ preserve compact objects. 
Hence by the observation above, $\id_{\Ind(\cC)}$ is a compact object if and only 
if $\id_{\Ind(\cA_1)}$ and $\id_{\Ind(\cA_2)}$ are compact objects.  
By Lemma~\ref{lemma-smooth-proper-criteria}\eqref{smoothness-criterion} 
this proves~\eqref{smoothness-sod-admissible}. 

Finally, assume $\cC$ is smooth and proper over $S$. 
Then by parts~\eqref{proper-sod} and~\eqref{smoothness-sod} the category 
$\cA_1$ is smooth and proper over $S$. 
By Lemma~\ref{lemma-smooth-proper-adjoints} it follows that $\cA_1 \subset \cC$ 
is admissible. 
Hence by Lemma~\ref{lemma-admissible-sod} there is a decomposition $\cC = \llangle \cA_1^{\perp}, \cA_1 \rrangle$, 
and by Lemma~\ref{lemma-mutation-functors} there is an equivalence $\cA_2 \simeq \cA_1^{\perp}$. 
So we can apply the above argument to conclude $\cA_2$ is also smooth and proper over $S$. 
This proves~\eqref{smooth-proper-sod}. 
\end{proof}

\subsection{Serre functors} 
\label{subsection-serre-functor}

If $\cC$ is a proper $S$-linear category, then by Lemma~\ref{lemma-smooth-proper-criteria} 
the mapping object $\cHom_S(C,D)$ lies in $\Perf(S)$ for every $C,D \in \cC$. 
In this situation, we have functors 
\begin{align*}
\cHom_{S} \colon & \cC^\op \times \cC \to  \Perf(S)  &  
\cHom_{S}^\op \colon & \cC^\op \times \cC \to \Perf(S)^\op   \\
& (X, Y)  \mapsto \cHom_{S}(X,Y) &  & (X,Y) \mapsto  \cHom_{S}(Y,X), 
\end{align*}
where $(-)^\op$ denotes the opposite category. 
Note that dualization gives an equivalence 
\begin{equation*}
(-)^{\svee} \colon \Perf(S)^{\op} \to \Perf(S). 
\end{equation*}

\begin{definition}
\label{definition-serre-functor}
Let $\cC$ be a proper $S$-linear category. 
A \emph{relative Serre functor} for $\cC$ over $S$ is an $S$-linear autoequivalence $\rS_{\cC/S} \colon \cC \xrightarrow{\sim} \cC$ 
such that there is a commutative diagram 
\begin{equation*}
\xymatrix{
\cC^{\op} \times \cC \ar[rr]^{\id \times \rS_{\cC/S}} \ar[d]_{\cHom_{S}^\op} && \cC^\op \times \cC \ar[d]^{\cHom_{S}} \\  
\Perf(S)^\op \ar[rr]^{(-)^{\svee}} &&  \Perf(S)
}
\end{equation*}
\end{definition} 

In other words, 
a relative Serre functor for $\cC$ is an $S$-linear autoequivalence characterized by the existence of natural equivalences
\begin{equation*}
 \cHom_{S}(C, \rS_{\cC/S}(D)) \simeq \cHom_{S}(D, C)^{\svee}
\end{equation*}
for all objects $C, D \in \cC$. 

\begin{remark}
Let $\pi \colon X \to S$ be a proper morphism of schemes such that $\pi_*$ preserves 
perfect complexes, and such that a relative dualizing complex $\omega_{\pi}$ 
exists and is a shift of a line bundle. 
Then by Grothendieck duality $\rS_{X/S} = (- \otimes \omega_{\pi})$ is a relative Serre functor for 
$\Perf(X)$ over $S$. 
\end{remark}

The following useful result follows easily from the definitions, see \cite[Proposition 3.6]{bondal-kapranov}. 

\begin{lemma}
\label{lemma-serre-functor-image}
Let $\cC$ be a proper $S$-linear category that admits a relative Serre functor $\rS_{\cC/S}$. 
\begin{enumerate}
\item If $\cA \subset \cC$ is a right admissible subcategory, then $\rS_{\cC/S}(\cA) = \cA^{\perp \perp}$. 
\item If $\cA \subset \cC$ is a left admissible subcategory, then $\rS^{-1}_{\cC/S}(\cA) = {^{\perp\perp}} \cA$. 
\end{enumerate}
\end{lemma}

A relative Serre functor exists for any smooth and proper category: 
\begin{lemma}
\label{lemma-serre-functor-exists}
Let $\cC$ be a smooth and proper $S$-linear category. 
Then there exists a relative Serre functor $\rS_{\cC/S}$ of $\cC$ over $S$. 
\end{lemma}

\begin{proof}
Let $\tau(-,-) \colon \cC^\op \times \cC \to \cC \times \cC^\op$ be the transposition functor. 
Then the argument of Lemma~\ref{lemma-smooth-proper-adjoints} applied to 
\begin{equation*}
\cHom_{S}(\tau(-,-))^{\svee} \colon \cC^\op \times \cC \to \Perf(S) 
\end{equation*}
in place of \eqref{Hom-phi-adjoint} shows that there exists a functor $\rS_{\cC/S} \colon \cC \to \cC$ 
such that there is a commutative diagram as in Definition~\ref{definition-serre-functor}. 
It remains to show that $\rS_{\cC/S}$ is an equivalence. 
The following computation for $C, D \in \cC$ shows that $\rS_{\cC/S}$ is fully faithful: 
\begin{align*}
\cHom_S(\rS_{\cC/S}(C), \rS_{\cC/S}(D)) & \simeq 
\cHom_S(D, \rS_{\cC/S}(C))^{\svee} \\ 
& \simeq \cHom_S(C, D)^{\svee \svee} \\ 
& \simeq \cHom_S(C,D). 
\end{align*}
Next note that 
by Lemma~\ref{lemma-smooth-proper-adjoints} there exists a left adjoint 
$\rS_{\cC/S}^* \colon \cC \to \cC$. 
For $C, D \in \cC$, we compute: 
\begin{align*}
\cHom_S(C, \rS_{\cC/S}\rS^*_{\cC/S}(D)) & \simeq 
\cHom_S(\rS_{\cC/S}^*(D), C)^{\svee} \\ 
& \simeq \cHom_S(D, \rS_{\cC/S}(C))^{\svee} \\ 
& \simeq \cHom_S(C, D)^{\svee \svee} \\ 
& \simeq \cHom_S(C, D). 
\end{align*} 
It follows that $\rS_{\cC/S}\rS^*_{\cC/S}(D) \simeq D$, and hence $\rS_{\cC/S}$ is essentially surjective. 
\end{proof}

Finally, we note that if a relative Serre functor exists, it is automatically linear over any scheme over which 
$\cC$ is defined. 

\begin{lemma}
\label{lemma-serre-functor-linear}
Let $\cC$ be a proper $S$-linear category with a relative Serre functor $\rS_{\cC/S}$. 
Assume the $S$-linear structure of $\cC$ is induced by a $T$-linear structure, where 
$T$ is an $S$-scheme. 
Then $\rS_{\cC/S}$ has the structure of a $T$-linear functor. 
\end{lemma} 

\begin{proof}
Similar to the proof of Lemma \ref{lemma-adjoint-S-linear}. 
\end{proof}

\subsection{Critical loci} 
\label{subsection-critical-loci} 

\begin{definition}
\label{definition-critical-loci}
Let $\cC$ be an $S$-linear category. 
Let $s \in S$ be a point and let $\Spec(\kappa(s)) \to S$ be the 
corresponding morphism from the spectrum of the residue field 
at $s \in S$. 
We say $s$ is a \emph{critical point} for $\cC$ if the base change 
\begin{equation*}
\cC \otimes_{\Perf(S)} \Perf(\Spec(\kappa(s))) 
\end{equation*} 
fails to be smooth over $\Spec(\kappa(s))$. 
The \emph{critical locus} of $\cC$ is the set 
\begin{equation*}
\Crit_S(\cC) = \{ \, s \in S \mid s \text{ is a critical point for } \cC \, \} \subset S . 
\end{equation*}
\end{definition}

\begin{remark}
\label{remark-crit-geometric}
Let $X \to S$ be a morphism of schemes which is flat and locally of finite presentation. 
Then by Lemma~\ref{lemma-smooth-proper-morphisms}\eqref{smooth-category-smooth-morphism}, 
we have $s \in \Crit_S(\Perf(X))$ if and only if the fiber $X_s \to \Spec(\kappa(s))$ 
fails to be smooth 
(the flatness of $X \to S$ guarantees $X_s \to \Spec(\kappa(s))$ is an ordinary, 
i.e. not derived, scheme). 
\end{remark}

\begin{remark}
\label{remark-crit-closed}
If $X \to S$ is a morphism of schemes which is flat, proper, and of finite presentation, 
then $\Crit_S(\Perf(X))$ is a closed subset of $S$. 
We expect that, under suitable assumptions, this holds for general $S$-linear categories $\cC$, 
but we have not checked this. 
\end{remark}

\begin{lemma}
\label{lemma-crit-smooth}
Let $\cC$ be a smooth $S$-linear category. 
Then $\Crit_S(\cC) = \varnothing$. 
\end{lemma}

\begin{proof}
Follows from Lemma~\ref{lemma-smoothness-properness-base-change}. 
\end{proof}

\begin{lemma}
\label{lemma-crit-sod}
Let $\cC$ be an $S$-linear category with an $S$-linear semiorthogonal decomposition 
\begin{equation*}
\cC = \llangle \cA_1, \dots, \cA_n \rrangle  
\end{equation*}
whose components are admissible. 
Then we have 
\begin{equation*}
\Crit_S(\cC) = \bigcup_{i=1}^n \Crit_S(\cA_i) \subset S. 
\end{equation*}
\end{lemma}

\begin{proof}
Follows from Lemmas~\ref{lemma-smooth-proper-sod}\eqref{smoothness-sod-admissible} 
and~\ref{lemma-base-change-admissible}. 
\end{proof}

\subsection{Bounded coherent categories}
\label{subsection-Db} 
Given an $S$-scheme $X$, in addition to the $S$-linear category $\Perf(X)$, 
there is another naturally associated $S$-linear category: the 
full subcategory $\Db(X) \subset \QCoh(S)$ spanned by complexes 
with bounded coherent cohomology. 
By the results of \cite[Theorem 1.1.3]{bznp}, under suitable hypotheses 
$\Perf(X)$ in fact determines $\Db(X)$. 

\begin{theorem}
\label{theorem-dual-Db}
Let $\pi \colon X \to S$ be a proper morphism of finite presentation between 
derived schemes over a field of characteristic $0$, with $S$ locally noetherian. 
Then there is an equivalence 
\begin{align*}
\Db(X)  & \xrightarrow{\sim} \Func{\Perf(S)}{\Perf(X)^\op}{\Db(S)} \\
\cE       & \mapsto \cHom_S(-, \cE) , 
\end{align*}
where $\Perf(X)^\op$ denotes the opposite category. 
\end{theorem}

\begin{proof}
By \cite[Theorem 1.1.3]{bznp} there is an equivalence 
\begin{align*}
\Db(X)  & \xrightarrow{\sim} \Func{\Perf(S)}{\Perf(X)}{\Db(S)} \\
\cE       & \mapsto \pi_* \circ (- \otimes \cE ) . 
\end{align*}
Now the result follows by composing with the equivalence 
$\Perf(X) \simeq \Perf(X)^\op$ given by dualization.  
\end{proof}

This motivates the following definition. 

\begin{definition}
Let $\cC$ be a proper $S$-linear category, where $S$ is locally 
noetherian over a field of characteristic $0$. The 
\emph{bounded coherent category} of $\cC$ is 
\begin{equation*}
\cC^{\Coh}= \Func{\Perf(S)}{\cC^\op}{\Db(S)} . 
\end{equation*}
\end{definition}

Note that in the above situation, there is a canonical functor 
\begin{equation*}
\cC \rightarrow \cC^{\Coh} , \quad C \mapsto \cHom_S(-, C ) , 
\end{equation*}
which in the geometric case $\cC = \Perf(X)$ corresponds to the embedding 
$\Perf(X) \subset \Db(X)$. 
Next we show that under suitable hypotheses, a semiorthogonal 
decomposition of $\cC$ induces a semiorthogonal decomposition on 
bounded coherent categories, compatible with the above functor. 
This can be used to recover the results of \cite{kuznetsov-base-change} 
and remove the geometricity assumptions there, see Remark~\ref{remark-sod-Db} below. 

\begin{proposition} 
\label{proposition-sod-Db} 
Let $\cC$ be a proper $S$-linear category, where $S$ is locally 
noetherian over a field of characteristic $0$. 
Let $\cC = \llangle \cA, \cB \rrangle$ be an $S$-linear semiorthogonal decomposition 
with $\cB \subset \cC$ admissible. 
Then there is an $S$-linear semiorthogonal decomposition 
\begin{equation*}
\cC^\Coh = \llangle \cA^\Coh, \cB^\Coh \rrangle , 
\end{equation*}
with the property that the canonical functor $\cC \to \cC^\Coh$ 
restricts to the canonical functors $\cA \to \cA^\Coh$ and 
$\cB \to \cB^\Coh$ on each component. 
\end{proposition} 

\begin{proof}
Let $\alpha \colon \cA \to \cC$ and $\beta \colon \cB \to \cC$ be the inclusions. 
Let $\alpha' \colon \cA' = {^\perp}\cB \to \cC$ be the inclusion of the orthogonal, 
so that by Lemma~\ref{lemma-admissible-sod} 
there is a semiorthogonal decomposition $\cC = \llangle \cB, \cA' \rrangle$. 
Passing to opposite categories, we obtain a semiorthogonal decomposition 
$\cC^\op = \llangle (\cA')^\op, \cB^\op \rrangle$, with embedding functors 
induced by $\alpha'$ and $\beta$. 
Note that the functor $\cC^\op \to \cB^\op$ induced by the left adjoint 
$\beta^* \colon \cC \to \cB$ is the \emph{right} adjoint to the inclusion 
$\cB^\op \to \cC^\op$; similarly, the functor $\cC^\op \to (\cA')^\op$ induced 
by $(\alpha')^! \colon \cC \to \cA'$ is \emph{left} adjoint to $(\cA')^\op \to \cC^\op$. 
By Lemma~\ref{lemma-functor-category-sod}, we conclude that there 
is a semiorthogonal decomposition 
\begin{equation*}
\cC^\Coh = \llangle (\cA')^\Coh, \cB^\Coh \rrangle . 
\end{equation*} 
Note that by Lemma~\ref{lemma-mutation-functors} the functor 
$\rR_{\cB} \circ \alpha$ induces an equivalence $\cA \simeq \cA'$, 
and hence an equivalence $\cA^\Coh \simeq (\cA')^\Coh$. 
Hence it remains to show that the above decomposition is compatible 
with the canonical functors to the bounded coherent categories. 
By the description of the embedding functors in Lemma~\ref{lemma-mutation-functors}, 
we see that the composition $\cA' \to (\cA')^\Coh \to \cC^\Coh$ 
is given by 
\begin{equation*} 
A' \mapsto \cHom_S(-, A') \mapsto \cHom_S((\alpha')^!(-), A') . 
\end{equation*} 
By Lemma~\ref{lemma-adjoint-adjoints} we have 
\begin{equation*}
\cHom_S((\alpha')^!(-), A') \simeq \cHom_S(-, \rL_{\cB} \circ \alpha'(A')). 
\end{equation*} 
Hence by Lemma~\ref{lemma-mutation-functors}, it follows that under the 
identifications $\cA \simeq \cA'$ and $\cA^{\Coh} \simeq (\cA')^\Coh$ given 
by $\rR_{\cB} \circ \alpha$, the canonical functor $\cC \to \cC^\Coh$ restricts to 
the canonical functor $\cA \to \cA^\Coh$. 
A similar (even easier) argument shows the compatibility of the canonical 
functors to the bounded coherent categories for the second component 
$\cB^\Coh \subset \cC^\Coh$. 
\end{proof} 

\begin{remark}
\label{remark-sod-Db} 
In \cite[Theorem 5.6]{kuznetsov-base-change}, the problem of inducing a semiorthogonal 
decomposition of $\Db(X)$ from one of $\Perf(X)$ is addressed under 
certain assumptions. 
Let $\cA \subset \Perf(X)$ be a component of a semiorthogonal decomposition 
with admissible components. 
Then the 
induced subcategory $\hat{\cA} \subset \Db(X)$ defined in 
\cite{kuznetsov-base-change} can be described as follows. 
By taking Ind-categories, we obtain an admissible subcategory 
$\Ind(\cA) \subset \QCoh(X)$. 
Note that also $\Db(X) \subset \QCoh(X)$. 
Then $\hat{\cA} \subset \Db(X)$ is the full subcategory 
with objects given by the intersection 
\begin{equation*}
\hat{\cA} = \Ind(\cA) \cap \Db(X). 
\end{equation*} 

On the other hand, in Proposition~\ref{proposition-sod-Db}  
we have defined a subcategory $\cA^\Coh \subset \Db(X)$. 
The categories $\hat{\cA}$ and $\cA^{\Coh}$ agree when the assumptions of 
\cite[Theorem 5.6]{kuznetsov-base-change} and 
Proposition~\ref{proposition-sod-Db} both hold. 
Indeed, we can describe $\cA^\Coh \subset \Db(X)$ by the formula 
\begin{equation*}
\cA^\Coh = \Ind(\cA^\Coh) \cap \Db(X), 
\end{equation*} 
where the intersection is taken in $\Ind(\Db(X))$. 
Note that the inclusions $\cA \subset \cA^\Coh$ and $\Perf(X) \subset \Db(X)$ 
induce inclusions $\Ind(\cA) \subset \Ind(\cA^\Coh)$ and $\QCoh(X) \subset \Ind(\Db(X))$. 
Hence the above formulas imply $\hat{\cA} \subset \cA^\Coh$. 
Writing $\Perf(X) = \llangle \cA, \cB \rrangle$, we similarly have 
$\hat{\cB} \subset \cB^{\Coh}$. 
Since $\Db(X) = \langle \hat{\cA}, \hat{\cB} \rangle = \llangle \cA^\Coh, \cB^\Coh \rrangle$, 
it follows that $\hat{\cA} = \cA^{\Coh}$ and $\hat{\cB} = \cB^\Coh$. 

An advantage of our definition of $\cA^{\Coh}$ is that it is defined 
canonically in terms of~$\cA$, and does not require an embedding into the 
derived category of a variety or the use of $t$-structures as in \cite{kuznetsov-base-change}. 
\end{remark}


\section{A formalism of Fourier--Mukai kernels}
\label{section-kernels} 

In this section, we introduce a formalism for describing certain 
functors between linear categories in terms of ``Fourier--Mukai kernels''. 
After reviewing the classical situation in \S\ref{subsection-FM-classical}, 
we define our notion of Fourier--Mukai kernels in the noncommutative 
setting in \S\ref{subsection-NC-FM}. 
In \S\ref{subsection-FM-properties} we develop some basic properties 
of these functors. 

\subsection{The geometric setting} 
\label{subsection-FM-classical}
Let $X_1 \to T$ and $X_2 \to T$ be morphisms of schemes, such that 
pushforward along the projection $\pr_2 \colon X_2 \times_T X_1 \to X_2$ 
preserves perfect complexes. 

\begin{remark}
By~\cite[Example~2.2(a)]{lipman}, the assumption on $\pr_{2*}$ 
is satisfied if $\pr_2$ is a perfect proper morphism. 
\end{remark}

In the above situation, we call an object  
\begin{equation*}
\cE \in \Perf ( X_2 \times_T X_1 )
\end{equation*}
a Fourier--Mukai kernel. To any such $\cE$, there is an associated 
$T$-linear functor 
\begin{equation*} 
\Phi_{\cE} \colon \Perf(X_1) \to \Perf(X_2), \quad F \mapsto \pr_{2*}(\cE \otimes \pr_1^*F). 
\end{equation*}

\begin{remark} 
Our convention that a kernel $\cE \in \Perf ( X_2 \times_T X_1 )$ corresponds to a 
functor $\Perf(X_2) \to \Perf(X_1)$, instead of $\Perf(X_2) \to \Perf(X_1)$, is 
slightly unconventional. But it will be convenient later. 
\end{remark}

This construction $\cE \mapsto \Phi_{\cE}$ gives rise to a functor 
\begin{equation}
\label{FM-functor}
\Perf ( X_2 \times_T X_1 ) \to \Fun_{\Perf(T)}(\Perf(X_1), \Perf(X_2)) 
\end{equation}
from the category of Fourier--Mukai kernels to the category of $T$-linear 
functors. 
We have the following fundamental result; the second part gives a criterion for the 
above functor to be an equivalence. 

\begin{theorem}[\protect{\cite[Theorem 1.2]{bzfn}}] 
\label{theorem-FM-functors}
\begin{enumerate}
\item \label{fiber-product}
For any $X_1 \to T$ and $X_2 \to T$, there is an equivalence 
\begin{equation*}
\Perf(X_1 \times_T X_2) \simeq \Perf(X_1) \otimes_{\Perf(T)} \Perf(X_2). 
\end{equation*}
\item \label{FM}
If $X_1 \to T$ is smooth and proper, \eqref{FM-functor} is an equivalence. 
\end{enumerate}
\end{theorem} 

\begin{remark}
In \cite{bzfn}, the theorem is formulated for $X_1, X_2,$ and $T$ being so-called perfect stacks. 
Any quasi-compact and separated scheme is a perfect 
stack \cite[Proposition 3.19]{bzfn}, so with our conventions from \S\ref{subsection-conventions} 
any scheme is perfect. 
\end{remark}

\subsection{The noncommutative setting} 
\label{subsection-NC-FM} 

We consider the following situation. 

\begin{setup} 
\label{setup-FM}
For $i=1,2$, assume given the following data: 
\begin{itemize}
\item[--] A diagram of schemes 
\begin{equation}
\label{Zi}
\vcenter{
\xymatrix{ 
& Z_ i\ar[dl] \ar[dr] \\
S_i & & T
} }  
\end{equation}
such that the pushforward along the 
projection $\pr_2 \colon Z_2 \times_T Z_1 \to Z_2$ preserves 
perfect complexes.  
\item[--] An $S_i$-linear category $\cC_i$.
\end{itemize}
\end{setup} 

In the above setup, the base change
\begin{equation*}
\cC_i  \otimes_{\Perf(S_i)} \Perf(Z_i) 
\end{equation*}
has the structure of $T$-linear category. 
In what follows, this category plays the role of 
the scheme $X_i$ from our discussion in \S\ref{subsection-FM-classical} 
of Fourier--Mukai functors in the geometric setting.  

\begin{definition}
In Setup~\ref{setup-FM}, we define the category of \emph{Fourier--Mukai kernels} 
as 
\begin{equation*}
\FM{\cC_1}{S_1}{\cC_2}{S_2}{Z_1}{Z_2}{T} = 
\Func{\Perf(S_1)}{\cC_1}{\cC_2 \otimes_{\Perf(S_2)}\Perf(Z_2 \times_T Z_1)} . 
\end{equation*}
When $\cC_i$ and $S_i$ are clear from the context, we sometimes 
refer to an object of the above category as a \emph{$Z_2 \times_T Z_1$-kernel}. 
\end{definition}

This definition is motivated by the following result. 
\begin{proposition}
\label{proposition-kernels} 
In Setup~\ref{setup-FM}, there is a natural functor 
\begin{equation*}
\FM{\cC_1}{S_1}{\cC_2}{S_2}{Z_1}{Z_2}{T} \to 
\Func{\Perf(T)}{\cC_1 \otimes_{\Perf(S_1)} \Perf(Z_1)}{\cC_2 \otimes_{\Perf(S_2)} \Perf(Z_2)}, 
\end{equation*}
which is an equivalence if $Z_1 \to T$ is smooth and proper. 
\end{proposition}

\begin{proof}
The classical functor 
\begin{equation*}
\Perf(Z_2 \times_T Z_1)  \to \Func{\Perf(T)}{\Perf(Z_1)}{\Perf(Z_2)} 
\end{equation*}
induces a functor
\begin{equation*}
\cC_2 \otimes_{\Perf(S_2)} \Perf(Z_2 \times_T Z_1) \to \cC_2 \otimes_{\Perf(S_2)} \Func{\Perf(T)}{\Perf(Z_1)}{\Perf(Z_2)} .
\end{equation*}
Composing with the canonical functor 
\begin{equation*}
\cC_2 \otimes_{\Perf(S_2)} \Func{\Perf(T)}{\Perf(Z_1)}{\Perf(Z_2)}  \to 
\Func{\Perf(T)}{\Perf(Z_1)}{\cC_2 \otimes_{\Perf(S_2)} \Perf(Z_2)} ,
\end{equation*}
this induces a functor 
\begin{equation*}
\FM{\cC_1}{S_1}{\cC_2}{S_2}{Z_1}{Z_2}{T} \to 
 \Func{\Perf(S_1)}{\cC_1}{\Func{\Perf(T)}{\Perf(Z_1)}{\cC_2 \otimes_{\Perf(S_2)} \Perf(Z_2)}} . 
\end{equation*}
Composing with the canonical equivalence
\begin{multline*}
\Func{\Perf(S_1)}{\cC_1}{\Func{\Perf(T)}{\Perf(Z_1)}{\cC_2 \otimes_{\Perf(S_2)} \Perf(Z_2)}}  \\
\xrightarrow{\sim} 
\Func{\Perf(T)}{\cC_1 \otimes_{\Perf(S_1)} \Perf(Z_1)}{\cC_2 \otimes_{\Perf(S_2)} \Perf(Z_2)}, 
\end{multline*}
gives the sought-after functor. 
It follows from Theorem~\ref{theorem-FM-functors} and the construction 
that this functor is an equivalence if $Z_1 \to T$ is smooth and proper. 
\end{proof}

In Setup~\ref{setup-FM}, given a kernel 
\begin{equation*}
\cE \in \FM{\cC_1}{S_1}{\cC_2}{S_2}{Z_1}{Z_2}{T} ,
\end{equation*}
we denote by 
\begin{equation*}
\Phi_{\cE} \in \Func{\Perf(T)}{ \cC_1 \otimes_{\Perf(S_1)} \Perf(Z_1) }{ \cC_2 \otimes_{\Perf(S_2)} \Perf(Z_2)}
\end{equation*}
the corresponding functor given by Proposition~\ref{proposition-kernels}.

\subsection{Properties of kernel functors} 
\label{subsection-FM-properties}
For the rest of this section, we describe the the effect of certain operations 
on $\cE$ in terms of the associated functor $\Phi_{\cE}$. 

\begin{setup}
\label{setup-FM-2}
In addition to the data of Setup~\ref{setup-FM}, 
assume given for $i=1,2$ another diagram of schemes
\begin{equation}
\label{Ziprime}
\vcenter{
\xymatrix{ 
& Z'_ i\ar[dl] \ar[dr] \\
S_i & & T'
} }  
\end{equation}
such that the pushforward along the 
projection $\pr_2 \colon Z'_2 \times_{T'} Z'_1 \to Z'_2$ preserves 
perfect complexes.  
\end{setup}

In the above setup, if we are given a functor 
\begin{equation*}
F \colon \Perf(Z_2 \times_T Z_1)  \to \Perf(Z'_2 \times_{T'} Z'_1) 
\end{equation*}
which is $S_2 \times_S S_1$-linear, 
we denote by the same symbol the induced functor 
\begin{equation*}
F \colon 
\FM{\cC_1}{S_1}{\cC_2}{S_2}{Z_1}{Z_2}{T} \to 
\FM{\cC_1}{S_1}{\cC_2}{S_2}{Z'_1}{Z'_2}{T'}
\end{equation*}
on kernel categories. 
In particular, this construction allows us to pullback and pushforward kernels along 
morphisms $Z'_2 \times_{T'} Z'_1 \to Z_2 \times_T Z_1$, and to tensor kernels with 
objects of $\Perf(Z_2 \times_T Z_1)$. 
As we discuss below, these operations formally behave the same as in the geometric case. 

The next two results follow by unwinding the definitions. 
\begin{lemma}
\label{lemma-functor-triangle-to-kernel}
Assume we are in Setup~\ref{setup-FM-2}.
Let 
\begin{equation*}
F_1 \to F_2 \to F_3
\end{equation*} 
be an exact triangle of $S_2 \times_S S_1$-linear functors 
$\Perf(Z_2 \times_T Z_1)  \to \Perf(Z'_2 \times_{T'} Z'_1)$. 
Let 
\begin{equation*}
\cE \in \FM{\cC_1}{S_1}{\cC_2}{S_2}{Z_1}{Z_2}{T}
\end{equation*}
be a kernel. Then there is an induced exact triangle of $Z'_2 \times_{T'} Z'_1$-kernels 
\begin{equation*}
F_1(\cE) \to F_2(\cE) \to F_3(\cE). 
\end{equation*}
\end{lemma}

\begin{lemma}
Assume we are in Setup~\ref{setup-FM-2}.
Let 
\begin{equation*}
\cE_1 \to \cE_2 \to \cE_3
\end{equation*} 
be an exact triangle of kernels in 
$\FM{\cC_1}{S_1}{\cC_2}{S_2}{Z_1}{Z_2}{T}$. 
Let 
\begin{equation*}
F \colon \Perf(Z_2 \times_T Z_1) \to \Perf(Z'_2 \times_{T'} Z'_1) 
\end{equation*}
be an $S_2 \times_S S_1$-linear functor. 
Then there is an induced exact triangle of kernels 
\begin{equation*}
F(\cE_1) \to F(\cE_2) \to F(\cE_3) .
\end{equation*}
in $\FM{\cC_1}{S_1}{\cC_2}{S_2}{Z'_1}{Z'_2}{T'}$. 
\end{lemma}

We leave it to the reader to check that the next two results reduce to 
the geometric case, where they are well-known. 
\begin{lemma}
\label{lemma-projection-formula}
Assume we are in Setup~\ref{setup-FM-2}.
Let $f \colon Z'_2 \times_{T'} Z'_1 \to Z_2 \times_T Z_1$ be a morphism such 
that $f_*$ preserves perfect complexes.
\begin{enumerate}
\item   
Let 
\begin{equation*}
\cE \in \FM{\cC_1}{S_1}{\cC_2}{S_2}{Z_1}{Z_2}{T}
\end{equation*}
be a kernel. 
Let $\cF' \in \Perf(Z'_2 \times_{T'} Z'_1)$. 
Then there is an equivalence of kernels 
\begin{equation*}
f_*(f^*(\cE) \otimes \cF') \simeq \cE \otimes f_*(\cF').
\end{equation*}

\item 
Let 
\begin{equation*}
\cE' \in \FM{\cC_1}{S_1}{\cC_2}{S_2}{Z'_1}{Z'_2}{T'}
\end{equation*}
be a kernel. 
Let $\cF \in \Perf(Z_2 \times_{T} Z_1)$. 
Then there is an equivalence of kernels 
\begin{equation*}
f_*(\cE' \otimes f^*(\cF)) \simeq f_*(\cE') \otimes \cF.
\end{equation*}
\end{enumerate}
\end{lemma}

\begin{lemma}
\label{lemma-twist-kernel}
Let 
\begin{equation*}
\cE \in \FM{\cC_1}{S_1}{\cC_2}{S_2}{Z_1}{Z_2}{T}
\end{equation*}
be a kernel. 
Let $\cF_2 \in \Perf(Z_2), \cF_1 \in \Perf(Z_1)$, and let 
$\cF_2 \boxtimes \cF_1 \in \Perf(Z_2 \times_T Z_1)$ be their exterior product. 
Then there is an equivalence 
\begin{equation*}
\Phi_{\cE \otimes (\cF_2 \boxtimes \cF_1)} \simeq (- \otimes \cF_2) \circ \Phi_{\cE} \circ (- \otimes \cF_1) . 
\end{equation*}
\end{lemma}

\begin{setup}
\label{setup-FM-3}
Assume that in Setup~\ref{setup-FM-2}, for $i=1,2$ 
the diagrams~\eqref{Zi} and~\eqref{Ziprime} fit 
into a commutative diagram  
\begin{equation}
\vcenter{
\label{Zi-diagram}
\xymatrix{ 
& \ar[dl] Z'_i \ar[r] \ar[d]^{f_i} & T' \ar[d] \\ 
S_i & \ar[l] Z_i \ar[r] & T 
}  } 
\end{equation}
\end{setup}

In the above setup, we have a morphism 
\begin{equation*}
f_2 \times f_1 \colon Z'_2 \times_{T'} Z'_1 \to Z_2 \times_{T} Z_1 , 
\end{equation*}
which by pushforward (when it preserves perfect complexes) and 
pullback induces functors between kernel categories as above. 
Further, the morphism $f_i$ induces by pushforward (when it 
preserves perfect complexes) and pullback functors 
\begin{align*}
f_{i*} & \colon \cC_i \otimes_{\Perf(S_i)} \Perf(Z'_i) \to \cC_i \otimes_{\Perf(S_i)} \Perf(Z_i), \\
f_i^* & \colon  \cC_i \otimes_{\Perf(S_i)} \Perf(Z_i) \to \cC_i \otimes_{\Perf(S_i)} \Perf(Z'_i). 
\end{align*}

The next two results easily reduce to the geometric case, where they are straightforward. 
\begin{lemma}
\label{lemma-pushforward-kernel}
In Setup~\ref{setup-FM-3}, assume further that 
$(f_2 \times f_1)_*$ preserves perfect complexes. 
Let 
\begin{equation*}
\cE' \in 
\FM{\cC_1}{S_1}{\cC_2}{S_2}{Z'_1}{Z'_2}{T'}
\end{equation*}
be a kernel, and 
\begin{equation*}
\cE = (f_2 \times f_1)_*\cE' \in 
\FM{\cC_1}{S_1}{\cC_2}{S_2}{Z_1}{Z_2}{T} . 
\end{equation*}
Then there is an equivalence  
\begin{equation*}
\Phi_{\cE} \simeq f_{2*} \circ \Phi_{\cE'} \circ f_1^*. 
\end{equation*}
\end{lemma}

\begin{lemma}
\label{lemma-pullback-kernel}
Assume we are in Setup~\ref{setup-FM-3}. 
Let 
\begin{equation*}
\cE \in 
\FM{\cC_1}{S_1}{\cC_2}{S_2}{Z_1}{Z_2}{T}
\end{equation*}
be a kernel, and 
\begin{equation*}
\cE'  = (f_2 \times f_1)^*\cE \in 
\FM{\cC_1}{S_1}{\cC_2}{S_2}{Z'_1}{Z'_2}{T'} . 
\end{equation*}
\begin{enumerate}
\item 
\label{pullback-kernel-1}
If the square 
\begin{equation*}
\xymatrix{
Z'_2 \ar[d]_{f_2} & \ar[l] Z'_2 \times_{T'} Z'_1 \ar[d]^{f_2 \times f_1} \\ 
Z_2 & \ar[l] Z_2 \times_T Z_1 
}
\end{equation*}
is cartesian, then there is an equivalence 
\begin{equation}
\label{base-change-pullbacks}
\Phi_{\cE'} \circ f_1^* \simeq f_2^* \circ \Phi_{\cE} .
\end{equation} 

\item 
\label{pullback-kernel-2}
If the square 
\begin{equation*}
\xymatrix{
Z'_2 \times_{T'} Z'_1 \ar[d]_{f_2 \times f_1} \ar[r] & Z'_1 \ar[d]^{f_1} \\ 
Z_2 \times_{T} Z_1 \ar[r] & Z_1 
}
\end{equation*}
is cartesian, then there is an equivalence 
\begin{equation}
\label{base-change-pushforwards}
f_{2*} \circ \Phi_{\cE'} \simeq \Phi_{\cE} \circ f_{1*} . 
\end{equation}
\end{enumerate}
\end{lemma}

\begin{remark}
\label{remark-pullback-kernel}
For $i=1,2$, the assumption of 
Lemma~\ref{lemma-pullback-kernel}($i$) automatically holds if the square in diagram~\eqref{Zi-diagram} is cartesian. 
\end{remark}

\begin{lemma}
\label{lemma-pullback-kernel-T-equal}
In Setup~\ref{setup-FM-3}, assume further that 
that $T' = T$. 
Let 
\begin{equation*}
\cE \in 
\FM{\cC_1}{S_1}{\cC_2}{S_2}{Z_1}{Z_2}{T}
\end{equation*}
be a kernel, and 
\begin{equation*}
\cE'  = (f_2 \times f_1)^*\cE \in 
\FM{\cC_1}{S_1}{\cC_2}{S_2}{Z'_1}{Z'_2}{T} . 
\end{equation*}
Then there is an equivalence 
\begin{equation}
\label{base-change-pullbacks-2}
\Phi_{\cE'} \simeq f_2^* \circ \Phi_{\cE} \circ f_{1*} .
\end{equation} 
\end{lemma}

\begin{proof}
Factor 
$f_2 \times f_1$ as the composition 
\begin{equation*}
Z'_2 \times_T Z'_1 \xrightarrow{\, \id \times f_1 \,} Z'_2 \times_T Z_1 \xrightarrow{\, f_2 \times \id \,} Z_2 \times_T Z_1
\end{equation*}
and apply successively Lemma~\ref{lemma-pullback-kernel} parts~\eqref{pullback-kernel-1} and~\eqref{pullback-kernel-2}. 
\end{proof}

\begin{proposition}
\label{proposition-base-change-ff} 
In Setup~\ref{setup-FM-3}, assume further that for $i=1,2$ the square in 
diagram~\eqref{Zi-diagram} is cartesian.  
Let 
\begin{equation*}
\cE \in 
\FM{\cC_1}{S_1}{\cC_2}{S_2}{Z_1}{Z_2}{T} 
\end{equation*}
be a kernel, and let $\cE' = (f_2 \times f_1)^* \cE$ be its base change. 
Assume the associated functors
\begin{align*}
\Phi_{\cE} & \colon \cC_1 \otimes_{\Perf(S_1)} \Perf(Z_1) \to \cC_2 \otimes_{\Perf(S_2)} \Perf(Z_2),  \\
\Phi_{\cE'} & \colon  \cC_1 \otimes_{\Perf(S_1)} \Perf(Z'_1) \to \cC_2 \otimes_{\Perf(S_2)} \Perf(Z'_2), 
\end{align*}
admit left adjoints $\Phi_{\cE}^*$ and $\Phi_{\cE'}^*$. 
\begin{enumerate}
\item \label{base-change-ff-1}
If $\Phi_{\cE}$ is fully faithful, so is $\Phi_{\cE'}$. 
\item \label{base-change-ff-2}
If $\Phi_{\cE}^*$ is fully faithful, so is $\Phi_{\cE'}^*$. 
\end{enumerate}
\end{proposition}

\begin{proof}
For~\eqref{base-change-ff-1}, we must check that the canonical morphism 
\begin{equation*}
\Phi_{\cE'}^* \Phi_{\cE'}(C) \to C 
\end{equation*}
is an equivalence for any $C \in \cC_1 \otimes_{\Perf(S_1)} \Perf(Z'_1)$. 
Note that by Theorem~\ref{theorem-FM-functors}\eqref{fiber-product}, there is an equivalence 
\begin{equation*}
\Perf(Z'_1) \simeq \Perf(Z_1) \otimes_{\Perf(T)} \Perf(T'). 
\end{equation*} 
Hence Lemma~\ref{lemma-box-tensor-generation} implies it is enough 
to check the morphism above is an equivalence for 
\begin{equation*}
C = C_1 \boxtimes F_1 \boxtimes G,
\end{equation*} 
where $C_1 \in \cC_1, F_1 \in \Perf(Z_1), G \in \Perf(T')$. 
This follows from $T'$-linearity of $\Phi_{\cE'}$ and $\Phi_{\cE'}^*$, the 
equivalence \eqref{base-change-pullbacks}, and the equivalence 
$\Phi_{\cE'}^* \circ f_2^* \simeq f_1^* \circ \Phi_{\cE}^*$ obtained from 
\eqref{base-change-pushforwards} by taking left adjoints. 

Part~\eqref{base-change-ff-2} is proved similarly. 
\end{proof}

\begin{proposition}
\label{proposition-base-change-splitting-functor}
In Setup~\ref{setup-FM-3}, assume further that for $i=1,2$ the 
diagram~\eqref{Zi-diagram} is cartesian.   
Let 
\begin{equation*}
\cE \in 
\FM{\cC_1}{S_1}{\cC_2}{S_2}{Z_1}{Z_2}{T} 
\end{equation*}
be a kernel such that the associated functor 
\begin{equation*}
\Phi_{\cE} \colon \cC_1 \otimes_{\Perf(S_1)} \Perf(Z_1) \to \cC_2 \otimes_{\Perf(S_2)} \Perf(Z_2) 
\end{equation*}
is left splitting. 
Let $\cE' = (f_2 \times f_1)^* \cE$ be the base changed kernel, and 
assume the associated functor 
\begin{equation*}
\Phi_{\cE'} \colon  \cC_1 \otimes_{\Perf(S_1)} \Perf(Z'_1) \to \cC_2 \otimes_{\Perf(S_2)} \Perf(Z'_2) 
\end{equation*}
admits a left adjoint. 
Then $\Phi_{\cE'}$ is left splitting. 
\end{proposition}

\begin{proof}
By the same argument used to prove Proposition~\ref{proposition-base-change-ff}, 
it follows that $\Phi_{\cE'}$ satisfies the condition of 
Theorem~\ref{theorem-splitting-functors}\ref{splitting-functor-2}. 
\end{proof}


\newpage
\part{Homological projective duality} 
\label{part-II}


In this part of the paper, we develop a theory of HPD relative to a fixed base scheme $S$. 
All schemes will be defined over $S$, and all categories and functors will be $S$-linear.  
To simplify notation, 
given schemes $X$ and $Y$ over $S$ we write 
\begin{equation*}
X \times Y = X \times_S Y. 
\end{equation*} 
Similarly, given $S$-linear categories $\cC$ and $\cD$ we write 
\begin{equation*}
\cC \otimes \cD = \cC \otimes_{\Perf(S)} \cD. 
\end{equation*} 

We denote by $V$ a fixed vector bundle on our base scheme $S$. 
We write $N$ for the rank of $V$ and $H$ for the relative hyperplane 
class on $\bP(V)$ such that $\cO(H) = \cO_{\bP(V)}(1)$. 
Similarly, we write $H'$ for the relative hyperplane class on $\bP(V^{\svee})$. 
We often use the notation $\bP = \bP(V)$ and $\bPv = \bP(V^{\svee})$. 
Further, given a $\bP(V)$-linear category $\cC$ and an object $C \in \cC$, we write 
$C(H) = C \otimes \cO_{\bP(V)}(H)$.


\section{Lefschetz categories}
\label{section-lefschetz-categories}

In this section, we introduce the notion of a Lefschetz category over 
a projective space, which plays the role of an embedded 
projective variety in homological projective geometry. 
We start by studying in \S\ref{subsection-lefschetz-center}-\ref{subsection-primitive-components} 
the notion of a Lefschetz center of a linear category equipped 
with an autoequivalence. 
A Lefschetz category over $\bP(V)$ is then defined as a $\bP(V)$-linear category equipped 
with a Lefschetz center with respect to the autoequivalence 
given by tensoring with $\cO_{\bP(V)}(1)$. 
In \S\ref{subsection-lefschetz-category} we show that there are 
induced semiorthogonal decompositions of the ``linear sections'' 
of a Lefschetz category over $\bP(V)$. 
Our definitions are modeled on those of~\cite{kuznetsov-HPD}, 
which treats the geometric case. 

\subsection{Lefschetz centers}
\label{subsection-lefschetz-center}

\begin{definition}
\label{definition-lefschetz-center} 
Let $\cA$ be an $S$-linear category, and let 
\mbox{$\rT \colon \cA \to \cA$} be an $S$-linear autoequivalence. 
An admissible $S$-linear subcategory $\cA_0 \subset \cA$ is called a 
\emph{Lefschetz center} of $\cA$ with respect to $\rT$ if the subcategories 
$\cA_i \subset \cA$, $i \in \bZ$, determined by 
\begin{align}
\label{Ai-igeq0}
\cA_i & = \cA_{i-1} \cap {}^\perp(\rT^{-i}(\cA_0))\hphantom{{}^\perp},
\quad 
i \ge 1, \\ 
\label{Ai-ileq0}
\cA_i & = \cA_{i+1} \cap \hphantom{{}^\perp}(\rT^{-i}(\cA_0))^\perp,
\quad 
i \le -1,  
\end{align}
are right admissible in $\cA$ for $i \geq 1$, 
left admissible in $\cA$ for $i \leq -1$, 
vanish for all $i$ of sufficiently large absolute value, say for $|i| \geq m$, and 
provide $S$-linear semiorthogonal decompositions 
\begin{align}
\label{eq:right-decomposition}
\cA & = \langle \cA_0, \rT(\cA_1), \dots, \rT^{m-1}(\cA_{m-1}) \rangle, \\ 
\label{eq:left-decomposition}
\cA & = \langle \rT^{1-m}(\cA_{1-m}), \dots, \rT^{-1}(\cA_{-1}), \cA_0 \rangle. 
\end{align} 
Here, $\rT^i$ denotes the $i$-fold composition of $\rT$ with itself if $i \geq 1$, 
and $\rT^i = (\rT^{-1})^{-i}$ if $i \leq -1$. 

The categories $\cA_i$, $i \in \bZ$, are called the \emph{Lefschetz components} of the Lefschetz 
center $\cA_0 \subset \cA$. 
The semiorthogonal decompositions~\eqref{eq:right-decomposition} 
and~\eqref{eq:left-decomposition} are called \emph{the right Lefschetz decomposition} and 
\emph{the left Lefschetz decomposition} of $\cA$. 
The minimal $m$ above is called the \emph{length} of the Lefschetz decompositions. 
\end{definition}

Note that the subcategories $\cA_i \subset \cA_0$, $i \in \bZ,$ associated to a 
Lefschetz center form two (different in general) chains of subcategories 
\begin{equation}
\label{eq:lefschetz-chain}
0 \subset \cA_{1-m} \subset \dots \subset \cA_{-1} \subset \cA_0 \supset \cA_1 \supset \dots \supset \cA_{m-1} \supset 0.
\end{equation} 
Sometimes it is convenient to consider the following weakening of the notion of a 
Lefschetz center.  

\begin{definition}
Let $\cA$ be an $S$-linear category
with an $S$-linear autoequivalence \mbox{$\rT \colon \cA \to \cA$}. 
\begin{itemize}
\item[--] A \emph{right Lefschetz chain} in $\cA$ with respect to $\rT$ is a 
chain $\cA_0 \supset \cA_1 \supset \cdots \supset \cA_{m-1}$ of 
$S$-linear subcategories of $\cA$, such that the sequence  
\begin{equation*}
\cA_0, \rT(\cA_1), \dots, \rT^{m-1}(\cA_{m-1}) ,  
\end{equation*} 
is semiorthogonal in $\cA$. 
\item[--]  A \emph{left Lefschetz chain} in $\cA$ with respect to $\rT$ is a 
chain $\cA_{1-m} \subset \cA_{-1} \subset \cdots \subset \cA_{0}$ 
of 
$S$-linear subcategories of $\cA$, such that the sequence  
\begin{equation*}
\rT^{1-m}(\cA_{1-m}), \dots, \rT^{-1}(\cA_{-1}), \cA_0 ,
\end{equation*}
is semiorthogonal in $\cA$. 
\end{itemize}
We say that a (right or left) Lefschetz chain is \emph{full} if the above corresponding 
semiorthogonal sequence gives a semiorthogonal decomposition of $\cA$. 
\end{definition} 

Note that the right and left Lefschetz components of a Lefschetz center form full  
right and left Lefschetz chains. 
Conversely, the following shows that if a full Lefschetz chain has admissible 
components, then its biggest component is a Lefschetz center. 
This is useful in practice for constructing Lefschetz centers. 

\begin{lemma}
\label{lemma-lc-from-ld}
Let $\cA$ be an $S$-linear category with 
an $S$-linear autoequivalence \mbox{$\rT \colon \cA \to \cA$}. 
\begin{enumerate}
\item \label{lc-from-ld-1} 
Let $\cA_0 \supset \cA_1 \supset \dots \supset \cA_{m-1}$ be a full 
right Lefschetz chain in $\cA$ with respect to $\rT$. 
Then the $\cA_i$, $0 \leq i \leq m-1$, satisfy \eqref{Ai-igeq0}, and 
the categories defined by \eqref{Ai-igeq0} for $i \geq m$ vanish. 

\item \label{lc-from-ld-2}
Let $\cA_{1-m} \subset \dots \subset \cA_{-1} \subset \cA_0$ be a full left Lefschetz 
chain in $\cA$ with respect to $\rT$. 
Then the $\cA_i$, $1-m \leq i \leq 0$, satisfy \eqref{Ai-ileq0}, and the categories defined 
by \eqref{Ai-ileq0} for $i \leq -m$ vanish. 
\end{enumerate} 
Moreover, in the situation of \eqref{lc-from-ld-1} or \eqref{lc-from-ld-2}, if the components 
of the Lefschetz chain are admissible in $\cA$, then $\cA_0 \subset \cA$ is a Lefschetz 
center. Further, in this case the length of the Lefschetz decompositions of $\cA$ is given by 
\begin{equation*}
\min \set{ i \geq 0 \st \cA_{i} = 0  } = \min \set{ i \geq 0 \st \cA_{-i} = 0 }. 
\end{equation*}  
\end{lemma}

\begin{proof}
We prove the results in the case of a right Lefschetz chain; the arguments for left Lefschetz 
chains are similar. 
So let $\cA_0 \supset \dots \supset \cA_{m-1}$ be a full right 
Lefschetz chain in $\cA$. 
Then the argument of \cite[Lemmas 2.18]{kuznetsov2008lefschetz} shows that 
the $\cA_i$, $0 \leq i \leq m-1$, satisfy \eqref{Ai-igeq0}. 
To show the categories $\cA_i$, $i \geq m$, defined by \eqref{Ai-igeq0} vanish, 
it suffices to show 
\begin{equation*}
\cA_{m} = \cA_{m-1} \cap {^\perp}(\rT^{-m}(\cA_0)) 
\end{equation*}
vanishes. Equivalently, we must show 
\begin{equation*}
\rT^m(\cA_m) = \rT^{m}(\cA_{m-1}) \cap {^\perp}\cA_0
\end{equation*}
vanishes. 
First note that by applying $\rT$ to the semiorthogonal decomposition given by the 
Lefschetz chain, we obtain a semiorthogonal decomposition 
\begin{equation*}
\cA = \llangle \rT(\cA_0), \rT^2(\cA_1), \dots, \rT^{m-1}(\cA_{m-2}), \rT^m(\cA_{m-1}) \rrangle . 
\end{equation*}
Since $\cA_i \subset \cA_{i-1}$ for $1 \leq i \leq m-1$, this shows that 
\begin{equation*}
\rT^m(\cA_{m-1}) \subset {^\perp} \llangle \rT(\cA_1), \rT^2(\cA_2), \dots, \rT^{m-1}(\cA_{m-1}) \rrangle. 
\end{equation*} 
Hence 
\begin{equation*}
\rT^{m}(\cA_{m-1}) \cap {^\perp}\cA_0 \subset {^\perp} \llangle \cA_0, \rT(\cA_1), \rT^2(\cA_2), \dots, \rT^{m-1}(\cA_{m-1}) \rrangle = 0. 
\end{equation*}

Assume now that that $\cA_i \subset \cA$, $0 \leq i \leq m-1$, are admissible. 
Then the argument of \cite[Lemma 2.19]{kuznetsov2008lefschetz} shows that 
the chain of categories $\cA_{1-m} \subset \dots \subset \cA_{-1} \subset \cA_0$ 
defined recursively by the formula~\eqref{Ai-ileq0} is a full left Lefschetz chain, 
and that furthermore for any $0 \leq i \leq m-1$ we have 
\begin{equation}
\label{lc-right-left} 
\llangle \cA_{-i}, \dots, \rT^{i-1}(\cA_{-1}), \rT^{i}(\cA_0) \rrangle = 
\llangle \cA_0, \rT(\cA_1), \dots, \rT^{i}(\cA_i) \rrangle. 
\end{equation}
Hence by part~\eqref{lc-from-ld-2}, to prove $\cA_0 \subset \cA$ is a Lefschetz center it remains to show the  
categories $\cA_{-i}$, $0 \leq i \leq m-1$, are left admissible. 
By Lemma~\ref{lemma-admissible-sequence} the subcategory \eqref{lc-right-left} of 
$\cA$ is admissible, and by Lemma~\ref{lemma-admissible-sod} the category 
$\cA_{-i}$ is left admissible in~\eqref{lc-right-left}. 
Hence $\cA_{-i} \subset \cA$ is left admissible. 
Finally, the formula for the length of the Lefschetz decompositions of $\cA$ 
also follows directly from \eqref{lc-right-left}. 
\end{proof} 

\subsection{Primitive Lefschetz components} 
\label{subsection-primitive-components}
Let $\cA$ be an $S$-linear category, and 
$\cA_0 \subset \cA$ a Lefschetz center with respect to an $S$-linear autoequivalence \mbox{$\rT \colon \cA \to \cA$}. 
The Lefschetz components $\cA_i$ associated to $\cA_0 \subset \cA$ can be broken into 
smaller building blocks, as follows. 

For $i \ge 1$ the $i$-th \emph{right primitive component} $\fa_i$ of a Lefschetz center is defined as 
the right orthogonal to $\cA_{i+1}$ in $\cA_i$, i.e.  
\begin{equation*}
\fa_i = \cA_{i+1}^\perp \cap \cA_{i}. 
\end{equation*}
Note that since $\cA_{i+1}$ is right admissible in $\cA$ (and hence also in $\cA_i$), 
by Lemma~\ref{lemma-admissible-sod} we have a semiorthogonal decomposition 
\begin{equation*}
\cA_i = \llangle \fa_i, \cA_{i+1} \rrangle = \llangle \fa_i, \fa_{i+1}, \dots, \fa_{m-1} \rrangle. 
\end{equation*}
Similarly, for $i \le -1$ the $i$-th \emph{left primitive component} $\fa_i$ of a Lefschetz center is 
the left orthogonal to $\cA_{i-1}$ in $\cA_i$, i.e. 
\begin{equation*}
\fa_i = {}^\perp\cA_{i-1} \cap \cA_{i}, 
\end{equation*}
so that 
\begin{equation*}
\cA_i = \llangle \cA_{i-1}, \fa_i \rrangle = \llangle \fa_{1-m}, \dots, \fa_{i-1}, \fa_{i} \rrangle. 
\end{equation*}
For $i = 0$, we have both right and left primitive components, defined by  
\begin{equation*}
\fa_{+0} = \cA_1^{\perp} \cap \cA_0 \quad \text{and} \quad 
\fa_{-0} = {}^\perp\cA_{-1} \cap \cA_0 . 
\end{equation*}

\begin{remark}
We have 
\begin{equation*}
\fa_{-0} = \rT(\fa_{+0}). 
\end{equation*}
Indeed, in this case the equality~\eqref{lc-right-left} for $i=1$ gives 
\begin{equation*}
\llangle \cA_{-1}, \rT(\cA_0) \rrangle = \llangle \cA_0, \rT(\cA_1) \rrangle. 
\end{equation*} 
(Note that in the argument of \cite[Lemma 2.19]{kuznetsov2008lefschetz} proving~\eqref{lc-right-left} 
for $i=1$ only uses the admissibility of $\cA_0 \subset \cA$, which holds by the definition of a Lefschetz 
category.) 
Substituting $\cA_0 = \llangle \fa_{+0}, \cA_{1} \rrangle$ on the left and 
$\cA_0 = \llangle \cA_{-1}, \fa_{-0} \rrangle$ on the right, we obtain 
\begin{equation*}
\llangle \cA_{-1}, \rT(\fa_{+0}),  \rT(\cA_{1}) \rrangle = 
\llangle \cA_{-1}, \fa_{-0}, \rT(\cA_1) \rrangle, 
\end{equation*} 
so the claim follows. 
\end{remark} 

If instead of a Lefschetz center $\cA_0 \subset \cA$ we are given a 
(right or left) Lefschetz chain, then we define 
its (right or left) primitive components by the same formulas as above. 

Later we will need to consider Lefschetz centers (and chains) that satisfy the following 
finiteness assumption.  

\begin{definition}
\label{definition-strong}
Let $\cA_0 \subset \cA$ be a Lefschetz center of an $S$-linear category with 
respect to an autoequivalence $\rT \colon \cA \to \cA$. 
We say that the center is \emph{right strong} if all of the right primitive 
components $\fa_{+0}$, $\fa_i$, $i \geq 1$, are admissible in $\cA$, 
\emph{left strong} if all of the left primitive components $\fa_{-0}$, $\fa_i$, $i \leq  -1$, are 
admissible in $\cA$, and \emph{strong} if all of the primitive components are admissible. 

Similarly, a (right or left) Lefschetz chain is called \emph{strong} if its (right or left) primitive components are 
admissible. 
\end{definition}

\begin{remark}
\label{remark-strong-lc} 
By Lemma~\ref{lemma-admissible-sequence}, if a Lefschetz center $\cA_0 \subset \cA$ 
is (right or left) strong then all of its (right or left) Lefschetz components 
are automatically admissible in $\cA$. The analogous statement holds for Lefschetz chains. 
\end{remark}

\begin{remark}
\label{remark-strong-center}
If $\cA$ is smooth and proper over $S$, then by Lemma~\ref{lemma-smooth-proper-sod} 
any Lefschetz center $\cA_0 \subset \cA$ is strong. 
\end{remark}

Next we show that a Lefschetz center $\cA_0 \subset \cA$ 
also admits a decomposition into certain ``twisted'' primitive components. 
This will be an essential ingredient in \S\ref{subsection-HPD-ls} when 
we build a Lefschetz structure on the homological projective dual category. 
Let 
\begin{equation*}
\alpha_0 \colon \cA_0 \to \cA
\end{equation*} 
denote the inclusion functor. 
Since $\cA_0 \subset \cA$ is admissible, $\alpha_0$ admits a left adjoint 
$\alpha_0^*$ and right adjoint $\alpha_0^!$. 
We define the \emph{twisted right primitive components} by 
\begin{equation}
\label{fatw}
\fatw_{+0} = \alpha^*_0(\rT(\fa_{+0})) \quad \text{and} \quad 
\fatw_i = \alpha_0^*(\rT^{i+1}(\fa_i)) \text{ for } i \geq 1. 
\end{equation}
and the \emph{twisted left primitive components} by 
\begin{equation}
\label{fatw-left}
\fatw_{-0} = \alpha^!_0(\rT^{-1}(\fa_{-0})) \quad \text{and} \quad 
\fatw_i = \alpha_0^!(\rT^{i-1}(\fa_i)) \text{ for } i \leq -1. 
\end{equation}

\begin{lemma}
\label{lemma-twisted-sod-A0}
Let $\cA$ be an $S$-linear category, and $\cA_0 \subset \cA$ a Lefschetz center 
with respect to an $S$-linear autoequivalence $\rT \colon \cA \to \cA$. 
\begin{enumerate}
\item 
\label{twisted-sod-A0-right}
The functor $\alpha^*_0 \colon \cA \to \cA_0$ is fully faithful on $\rT(\fa_{+0})$ and 
on $\rT^{i+1}(\fa_i)$ for $i \geq 1$, and hence the right twisted primitive components are stable 
subcategories of $\cA_0$. 
Moreover, there is an $S$-linear semiorthogonal decomposition 
\begin{equation*}
\cA_0 = \llangle \fatw_{+0}, \fatw_1, \dots, \fatw_{m-1} \rrangle. 
\end{equation*} 
If further the Lefschetz center $\cA_0$ is right strong, then the right twisted primitive 
components are admissible subcategories of $\cA$. 

\item 
\label{twisted-sod-A0-left}
The functor $\alpha^!_0 \colon \cA \to \cA_0$ is fully faithful on $\rT(\fa_{-0})$ and 
on $\rT^{i-1}(\fa_i)$ for $i \leq -1$, and hence the left twisted primitive components are stable 
subcategories of $\cA_0$. 
Moreover, there is an $S$-linear semiorthogonal decomposition 
\begin{equation*}
\cA_0 = \llangle \fatw_{1-m}, \dots, \fatw_{-1}, \fatw_{-0} \rrangle.  
\end{equation*} 
If further the Lefschetz center $\cA_0$ is left strong, then the left twisted primitive 
components are admissible subcategories of $\cA$. 

\end{enumerate}
\end{lemma}

\begin{proof}
We prove~\eqref{twisted-sod-A0-right}; the proof of \eqref{twisted-sod-A0-left} is similar. 
Applying the autoequivalence $\rT$ to the right Lefschetz decomposition \eqref{eq:right-decomposition} 
and using $\cA_i = \llangle \fa_i, \cA_{i+1} \rrangle$, we obtain a decomposition 
\begin{equation}
\label{T-twisted-sod}
\cA = \llangle \rT(\fa_{+0}), \rT(\cA_1), \rT^2(\fa_1), \rT^2(\cA_2), \dots,  \rT^{m-1}(\fa_{m-2}), \rT^{m-1}(\cA_{m-1}), \rT^{m}(\fa_{m-1})  \rrangle. 
\end{equation}
Hence using Lemma~\ref{lemma-mutation-sod}, we find a decomposition 
\begin{multline*}
\cA = \langle \rT(\fa_{+0}), \rL_{\rT(\cA_1)}( \rT^2(\fa_1)), \dots, \rL_{\llangle \rT(\cA_1), \dots, \rT^{m-1}(\cA_{m-1}) \rrangle}(\rT^m(\fa_{m-1})), \\ 
\rT(\cA_1),  \dots, \rT^{m-1}(\cA_{m-1}) \rangle .
\end{multline*}
Comparing with~\eqref{eq:right-decomposition}, this implies there is a decomposition
\begin{equation*}
\cA_0 = \llangle \rT(\fa_{+0}), \rL_{\rT(\cA_1)}( \rT^2(\fa_1)), \dots, \rL_{\llangle \rT(\cA_1), \dots, \rT^{m-1}(\cA_{m-1}) \rrangle}(\rT^m(\fa_{m-1})) \rrangle. 
\end{equation*}
Since by~\eqref{T-twisted-sod} we have
\begin{equation*}
\rT^{i+1}(\fa_i) \subset \llangle \rT^{i+1}(\cA_i), \rT^{i+2}(\cA_{i+2}), \dots, \rT^{m-1}(\cA_{m-1}) \rrangle^{\perp}, 
\end{equation*}
it follows that $\alpha_0^*$ is given on $\rT^{i+1}(\fa_i)$ by the mutation functor 
$\rL_{\llangle \rT(\cA_1), \dots, \rT^i(\cA_i) \rrangle}$. 
Hence the fully faithful claim follows from Lemma~\ref{lemma-mutation-functors}, and 
the claimed semiorthogonal decomposition of $\cA_0$ holds. 

Finally, assume the center $\cA_0 \subset \cA$ is right strong. 
Then by Remark~\ref{remark-strong-lc} the right Lefschetz components 
$\cA_i$, $i \geq 0$, are admissible, and so by Lemma~\ref{lemma-admissible-sequence} 
the subcategory 
\begin{equation*}
{^\perp}\cA_0 = \llangle \rT(\cA_1), \dots, \rT^{m-1}(\cA_{m-1}) \rrangle 
\end{equation*}
is admissible in $\cA$. 
Hence by Lemma~\ref{lemma-adjoint-adjoints} the functor 
$\alpha_0^* \colon \cA \to \cA_0$ admits a left adjoint, and it 
tautologically admits a right adjoint. 
From this, it follows that the right twisted primitive components are admissible subcategories of $\cA_0$, 
and hence also of $\cA$. 
\end{proof}

\subsection{Lefschetz categories} 
\label{subsection-lefschetz-category}

Let $V$ be a vector bundle on $S$, 
and let $H$ denote the relative hyperplane class on the projective bundle $\bP(V)$ 
such that $\cO(H) = \cO_{\bP(V)}(1)$. 

\begin{definition}
\label{definition-lc}
A \emph{Lefschetz category} over $\bP(V)$ is a $\bP(V)$-linear 
category $\cA$ equipped with a Lefschetz center 
$\cA_0 \subset \cA$ with respect 
to the autoequivalence $- \otimes \cO_{\bP(V)}(H) \colon \cA \to \cA$. 
We say $\cA$ is (\emph{right} or \emph{left}) \emph{strong} if its center $\cA_0 \subset \cA$ is 
(right or left) strong. 
The \emph{length} of $\cA$ is the length of its Lefschetz decompositions, and 
is denoted by $\length(\cA)$. 
We say $\cA$ is \emph{moderate} if $\length(\cA) < \rank(V)$. 

If $\cA$ and $\cB$ are Lefschetz categories over $\bP(V)$, an 
\emph{equivalence of Lefschetz categories} or a \emph{Lefschetz equivalence} 
is a $\bP(V)$-linear equivalence $\cA \simeq \cB$ which induces an $S$-linear equivalence 
$\cA_0 \simeq \cB_0$ of centers. 
\end{definition}

\begin{remark}
Given a Lefschetz category $\cA$, the symbol $\cA_i$ for $i \in \bZ$ will 
be used to denote the Lefschetz component associated to the center $\cA_0 \subset \cA$, 
as in Definition~\ref{definition-lefschetz-center}. 
\end{remark}

\begin{remark} 
\label{remark-smooth-proper-strong-lc} 
Analogous to Remark~\ref{remark-strong-center}, if $\cA$ is smooth and proper over $S$, then 
any Lefschetz structure on $\cA$ is automatically strong. 
\end{remark}

\begin{remark}
\label{remark-moderate}
Moderateness of a Lefschetz category $\cA$ over $\bP(V)$ is 
a very mild condition. 
First of all, we can always embed $V$ into a larger rank vector bundle $V'$ so that 
$\cA$ is a moderate Lefschetz category over $\bP(V')$. 
Second, the Lefschetz categories that arise in practice are essentially 
always moderate. 
In fact, Corollary~\ref{corollary-lc-maximal-length} below shows that  
$\length(\cA) \leq \rank(V)$, 
so the only way for moderateness to fail is if equality holds. 
Moreover, if equality holds we show $\cA$ is closely related to $\Perf(\bP(V))$.  
\end{remark}

Here we recall a basic example of a Lefschetz category; see \cite{kuznetsov-icm} for many more 
examples. 

\begin{example} 
\label{example-projective-bundle-lc}
Let $W \subset V$ be a subbundle of rank $m$. 
The morphism $\bP(W) \to \bP(V)$ induces a $\bP(V)$-linear structure on $\Perf(\bP(W))$. 
Note that pullback along the projection $\bP(W) \to S$ gives an embedding $\Perf(S) \subset \Perf(\bP(W))$. 
The category $\Perf(\bP(W))$ is a strong Lefschetz category over $\bP(V)$ with center 
$\Perf(S)$;  
the corresponding right and left Lefschetz decompositions are 
given by the standard semiorthogonal decompositions 
\begin{align*} 
\Perf(\bP(W)) & = \llangle \Perf(S), \Perf(S)(H), \dots, \Perf(S)((m-1)H)  \rrangle , \\
\Perf(\bP(W)) & = \llangle \Perf(S)((1-m)H), \dots, \Perf(S)(-H), \Perf(S) \rrangle. 
\end{align*} 
Note that $\Perf(\bP(W))$ is a moderate Lefschetz category over $\bP(V)$ as long as $W \neq V$. 
\end{example}

Next we describe the behavior of Lefschetz categories under passage to 
linear sections. 
From now on, we simplify our notation by writing $\bP =\bP(V)$. 
Further, for $0 \leq r \leq N$, let $\bG_r = \G(r,V)$ 
be the Grassmannian of rank $r$ subbundles of $V$, and let 
$\cU_r$ be the universal rank $r$ bundle on $\bG_r$. 
Let $\bL_r = \bP_{\bG_r}(\cU_r)$ 
be the corresponding universal family of (projective) linear subspaces.  
These spaces fit into a commutative diagram
\begin{equation}
\label{diagram-Lr}
\vcenter{\xymatrix{
&&  \ar[dll]_{p_r} \bL_r \ar[drr]^{f_r} \ar[d]_{\iota_r} &&  \\ 
\bP && \ar[ll]_{\pr_1} \bP \times \bG_r  \ar[rr]^{\pr_2} && \bG_r 
}}
\end{equation}
where $\iota_r$ is a closed embedding. 

\begin{definition}
Let $\cA$ be a $\bP$-linear category. The \emph{universal family of $r$-dimensional linear sections} 
of $\cA$ is defined by 
\begin{equation*}
\bL_r(\cA) = \cA \otimes_{\Perf(\bP)} \Perf(\bL_r). 
\end{equation*}
\end{definition} 

\begin{remark}
If there exists a morphism of schemes $X \to \bP$ and a $\bP$-linear 
equivalence $\cA \simeq \Perf(X)$, then by Theorem~\ref{theorem-FM-functors}\eqref{fiber-product} 
\begin{equation*}
\bL_r(\cA) = \Perf(X) \otimes_{\Perf(\bP)} \Perf(\bL_r) \simeq 
\Perf(X \times_{\bP} \bL_r) . 
\end{equation*}
Hence, at the level of perfect complexes, the category $\bL_r(\cA)$ recovers 
the universal family of linear sections $X \times_{\bP} \bL_r$ of $X$. 
\end{remark}

\begin{remark}
\label{remark-linear-section-C}
Let $L \subset V$ be a fixed rank $r$ subbundle of $V$. 
Let $x_L \colon S \to \bG_r$ be the corresponding morphism, so that 
the pullback of $\bL_r \to \bG_r$ along $x_L$ is  $\bP(L) \to S$. 
Then by Theorem~\ref{theorem-FM-functors}\eqref{fiber-product}, base 
change along $x_L^* \colon \Perf(\bG_r) \to \Perf(S)$ gives 
\begin{equation*}
\bL_r(\cA) \otimes_{\Perf(\bG_r)} \Perf(S) 
\simeq \cA \otimes_{\Perf(\bP)} \Perf(\bP(L)). 
\end{equation*}
The category on the right should be thought of as the ``linear section'' of $\cA$ by $\bP(L)$. 
Indeed, if there exists a morphism of schemes $X \to \bP$ and a $\bP$-linear 
equivalence $\cA \simeq \Perf(X)$, then by Theorem~\ref{theorem-FM-functors}\eqref{fiber-product} 
\begin{equation*}
 \cA \otimes_{\Perf(\bP)} \Perf(\bP(L)) \simeq \Perf( X \times_{\bP} \bP(L)). 
\end{equation*}
Hence in general  
the ``fibers'' of the natural $\bG_r$-linear structure 
on $\bL_r(\cA)$ are the ``linear sections'' of $\cA$. 
\end{remark}

\begin{remark}
\label{Lr-special-values}
The category $\bL_r(\cA)$ is easy to describe for some extreme values of $r$: 
\begin{itemize}
\item[--] $\bG_0 = S$ and $\bL_0 = \varnothing$, so $\bL_0(\cA) = 0$.
\item[--] $\bG_1 = \bL_1 = \bP$, so $\bL_1(\cA) = \cA$. 
\item[--] $\bG_N = S$ and $\bL_{N} = \bP$, so $\bL_N(\cA) = \cA$. 
\end{itemize}
\end{remark}

We can also form the product of the $\bP$-linear category $\cA$ with $\bG_r$, i.e. 
$\cA \otimes \Perf(\bG_r)$. 
If $\cA$ is a Lefschetz category over $\bP$, then by Lemma~\ref{lemma-sod-base-change} 
there is a   
$\bG_r$-linear semiorthogonal decomposition
\begin{equation}
\label{lc-product-decomp}
\cA \sotimes \Perf(\bG_r) = 
\llangle \cA_0 \sotimes \Perf(\bG_r), 
\dots, \cA_{m-1}((m-1)H)  \sotimes \Perf(\bG_r) \rrangle. 
\end{equation}
We are going to show that part of this semiorthogonal decomposition 
embeds into $\bL_r(\cA)$. 
The argument works more generally for a category equipped with a 
Lefschetz chain. We begin by discussing the ``right'' case. 

Note that there is a canonical equivalence 
\begin{equation*}
\cA \sotimes \Perf(\bG_r) \simeq \cA \otimes_{\Perf(\bP)} \Perf(\bP \stimes \bG_r).
\end{equation*}
induced by pullback along the projection $\bP \stimes \bG_r \to \bG_r$. 
Via this identification, we regard the functor induced by pullback along 
$\iota_r$ as a functor 
\begin{equation}
\label{linear-section-functor}
\iota_r^* \colon \cA \sotimes \Perf(\bG_r) \to \bL_r(\cA). 
\end{equation}

\begin{lemma}
\label{lemma-linear-section}
Let $\cA$
be $\bP$-linear category 
with a right Lefschetz chain  
$\cA_0 \supset \dots \supset \cA_{m-1}$ with respect to $- \otimes \cO_{\bP}(H)$. 
Fix nonnegative integers $r$ and $s$ such that $r+s = N$. 
Then the functor~\eqref{linear-section-functor} 
is fully faithful on the subcategory
\begin{equation*}
\cA_i \sotimes \Perf(\bG_r) \subset \cA \sotimes \Perf(\bG_r)
\end{equation*}
for $s \leq i \leq m-1$, 
and the images under~\eqref{linear-section-functor} of the categories 
\begin{equation}
\label{linear-section-sos}
\cA_s(sH) \sotimes \Perf(\bG_r), 
\cA_{s+1}((s+1)H) \sotimes \Perf(\bG_r), 
\dots, \cA_{m-1}((m-1)H) \sotimes \Perf(\bG_r) 
\end{equation} 
form a semiorthogonal sequence in $\bL_r(\cA)$. 
Moreover, if the subcategories $\cA_i \subset \cA$ are (right or left) admissible, 
so are the images of \eqref{linear-section-sos} in $\bL_r(\cA)$. 
\end{lemma}

\begin{proof}
By the semiorthogonal decomposition~\eqref{lc-product-decomp}, to prove the result 
it suffices to show that if $C \in \cA_i(iH)$, $D \in \cA_j(jH)$, $s \leq i \leq j \leq m-1$, and 
$F,G$ are in the image of the pullback functor $\Perf(\bG_r) \to \Perf(\bP \stimes \bG_r)$, then 
\begin{equation*}
\cHom_{\bG_r}(\iota_r^*(D \boxtimes G), \iota_r^*(C \boxtimes F)) 
\simeq \cHom_{\bG_r}(D \boxtimes G, C \boxtimes F). 
\end{equation*}
For this, first observe that by adjunction 
\begin{equation*}
\cHom_{\bG_r}(\iota_r^*(D \boxtimes G), \iota_r^*(C \boxtimes F))  \simeq 
 \cHom_{\bG_r}(D \boxtimes G , C \boxtimes \iota_{r*}\iota_r^*F). 
\end{equation*}
Next note that
$\iota_r \colon \bL_r \hookrightarrow \bP \stimes \bG_r$ is a codimension $s$ subscheme, 
cut out by the tautological section of the vector bundle $\cQ_r(H)$, where 
$\cQ_r = V/\cU_r$ is the universal quotient bundle. 
Hence there is a Koszul resolution 
\begin{equation}
\label{koszul-res-Lr}
0 \to (\wedge^s\cQ_r^{\svee})(-sH) \to \dots \to \cQ_r^{\svee}(-H) \to \cO_{\bP \stimes \bG_r} \to \cO_{\bL_r} \to 0. 
\end{equation}
In view of the equivalence $\iota_{r*} \iota_r^*F \simeq F \otimes \cO_{\bL_r}$, 
it thus suffices to show that 
\begin{equation*}
\cHom_{\bG_r}(D \boxtimes G, C(-tH) \boxtimes (F \otimes \wedge^t\cQ^{\svee}_r))  
\end{equation*}
vanishes for $1 \leq t \leq s$. 
But $C(-tH) \in \cA_i((i-t)H) \subset \cA_{i-t}((i-t)H)$, 
so the vanishing holds by~\eqref{lc-product-decomp}.
This proves the fully faithful and semiorthogonal statements of the lemma. 
The images of the subcategories~\eqref{linear-section-sos} are
(right or left) admissible if $\cA_i \subset \cA$ are, because 
then each $\cA_i \sotimes \Perf(\bG_r) \subset \cA \sotimes \Perf(\bG_r)$ 
is (by Lemma~\ref{lemma-base-change-admissible}) and 
the functor~\eqref{linear-section-functor} 
has left and right adjoints (by Remark~\ref{remark:good-morphism} and Lemma~\ref{lemma-bc-adjoints}). 
\end{proof}

The following result was anticipated in Remark~\ref{remark-moderate}. 
\begin{corollary}
\label{corollary-lc-maximal-length} 
Let $\cA$ be $\bP$-linear category with a right Lefschetz chain $\cA_0 \supset \dots \supset \cA_{m-1}$ with respect to $- \otimes \cO_{\bP}(H)$. 
Let $N = \rank(V)$. 
\begin{enumerate}
\item \label{lc-length-1}
$\cA_{i} = 0$ for $i \geq N$. 
\item \label{lc-length-2}
The action functor $\cA \sotimes \Perf(\bP) \to \cA$ is fully 
faithful on the subcategory 
\begin{equation*}
\cA_{N-1} \sotimes \Perf(\bP) \subset \cA \sotimes \Perf(\bP). 
\end{equation*}
\item \label{lc-length-3}
If $\cA_i = \cA_0$ for $0 \leq i \leq N-1$,  
then the action functor $\cA \sotimes \Perf(\bP) \to \cA$ 
induces an equivalence of Lefschetz categories 
\begin{equation*}
\cA_0 \sotimes \Perf(\bP) \simeq \cA ,
\end{equation*}
where the Lefschetz decomposition on the left side is the 
one induced by the standard semiorthogonal decomposition 
of $\Perf(\bP)$, namely 
\begin{equation*}
\cA_0 \sotimes \Perf(\bP) = 
\llangle \cA_0 , \cA_0(H), \dots, \cA_0((N-1)H) \rrangle. 
\end{equation*}
\end{enumerate}
\end{corollary}

\begin{proof}
Taking $r = 0$ in Lemma~\ref{lemma-linear-section} gives 
\eqref{lc-length-1}, taking $r =1$ gives~\eqref{lc-length-2}, 
and then \eqref{lc-length-3} follows from the definitions.  
\end{proof}

If $\cA$ is a $\bP$-linear category 
with a right Lefschetz chain  
$\cA_0 \supset \dots \supset \cA_{m-1}$ with respect to $- \otimes \cO_{\bP}(H)$, 
we denote by $\Ku_r(\cA)$ 
the right orthogonal in $\bL_r(\cA)$ to the 
twist by $\cO(-(s-1)H)$ of the image of the sequence of categories~\eqref{linear-section-sos}. 
The twist by $\cO(-(s-1)H)$ in our definition does not affect $\Ku_r(\cA)$ 
up to equivalence, 
but this normalization will be convenient later. 
Since $\iota_r^*$ is an equivalence when restricted to any of the categories in~\eqref{linear-section-sos}, 
we will typically omit it in our notation for their images. 
Thus if the $\cA_i \subset \cA$ are right admissible, 
we have a $\bG_r$-linear semiorthogonal decomposition
\begin{equation}
\label{Cr}
\bL_r(\cA) = \llangle \Ku_r(\cA), 
\cA_s(H) \sotimes \Perf(\bG_r), 
\dots, 
\cA_{m-1}((m-s)H) \sotimes \Perf(\bG_r) \rrangle. 
\end{equation}
For $s \geq m$ the sequence~\eqref{linear-section-sos} is empty, 
so by definition $\bL_r(\cA) = \Ku_r(\cA)$ in this case. 
Further, if $m < N$ 
then $\Ku_1(\cA) = \bL_1(\cA) = \cA$. 

The $\bG_r$-linear category 
$\Ku_r(\cA)$ should be thought of as the ``interesting'' component of the 
family of linear sections $\bL_r(\cA)$, since the remaining components in~\eqref{Cr} 
come from the ambient Lefschetz chain. 
Our ultimate goal is to show that, when $\cA$ is a right strong, moderate Lefschetz category, the 
category 
$\Ku_{N-1}(\cA)$ controls all of the other $\Ku_r(\cA)$. 
Given a rank $r$ subbundle $L \subset V$, we define  
\begin{equation}
\label{CL}
\Ku_L(\cA) = \Ku_r(\cA) \otimes_{\Perf(\bG_r)} \Perf(S)
\end{equation}
as the base change along the morphism $S \to \bG_r$ classifying 
$L \subset V$. 

\begin{lemma}
\label{lemma-CL}
Let $\cA$ be $\bP$-linear category with a right Lefschetz chain 
$\cA_0 \supset \dots \supset \cA_{m-1}$ with respect to $- \otimes \cO_{\bP}(H)$, 
with each $\cA_i \subset \cA$ right admissible. 
Let $L \subset V$ be a subbundle of corank~$s$. 
Then there is a semiorthogonal decomposition 
\begin{equation*}
\cA \otimes_{\Perf(\bP)} \Perf(\bP(L)) = \llangle 
\Ku_L(\cA), 
\cA_s(H) , 
\dots, 
\cA_{m-1}((m-s)H)  \rrangle. 
\end{equation*}
\end{lemma}

\begin{proof}
Follows from Remark~\ref{remark-linear-section-C}, Lemma~\ref{lemma-sod-base-change}, and the 
decomposition~\eqref{Cr}. 
\end{proof}

\begin{remark}
Lemma~\ref{lemma-CL} says that for a range of degrees, the components 
of a Lefschetz sequence in $\cA$ embed into any linear section of $\cA$. 
This is analogous to the behavior of the cohomology of a projective variety, 
as governed by the Lefschetz hyperplane 
theorem. 
This analogy is the source of our terminology, 
which goes back to~\cite{kuznetsov-HPD}.  
\end{remark}

The following result is a combination of Lemmas~\ref{lemma-linear-section} and 
\ref{lemma-CL} in the ``left'' case, and holds by the same proof. 
\begin{lemma}
\label{lemma-left-linear-section}
Let $\cA$ be $\bP$-linear category with a left Lefschetz chain 
$\cA_{1-m} \subset \dots \subset \cA_0$ with respect to $- \otimes \cO_{\bP}(H)$. 
Fix nonnegative integers $r$ and $s$ such that \mbox{$r+s = N$}. 
\begin{enumerate}
\item The functor $\iota_r^* \colon \cA \sotimes \Perf(\bG_r) \to \bL_r(\cA)$ 
is fully faithful on the subcategory
\begin{equation*}
\cA_i \sotimes \Perf(\bG_r) \subset \cA \sotimes \Perf(\bG_r)
\end{equation*}
for $1-m \leq i \leq -s$, and the image is (right or left) admissible if $\cA_i \subset \cA$ is. 

\item If the subcategories $\cA_i \subset \cA$ are left admissible,  
then there is a $\bG_r$-linear semiorthogonal decomposition 
\begin{equation*}
\bL_{r}(\cA) = \llangle 
\cA_{1-m}((s-m)H) \sotimes \Perf(\bG_r), \dots, \cA_{-s}(-H)\sotimes \Perf(\bG_r), \Ku'_r(\cA) \rrangle. 
\end{equation*}

\item 
\label{left-linear-section-L}
If the subcategories $\cA_i \subset \cA$ are left admissible and 
$L \subset V$ is a subbundle of corank $s$, then there is a semiorthogonal decomposition 
\begin{equation*}
\cA \otimes_{\Perf(\bP)} \Perf(\bP(L)) = \llangle 
\cA_{1-m}((s-m)H), \dots, \cA_{-s}(-H), 
\Ku'_L(\cA) 
\rrangle. 
\end{equation*}
\end{enumerate}
\end{lemma}


\section{The homological projective dual of a Lefschetz category}
\label{section-HPD-category}

In this section we introduce the main object of homological projective duality, namely 
the HPD category of a Lefschetz category $\cA$ over $\bP$. 
In fact, in \S\ref{subsection-HPD-category} we define both ``right'' and ``left'' HPD categories 
$\cAd$ and $\dcA$, which are linear over $\bPv = \bP(V^{\svee})$. 
The categories $\cAd$ and $\dcA$ are actually equivalent as $\bPv$-linear categories (Lemma~\ref{lemma-HPD-admissible}), 
but it is useful to distinguish them notationally, for the following reason. 
Assuming $\cA$ is moderate, in \S\ref{subsection-HPD-ls} 
we construct a left Lefschetz chain in $\cAd$ and a right Lefschetz chain in $\dcA$. 
Under suitable hypotheses, we will later show in \S\ref{section-main-theorem} that 
these Lefschetz chains give $\cAd$ and $\dcA$ the structure of Lefschetz categories 
over $\bPv$. 
In general, there may not be a Lefschetz equivalence $\cAd \simeq \dcA$, but in \S\ref{subsection-right-vs-left-HPD} 
we show that there is under suitable finiteness conditions. 
Finally, in \S\ref{subsection-CPD} we explain how the HPD category can be thought of as a 
categorification of classical projective duality.

\subsection{The HPD category} 
\label{subsection-HPD-category} 
Note that there is an isomorphism $\bG_{N-1} \cong \bP(V^{\svee})$. 
Parallel to our notation in \S\ref{section-lefschetz-categories}, we write $\bPv =  \bP(V^{\svee})$, 
and denote by $H'$ the relative hyperplane class on $\bPv$. 
Moreover, we use the notation $\bH = \bL_{N-1}$ for the universal hyperplane in~$\bP$. 
We also drop the subscripts in our notation for the maps $p_{N-1}, f_{N-1}$, $\iota_{N-1}$, so 
that diagram~\eqref{diagram-Lr} takes the form
\begin{equation*}
\vcenter{\xymatrix{
&&  \ar[dll]_{p} \bH \ar[drr]^{f} \ar[d]_{\iota} &&  \\ 
\bP && \ar[ll]_{\pr_1} \bP \times \bPv  \ar[rr]^{\pr_2} && \bPv 
}}
\end{equation*}
Finally, given a $\bP$-linear category $\cA$, we write $\bH(\cA) = \bL_{N-1}(\cA)$ and call this category 
the \emph{universal hyperplane section} of~$\cA$. 

\begin{definition}
\label{definition-HPD}
Let $\cA$ be a Lefschetz category over $\bP$. 
The \emph{right} and \emph{left HPD categories} $\cAd$ and $\dcA$ of $\cA$ 
are the $\bPv$-linear categories defined by $\cAd = \Ku_{N-1}(\cA)$ and $\dcA = \Ku'_{N-1}(\cA)$. 
Explicitly, these categories are defined by $\bPv$-linear semiorthogonal decompositions 
\begin{align}
\label{Ad}
\bH(\cA) & = \llangle \cAd, 
\iota^*(\cA_1(H) \sotimes \Perf(\bPv)), 
\dots, 
\iota^*(\cA_{m-1}((m-1)H) \sotimes \Perf(\bPv)) \rrangle , \\ 
\label{dA} \bH(\cA) & = \llangle  
\iota^*(\cA_{1-m}((1-m)H) \sotimes \Perf(\bPv)), 
\dots, 
\iota^*(\cA_{-1}(-H) \sotimes \Perf(\bPv)),
\dcA
\rrangle , 
\end{align}
where $m = \length(\cA)$. 
\end{definition}

These categories can be described more symmetrically as follows. 

\begin{lemma} 
\label{lemma-HPD-admissible}
Let $\cA$ be a Lefschetz category over $\bP$. 
Then 
\begin{align}
\label{Ad-symmetric} 
\cAd & = \set{ C \in \bH(\cA) \st \iota_*(C) \in \cA_0 \sotimes \Perf(\bPv) },  \\ 
\dcA & = \set{ C \in \bH(\cA) \st \iota_!(C) \in \cA_0 \sotimes \Perf(\bPv) }, 
\end{align}
where $\iota_!$ denotes the left adjoint to $\iota^*$. 
Further, there are $\bPv$-linear semiorthogonal decompositions 
\begin{align}
\label{Cd-right} 
\bH(\cA) & = \llangle 
\iota^!(\cA_{1-m}((1-m)H) \sotimes \Perf(\bPv)), 
\dots, 
\iota^!(\cA_{-1}(-H) \sotimes \Perf(\bPv)),
\cAd \rrangle, \\ 
\bH(\cA) & = \llangle \dcA, 
\iota^!(\cA_1(H) \sotimes \Perf(\bPv)), 
\dots, 
\iota^!(\cA_{m-1}((m-1)H) \sotimes \Perf(\bPv)) \rrangle ,
\end{align} 
where $m = \length(\cA)$ and $\iota^!$ denotes the right adjoint to $\iota_*$. 
In particular, $\dcA$ and $\cAd$ are admissible subcategories of $\bH(\cA)$. 
Finally, the functor 
\begin{equation*}
- \otimes \cO(H) \colon \bH(\cA) \to \bH(\cA) 
\end{equation*}
induces a $\bPv$-linear equivalence $\dcA \simeq \cAd$. 
\end{lemma} 

\begin{proof}
Using the right Lefschetz decomposition of $\cA$, 
the first claim follows easily by adjunction. 
For the second claim, note that $\iota \colon \bH \to \bP \times \bPv$ is 
the embedding of the divisor cut out by the canonical section of $\cO(H + H')$ 
on $\bP \times \bPv$, and thus $\iota^!(C) = \iota^*(C) \otimes \cO(H + H')[-1]$. 
Using this, it follows from Lemma~\ref{lemma-left-linear-section} that we have a 
semiorthogonal decomposition 
\begin{equation*}
\bH(\cA) = \llangle 
\iota^!(\cA_{1-m}((1-m)H) \sotimes \Perf(\bPv)), 
\dots, 
\iota^!(\cA_{-1}(-H) \sotimes \Perf(\bPv)),
\Ku'_{N-1}(\cA)(H) \rrangle. 
\end{equation*} 
By adjunction and the left Lefschetz decomposition of $\cA$, we 
find that $\Ku'_{N-1}(\cA)(H)$ is also given by~\eqref{Ad-symmetric}. 
This also implies the final claim since $\dcA = \cK'_{N-1}(\cA)$ by definition. 
\end{proof}

\subsection{Construction of the Lefschetz sequence} 
\label{subsection-HPD-ls} 
Let $\cA$ be a moderate Lefschetz category over~$\bP$ of length $m$. 
Let $\gamma \colon \cAd \to \bH(\cA)$ denote the inclusion functor. 
Let $\gamma^* \colon \bH(\cA) \to \cAd$ denote its left adjoint, which exists 
by Lemma~\ref{lemma-admissible-sod}. 
The functor $p^* \colon \Perf(\bP) \to \Perf(\bH)$ induces a functor 
\begin{equation*}
p^*  \colon \cA \simeq \cA \otimes_{\Perf(\bP)} \Perf(\bP) \to \bH(\cA) 
\end{equation*} 
which is abusively denoted by the same symbol. 
Further, recall that by Lemma~\ref{lemma-twisted-sod-A0} there is a semiorthogonal decomposition
\begin{equation}
\label{twisted-sod-A0-repeat}
\cA_0 = \llangle \fatw_{0}, \fatw_1, \dots, \fatw_{m-1} \rrangle
\end{equation} 
into twisted primitive components, where for simplicity we write $\fatw_0$ for $\fatw_{+0}$. 
The following is the key ingredient in constructing the desired 
Lefschetz sequence in $\cAd$. 
We postpone its proof until \S\ref{subsection-proof-key-lemma} below.

\begin{lemma}
\label{lemma-lls-Cd}
Let $\cA$ be a moderate Lefschetz category over $\bP$. 
\begin{enumerate}
\item \label{A0-fully-faithful}
The functor 
\begin{equation*}
\gamma^* \circ p^* \colon \cA \to \cAd 
\end{equation*}
is fully faithful on the subcategory $\cA_0 \subset \cA$. 
\item \label{lls-orthogonality}
Let $C \in \llangle \fatw_0, \fatw_1, \dots, \fatw_i \rrangle \subset \cA_0$ 
for some $0 \leq i \leq m-1$. Then for any $D \in \cA$ and $1 \leq t \leq N-2-i$, 
we have
\begin{equation*}
\cHom_{\bPv}(\gamma^*p^*(D), \gamma^*p^*(C)(-tH')) 
\simeq 0. 
\end{equation*}
\end{enumerate}
\end{lemma} 

For $\cA$ a moderate Lefschetz category over $\bP$, 
we define 
\begin{equation*}
\cAd_0 = \gamma^*p^* (\cA_0), 
\end{equation*} 
which by 
virtue of Lemma~\ref{lemma-lls-Cd}\eqref{A0-fully-faithful} is a stable 
subcategory of $\cAd$ equivalent to $\cA_0$. 
This will be the biggest component in the promised left Lefschetz 
sequence in $\cAd$. 

By moderateness of $\cA$, the number $m$ of components 
in~\eqref{twisted-sod-A0-repeat} satisfies $m < N$. 
Hence, since $\fatw_i = \fa_i = 0$ for $i \geq m$, 
we can rewrite~\eqref{twisted-sod-A0-repeat} as 
\begin{equation*}
\cA_0 = \llangle \fatw_0, \fatw_1, \dots, \fatw_{N-2} \rrangle. 
\end{equation*}
For $2-N \leq j \leq 0$, we define 
\begin{equation*}
\fad_j = \gamma^*p^*(\fatw_{N-2+j}). 
\end{equation*}
Applying $\gamma^*p^*$ to the above semiorthogonal decomposition then gives
\begin{equation*}
\cAd_0 = \llangle  \fad_{2-N}, \dots , \fad_{-1}, \fad_0 \rrangle. 
\end{equation*}
Further, we define 
\begin{equation*}
n = N - \# \{ \, i \geq 0 \mid \cA_i = \cA_0 \, \}. 
\end{equation*}
Then it follows from the definitions that $\fad_j = 0$ for $j \leq -n$ 
and $\fad_{1-n} \neq 0$. 
Hence we have 
\begin{equation*}
\cAd_0 = \llangle  \fad_{1-n}, \dots , \fad_{-1}, \fad_0 \rrangle. 
\end{equation*} 
The $\fad_j$ determine a sequence of $S$-linear categories 
$\cAd_{1-n} \subset \dots \subset \cAd_{-1} \subset \cAd_0$ by the formula 
\begin{equation}
\label{Adj}
\cAd_j = \llangle \fad_{1-n}, \dots, \fad_{j-1}, \fad_j \rrangle. 
\end{equation}

\begin{proposition}
\label{proposition-Ad-lefschetz-sequence}
Let $\cA$ be a moderate Lefschetz category over $\bP$. 
Then 
\begin{equation*}
\cAd_{1-n} \subset \dots \subset \cAd_{-1} \subset \cAd_0
\end{equation*} 
is a left Lefschetz chain in $\cAd$ with respect to the autoequivalence $- \otimes \cO_{\bPv}(H')$, 
which is strong if $\cA$ is right strong. 
\end{proposition}

\begin{proof}
The fact that the $\cAd_j$ form a left Lefschetz chain is a direct consequence of 
Lemma~\ref{lemma-lls-Cd}\eqref{lls-orthogonality} and the definitions. 
Now suppose that $\cA$ is right strong. 
Then the categories to the right of $\cAd$ in the defining decomposition \eqref{Ad} 
are admissible by a combination of Lemma~\ref{lemma-linear-section} and Remark~\ref{remark-strong-lc}, and so 
${^\perp}\cAd$ is admissible in $\bH(\cA)$ by Lemma~\ref{lemma-admissible-sequence}. 
Hence by Lemma~\ref{lemma-adjoint-adjoints} the functor $\gamma^* \colon \bH(\cA) \to \cAd$ admits a 
left adjoint, and it tautologically admits a right adjoint. 
Further, $p^* \colon \Perf(\bP) \to \Perf(\bH)$ admits left and right adjoints. 
Since $\fatw_{i}$, $i \geq 0$, is admissible in $\cA$ by Lemma~\ref{lemma-twisted-sod-A0}, 
it follows that $\fad_j$, $j \leq 0$, is admissible in $\cAd$. 
This proves that the left Lefschetz chain formed by the $\cAd_j$ is strong. 
\end{proof}

Let us state the analogous results for the left HPD category. 
We omit the proofs as they are parallel to the right case. 
As above $\cA$ denotes a moderate Lefschetz category over $\bP$. 
Let  $\lambda \colon \dcA \to \bH(\cA)$ denote the inclusion functor, 
and let $\lambda^! \colon \bH(\cA) \to \dcA$ denote its right adjoint. 
Then, analogous to Lemma~\ref{lemma-lls-Cd}\eqref{A0-fully-faithful}, the functor 
\begin{equation*}
\lambda^! \circ p^* \colon \cA \to \dcA 
\end{equation*}
is fully faithful on the subcategory $\cA_0 \subset \cA$. 
We set 
\begin{equation*}
\dcA_0 = \lambda^! p^*(\cA_0). 
\end{equation*} 
Further, we can write the decomposition of $\cA_0$ into left primitive components as 
\begin{equation*}
\cA_0 = \llangle \fatw_{2-N}, \dots, \fatw_{-1}, \fatw_0 \rrangle, 
\end{equation*} 
where for simplicity we write $\fatw_0$ for $\fatw_{-0}$. 
For $0 \leq j \leq N-2$, we define 
\begin{equation*}
\dfa_j = \lambda^!p^*(\fatw_{2-N+j})
\end{equation*}
Then $\dfa_{j} = 0$ for $j \geq n$ where $n = N - \# \{ \, i \leq 0 \mid \cA_i = \cA_0 \, \}$, and we have 
\begin{equation*}
\dcA_0 = \llangle \dfa_0, \dfa_1, \dots, \dfa_{n-1} \rrangle. 
\end{equation*}
We define a sequence of categories $\dcA_0 \supset \cdots \supset \dcA_{n-1}$ 
by the formula 
\begin{equation}
\dcA_j = \llangle  \dfa_j, \dfa_{j+1}, \dots, \dfa_{n-1} \rrangle. 
\end{equation}
Finally, here is the left analogue of Proposition~\ref{proposition-Ad-lefschetz-sequence}. 

\begin{proposition}
\label{proposition-dA-lefschetz-sequence}
Let $\cA$ be a moderate Lefschetz category over $\bP$. 
Then 
\begin{equation*}
\dcA_0 \supset \dcA_1 \supset  \cdots \supset \dcA_{n-1}
\end{equation*}
is a right Lefschetz chain in $\dcA$ with respect to the autoequivalence 
$- \otimes \cO_{\bPv}(H')$, which is strong if $\cA$ is left strong. 
\end{proposition}

\subsection{Proof of the key lemma} 
\label{subsection-proof-key-lemma}
In this subsection, we prove Lemma~\ref{lemma-lls-Cd}. 
We will need some auxiliary results. 
Let $p_* \colon \bH(\cA) \to \cA$ denote the functor induced by $p_* \colon \Perf(\bH) \to \Perf(\bP)$. 

\begin{lemma}
\label{lemma-p-pullback}
Let $\cA$ be a $\bP$-linear category. 
Then: 
\begin{enumerate}
\item \label{lemma-p-pullback-ff}
The functor $p^* \colon \cA \to \bH(\cA)$ is fully faithful. 
\item \label{lemma-p-pullback-kill}
The functor $p_* \colon \bH(\cA) \to \cA$ kills the subcategory 
$p^*(\cA)(-tH')$
for $1 \leq t \leq N-2$. 
\end{enumerate}
\end{lemma}

\begin{proof}
The morphism $p \colon \bH \to \bP$ is the projectivization of a rank $N-1$ vector 
bundle on $\bP$. Namely, $\bH = \bP(\cK)$ where $\cK$ is the kernel of the canonical 
surjection $V^{\svee} \otimes \cO_{\bP} \to \cO_{\bP}(H)$. 
It follows that $p^* \colon \Perf(\bP) \to \Perf(\bH)$ is fully faithful, and $p_* \colon \Perf(\bH) \to \Perf(\bP)$ kills the subcategory 
\begin{equation*}
p^*(\Perf(\bP))(-tH') \subset \Perf(\bH)
\end{equation*}
for $1 \leq t \leq N-2$. From this, the lemma follows formally. 
\end{proof}

Lemma~\ref{lemma-p-pullback} shows in particular that $p^* \colon \cA \to \bH(\cA)$ embeds $\cA_{k}$ into $\bH(\cA)$ for all $k$; 
below we abusively denote the image also by $\cA_k \subset \bH(\cA)$. 
The following result controls the morphisms between objects in various twists of these categories. 

\begin{lemma}
\label{lemma-vanishing}
Let $\cA$ be a Lefschetz category over $\bP$. 
\begin{enumerate}
\item \label{lemma-vanishing-1}
The pair of categories 
\begin{equation*}
\cA_k, \, \cA_{\ell}(aH + bH')
\end{equation*} 
is semiorthogonal in $\bH(\cA)$ provided one of the following conditions 
hold: 
\begin{itemize}
\item[--] $1 \leq a \leq \ell -1$. 
\item[--] $1 \leq b \leq N-2$. 
\item[--] $a = 0$ and $b = N-1$. 
\item[--] $a = \ell$ and $b = 0$. 
\end{itemize}
\item \label{lemma-vanishing-2}
The pair of categories
\begin{equation*}
p^*(\llangle \fa'_0, \fa'_1, \dots, \fa'_i \rrangle) , \, 
\cA_{\ell}(\ell H + bH')
\end{equation*}
is semiorthogonal in $\bH(\cA)$ provided $i < \ell$. 
\end{enumerate}
\end{lemma}

\begin{proof}
Let $C, D \in \cA$. 
Note that we have equivalences 
\begin{align*}
p^*(D)(aH +bH') & \simeq \iota^*(D(aH) \boxtimes \cO(bH')),  \\
p^*(C) & \simeq \iota^*(C \boxtimes \cO), 
\end{align*}
where $\iota \colon \bH \to \bP \times_{S} \bPv$ denotes the embedding 
and $D(aH) \boxtimes \cO(bH')$ and $C \boxtimes \cO$ are regarded as objects of 
\begin{equation*}
\cA \sotimes \Perf(\bPv) \simeq \cA \otimes_{\Perf(\bP)} \Perf(\bP \stimes \bPv). 
\end{equation*}
Specializing the argument in Lemma~\ref{lemma-linear-section} to the case $r = N-1$, we thus 
obtain an exact triangle 
\begin{equation}
\label{vanishing-triangle}
\begin{aligned}
\cHom_{\bPv}(D((a+1)H) \boxtimes \cO((b+1)H'), C \boxtimes \cO) 
& \to \cHom_{\bPv}(D(aH) \boxtimes \cO(bH'), C \boxtimes \cO)  \\ 
& \to \cHom_{\bPv}(p^*(D)(aH+bH'), p^*(C)). 
\end{aligned}
\end{equation}
By Lemma~\ref{lemma-homs-base-change} 
the first two terms of this triangle can be written as 
\begin{align}
\label{vanishing-vertex-1} 
& \cHom_{S}(D((a+1)H), C) \otimes \cHom_{\bPv}(\cO((b+1)H'), \cO)  \\ 
\label{vanishing-vertex-2}
& \cHom_{S}(D(aH), C) \otimes \cHom_{\bPv}(\cO(bH'), \cO) . 
\end{align}

To prove~\eqref{lemma-vanishing-1}, assume $C \in \cA_k$ and $D \in \cA_{\ell}$. 
We must show that 
\begin{equation*}
\cHom_{\bPv}(p^*(D)(aH+bH'), p^*(C))
\end{equation*}
vanishes for $a,b$ satisfying any of the stated conditions. 
By the exact triangle~\eqref{vanishing-triangle}, it suffices to show that 
the terms~\eqref{vanishing-vertex-1} and~\eqref{vanishing-vertex-2} vanish. 
But~\eqref{vanishing-vertex-1} vanishes if either $0 \leq a  \leq \ell-1$ 
or $0 \leq b \leq N-2$, and~\eqref{vanishing-vertex-2} vanishes if either 
$1 \leq a \leq \ell$ or $1 \leq b \leq N-1$. 

To prove~\eqref{lemma-vanishing-2}, assume $C \in \llangle \fatw_0, \fatw_1, \dots, \fatw_i \rrangle$, 
$D \in \cA_{\ell}$, $a = \ell$, and $i < \ell$. 
As above, it suffices to show that~\eqref{vanishing-vertex-1} and~\eqref{vanishing-vertex-2} vanish. 
This is clear for~\eqref{vanishing-vertex-2}. 
Note that since $C \in \cA_k \subset \cA_0$, we have 
\begin{equation}
\label{vanishing-vertex-1-reduction}
\cHom_{S}(D((\ell+1)H), C) \simeq \cHom_{S}(\alpha_0^*(D((\ell+1)H)), C). 
\end{equation}
It suffices to show that this morphism space vanishes to show that~\eqref{vanishing-vertex-1} does. 
Observe that $\alpha_0^*$ kills the second term in the semiorthogonal decomposition 
\begin{equation*}
\cA_{\ell}((\ell+1)H) = \llangle \fa_{\ell}((\ell+1)H), \cA_{\ell+1}((\ell+1)H) \rrangle,
\end{equation*} 
hence $\alpha_0^*(D((\ell+1)H)) \in \fatw_{\ell}$. 
Therefore~\eqref{vanishing-vertex-1-reduction}
vanishes by the semiorthogonal decomposition~\eqref{twisted-sod-A0-repeat}.
\end{proof}

For any object $X \in \bH(\cA)$, there is an exact triangle 
\begin{equation}
\label{decomposition-bH-object}
\rR_{\cAd}(X) \to X \to \gamma \gamma^*(X).  
\end{equation}
A priori $\rR_{\cAd}(X)$ can be any object in the subcategory generated by the categories 
to the right of $\cAd$ in~\eqref{Ad}. 
The following lemma shows that for $X$ pulled back from the subcategory 
generated by a subset of the twisted primitive components of $\cA_0$, 
the object $\rR_{\cAd}(X)$ lies in a restricted subcategory. 

\begin{lemma}
\label{lemma-pullback-A0-decomposition}
Let $\cA$ be a Lefschetz category over $\bP$. 
Let $C \in \llangle \fatw_0, \fatw_1, \dots, \fatw_i \rrangle \subset \cA_0$ 
for some $0 \leq i \leq m-1$. 
Then $\rR_{\cAd}p^*(C)$ lies in the subcategory of $\bH(\cA)$ generated by the categories 
\begin{equation*}
p^*{\left( \llangle \cA_1(H), \cA_2(2H), \dots, \cA_{i-t+1}((i-t+1)H) \rrangle \right)} \otimes \cO(-tH'), \quad 
1 \leq t \leq i. 
\end{equation*}
\end{lemma}

\begin{proof}
Let $X = p^*(C)$. By definition, the functor $\gamma \gamma^*$ 
coincides with the left mutation functor through the subcategory 
\begin{equation*}
{^{\perp}}\cAd = \llangle \cA_1(H) \sotimes \Perf(\bPv), 
\dots, 
\cA_{m-1}((m-1)H) \sotimes \Perf(\bPv) \rrangle 
\subset \bH(\cA), 
\end{equation*}
so that $\rR_{\cAd}p^*(C)$ is determined by an  
exact triangle
\begin{equation*}
\rR_{\cAd}p^*(C) \to X \to \rL_{{^{\perp}}\cAd}(X). 
\end{equation*}
To prove the lemma, we factor $\rL_{{^{\perp}}\cAd}$ according to 
the above semiorthogonal decomposition of ${^{\perp}}\cAd$, and 
inductively control the corresponding exact triangle at each step. 

Namely, for $1 \leq \ell \leq m-1$, we define
\begin{equation*}
\cD_{\ell} =  \llangle \cA_{\ell}(H) \sotimes \Perf(\bPv), 
\dots, 
\cA_{m-1}((m-1)H) \sotimes \Perf(\bPv) \rrangle 
\subset \bH(\cA), 
\end{equation*}
and define $X_{\ell}$ by the exact triangle 
\begin{equation*}
X_{\ell} \to X \to \rL_{\cD_{\ell}}(X) . 
\end{equation*}
Then we claim that 
$X_{\ell}$ is contained in the subcategory of $\bH(\cA)$ generated by 
\begin{equation}
\label{induction-decomposition}
\cA_{k}(kH - tH'), ~ \ell \leq k \leq i, ~ 1 \leq t \leq i - k + 1 , 
\end{equation}
where for $\ell > i$ this subcategory is by definition $0$. 
The case $\ell = 1$ gives the statement of the lemma. 

Note that for any $\ell$, 
there is a semiorthogonal decomposition 
\begin{equation*}
\label{sod-A-tensor-P}
\cA_{\ell}(\ell H) \sotimes \Perf(\bPv) = 
\llangle \cA_{\ell}(\ell H), \cA_{\ell}(\ell H+H'), \dots, 
\cA_{\ell}(\ell H+(N-1)H') \rrangle 
\end{equation*} 
induced via base change by the standard decomposition of $\Perf(\bPv)$. 
By Lemma~\ref{lemma-vanishing}\eqref{lemma-vanishing-2} 
it thus follows that for $\ell > i$ the category $\cD_{\ell}$ is left orthogonal to $X$, 
hence $\rL_{\cD_{\ell}}(X) = X$ and $X_{\ell} = 0$. 
This proves the claim for $\ell > i$. 

For $1 \leq \ell \leq i+1$, we argue by descending induction on $\ell$. 
The base case $\ell = i+1$ was handled above. 
Assume the claim holds for $\ell$. 
By Lemma~\ref{lemma-admissible-sequence}, we have 
\begin{equation}
\label{pullback-A0-decomposition-induction}
\rL_{\cD_{\ell-1}}(X) \simeq \rL_{\cA_{\ell-1}((\ell-1)H) \sotimes \Perf(\bPv)}
 \rL_{\cD_{\ell}}(X) .
\end{equation}
It follows from Lemma~\ref{lemma-vanishing}\eqref{lemma-vanishing-1} that 
$\cA_{\ell-1}((\ell-1)H +bH')$ is left orthogonal to $X$ for $0 \leq b \leq N-2$ 
and to $X_{\ell}$ for $0 \leq b \leq N-i+\ell-3$ (here we used the induction 
assumption). 
Thus, by the exact triangle 
defining $X_{\ell}$, we find that $\cA_{\ell-1}((\ell-1)H +bH')$ is left orthogonal 
to $\rL_{\cD_{\ell}}(X)$ for $0 \leq b \leq N-i + \ell -3$. 
That is, in the decomposition of $\cA_{\ell-1}((\ell-1)H) \sotimes \Perf(\bPv)$ 
given by 
\begin{multline*}
\langle \cA_{\ell-1}((\ell-1) H - (i-\ell+2)H'), \dots, 
\cA_{\ell-1}((\ell-1) H - H'), \\
 \cA_{\ell-1}((\ell-1)H),   
\dots, \cA_{\ell-1}((\ell-1) H+(N-i+\ell-3)H') \rangle,  
\end{multline*}
the second row is left orthogonal to $\rL_{\cD_{\ell}}(X)$. 
It follows that the right side of~\eqref{pullback-A0-decomposition-induction} 
can be rewritten as
\begin{equation*}
\rL_{\llangle \cA_{\ell-1}((\ell-1) H - (i-\ell+2)H'), \dots, 
\cA_{\ell-1}((\ell-1) H - H') \rrangle} \rL_{\cD_{\ell}}(X), 
\end{equation*}
and therefore the cone of the canonical morphism 
$\rL_{\cD_{\ell}}(X) \to \rL_{\cD_{\ell-1}}(X)$
is contained in the subcategory of $\bH(\cA)$ generated by 
\begin{equation*}
\cA_{\ell-1}((\ell-1)H - tH'), ~ 1 \leq t \leq i - \ell + 2. 
\end{equation*}
By the induction assumption, the cone of $X \to \rL_{\cD_{\ell}}(X)$ 
is contained in the subcategory generated by the categories~\eqref{induction-decomposition}. 
We conclude that the cone of the composite morphism $X \to \rL_{\cD_{\ell-1}}(X)$ --- 
which agrees with $X_{\ell-1}$ up to a shift --- is contained in the 
claimed subcategory of $\bH(\cA)$. 
\end{proof}

\begin{proof}[Proof of Lemma~\ref{lemma-lls-Cd}]
Suppose $\cA$ is a moderate Lefschetz category over $\bP$. 
Then by adjunction, Lemma~\ref{lemma-lls-Cd} can be rephrased as saying that 
\begin{equation*}
p_*( \gamma \gamma^*p^*(C)(-tH') )
\end{equation*}
is canonically equivalent to: 
\begin{enumerate}
\item \label{lemma-lls-Cd-1} $C$ for $C \in \cA_0$ and $t = 0$, 
\item \label{lemma-lls-Cd-2} 
$0$ for $C \in \llangle \fatw_0, \fatw_1, \dots, \fatw_i \rrangle$, $0 \leq i \leq m-1$, and $1 \leq t \leq N-2-i$. 
\end{enumerate}

Taking $X = p^*(C)$ in~\eqref{decomposition-bH-object} and tensoring by 
$\cO(-tH')$, we get an exact triangle 
\begin{equation*}
\rR_{\cAd}(p^*(C))(-tH') \to p^*(C)(-tH') \to 
(\gamma \gamma^*p^*(C))(-tH')  .
\end{equation*}

Assume $C \in \cA_0$ and $t = 0$ as in~\eqref{lemma-lls-Cd-1}. 
Then by Lemma~\ref{lemma-pullback-A0-decomposition} the first term of this triangle lies in the subcategory of $\bH(\cA)$ generated by
$p^*(\cA)(-uH')$
for $1 \leq u \leq m-1$. 
By moderateness of $\cA$ we have $m -1 \leq N-2$, so by Lemma~\ref{lemma-p-pullback}\eqref{lemma-p-pullback-kill} 
the functor $p_* \colon \bH(\cA) \to \cA$ kills all of these categories, and thus also the first term of the above triangle. 
Further, by Lemma~\ref{lemma-p-pullback}\eqref{lemma-p-pullback-ff} the functor $p_*$ applied to the 
second term $p^*(C)$ is canonically equivalent to $C$. 
Hence applying $p_*$ to the above triangle proves~\eqref{lemma-lls-Cd-1}. 

Assume $C \in \llangle \fatw_0, \fatw_1, \dots, \fatw_i \rrangle$ and $1 \leq t \leq N-2-i$ as in~\eqref{lemma-lls-Cd-2}.  
Then arguing as above we find that $p_*$ kills the first two terms of the above triangle, and hence 
also the last, proving~\eqref{lemma-lls-Cd-2}. 
\end{proof} 

\subsection{Right versus left HPD} 
\label{subsection-right-vs-left-HPD}
In \S\ref{section-main-theorem} we will prove that the Lefschetz chains of 
Propositions~\ref{proposition-Ad-lefschetz-sequence} and~\ref{proposition-dA-lefschetz-sequence} 
are full if $\cA_0 \subset \cA$ is right or left strong, respectively. 
Hence by Lemma~\ref{lemma-lc-from-ld}, in these cases $\cAd_0 \subset \cAd$ and 
$\dcA_0 \subset \dcA$ define Lefschetz structures. 
Recall that by Lemma~\ref{lemma-HPD-admissible} there is a 
$\bPv$-linear equivalence $\dcA \simeq \cAd$. 
In this subsection, we consider the question of whether there exists an 
equivalence of Lefschetz categories. 

Our first goal is to identify the image of the Lefschetz center $\dcA_0$ under 
the equivalence $\dcA \simeq \cAd$ of Lemma~\ref{lemma-HPD-admissible}. 
This is Corollary~\ref{corollary-dA0-Ad0} below. 
The key observation is the following. 

\begin{lemma} 
\label{lemma-deltap-gammap}
Let $\cA$ be a Lefschetz category over $\bP$. 
Let $\gamma \colon \cAd \to \bH(\cA)$ and $\lambda \colon \dcA \to \bH(\cA)$ be 
the inclusion functors. 
Then there is an equivalence of functors 
\begin{equation*}
\gamma^! \circ p^! \simeq (- \otimes \omega_{\bH/\bP}) \circ \lambda^! \circ p^* . 
\end{equation*}
\end{lemma}

\begin{proof}
We need an explicit formula for $p^!$. Set  
\begin{equation*}
\cK = \ker(V^{\svee} \otimes_{\bP} \to \cO_{\bP}(H)). 
\end{equation*}
Then it is easy to see there is an isomorphism $\bH \cong \bP_{\bP}(\cK)$, under which 
$H'$ corresponds to the tautological $\cO(1)$ line bundle. 
From this, a computation shows 
\begin{equation}
\label{omegaHP}
\omega_{\bH/\bP} = \cO(H - (N-1)H')[N-2], 
\end{equation}  
and hence 
\begin{equation*}
p^! \simeq (- \otimes \cO(H - (N-1)H')[N-2]) \circ p^*. 
\end{equation*} 
Further, since by Lemma~\ref{lemma-HPD-admissible} the functor 
$(- \otimes \cO(H))$ induces an equivalence $\dcA \simeq \cAd$, it 
follows that 
\begin{equation*}
\gamma^! \simeq (- \otimes \cO(H)) \circ \lambda^! \circ (- \otimes \cO(-H)). 
\end{equation*}
Combining the above and using the $\bPv$-linearity of $\lambda^!$ proves the result. 
\end{proof}

\begin{corollary}
\label{corollary-dA0-Ad0} 
Let $\cA$ be a moderate Lefschetz category over $\bP$. 
Then the functor 
\begin{equation*}
(- \otimes \omega_{\bH/\bP}) \colon \bH(\cA) \to \bH(\cA) 
\end{equation*}
induces a $\bPv$-linear equivalence $\dcA \simeq \cAd$, which takes 
$\dcA_0 \subset \dcA$ to $\gamma^! p^!(\cA_0) \subset \cAd$. 
\end{corollary}

\begin{proof}
Combine Lemma~\ref{lemma-deltap-gammap}, 
the formula~\eqref{omegaHP}, and the last statement of 
Lemma~\ref{lemma-HPD-admissible}. 
\end{proof}

The category $\gamma^! p^!(\cA_0)$ does not coincide with $\cAd_0$. 
However, as the following lemma shows, it is closely related. 
Recall from \S\ref{subsection-splitting-functors} the notion of a splitting functor. 

\begin{lemma}
\label{lemma-im-ker-p-gamma}
Let $\cA$ be a moderate Lefschetz category over $\bP$. 
Then the functor $p_* \circ \gamma \colon \cAd \to \cA$ is splitting. 
Moreover, the images and kernels of this functor and its left and 
right adjoints are given by: 
\begin{alignat}{4}
\label{impgamma} & \im(p_* \circ \gamma) &&= \cA_0 ,  \quad && \ker (p_* \circ \gamma) && = \cAd_0{^\perp}  , \\
\label{imgammap-left} & \im(\gamma^* \circ p^*) &&= \cAd_0 , \quad && \ker(\gamma^* \circ p^*) && = {^\perp}\cA_0 , \\ \label{imgammap-right} & \im(\gamma^! \circ p^!) &&= (\cAd_0)^{\perp \perp}, \quad &&  \ker(\gamma^! \circ p^!) && = \cA_0^{\perp}. 
\end{alignat}
\end{lemma}

\begin{proof}
We start by proving~\eqref{impgamma} and~\eqref{imgammap-left}. 
First we prove $\im(p_* \circ \gamma) \subset \cA_0$. 
For $C \in \cA$ and $D \in \cAd$, we have 
\begin{equation*}
\cHom_{S}(C, p_*\gamma(D)) \simeq \cHom_{S}(p^*(C), \gamma(D)), 
\end{equation*}
which vanishes for $C \in \cA_i(iH)$, $1 \leq i \leq m-1$, by the defining 
semiorthogonal decomposition~\eqref{Ad} of $\cAd$. 
Hence 
\begin{equation*}
p_*\gamma(D) \in \llangle \cA_1(H), \dots, \cA_{m-1}((m-1)H) \rrangle^{\perp} = \cA_0,
\end{equation*} 
as desired. Since $\gamma^* \circ p^*$ is fully faithful on $\cA_0$ by 
Lemma~\ref{lemma-lls-Cd}\eqref{A0-fully-faithful}, 
it follows that in fact $\im(p_* \circ \gamma) = \cA_0$. 
From this, $\ker(\gamma^* \circ p^*) = {^\perp}\cA_0$ follows by adjunction. 
Hence $\im(\gamma^* \circ p^*)$ coincides with $\cAd_0 = (\gamma^* \circ p^*)(\cA_0)$. 
Again by adjunction, $\ker (p_* \circ \gamma) = \cAd_0{^\perp}$ follows formally from this. 

Next we aim to show $p_* \circ \gamma$ is a splitting functor. 
By Theorem~\ref{theorem-splitting-functors}, it suffices to show 
its left adjoint $\gamma^* \circ p^*$ is right splitting and its 
right adjoint $\gamma^! \circ p^!$ is left splitting. 
Note that $\ker(\gamma^* \circ p^*) = {^\perp}\cA_0$ is right admissible in $\cA$, 
$\gamma^* \circ p^*$ is fully faithful on the subcategory $\ker(\gamma^* \circ p^*)^{\perp} = \cA_0$ 
by Lemma~\ref{lemma-lls-Cd}, and $\im(\gamma^* \circ p^*) = \cAd_0$ is right admissible since $\gamma^*$ 
and $p^*$ admit right adjoints. 
Therefore $\gamma^* \circ p^*$ verifies condition \ref{splitting-functor-1-r} of Theorem~\ref{theorem-splitting-functors}, and so is right splitting. 

By Lemma~\ref{lemma-deltap-gammap} the functor $\gamma^! \circ p^!$ differs from 
$\lambda^! \circ p^*$ by the autoequivalence $- \otimes \omega_{\bH/\bP}$, and 
hence is left splitting if and only if $\lambda^! \circ p^*$ is. 
By an argument analogous to the one above for $\gamma^* \circ p^*$, we find that 
$\ker(\lambda^! \circ p^*) = \cA_0^{\perp}$ and $\lambda^! \circ p^*$ is left splitting. 
This completes the proof that $p_* \circ \gamma$ is a splitting functor. 
All that remains is to check that $\im(\gamma^! \circ p^!) = (\cAd_0)^{\perp \perp}$. 
This follows from condition \ref{splitting-functor-3-r} of Theorem~\ref{theorem-splitting-functors} and 
the equality $\ker (p_* \circ \gamma) = \cAd_0{^\perp}$. 
\end{proof}

Now we can give a criterion under which the right and left HPD categories are identified as 
Lefschetz categories. 

\begin{proposition}
\label{proposition-left-equals-right-HPD}
Let $\cA$ be a strong, moderate Lefschetz category over $\bP$, which is proper over $S$. 
Assume $\cAd$ admits a relative Serre functor $\rS_{\cAd/S}$ over $S$. 
Then there is a $\bPv$-linear equivalence $\dcA \simeq \cAd$ which 
induces an equivalence $\dcA_0 \simeq \cAd_0$. 
\end{proposition}

\begin{proof}
By Corollary~\ref{corollary-dA0-Ad0} there is a $\bPv$-linear equivalence 
$\dcA \simeq \cAd$ taking $\dcA_0$ to $\gamma^!p^!(\cA_0)$. 
By \eqref{imgammap-right} the latter coincides with $(\cAd_0)^{\perp \perp}$. 
So it suffices to show that $\cAd$ admits a $\bPv$-linear autoequivalence 
that takes $\cAd_0$ to $(\cAd_0)^{\perp \perp}$. 
By Lemmas~\ref{lemma-serre-functor-image} and \ref{lemma-serre-functor-linear}, 
the Serre functor $\rS_{\cAd/S}$ is precisely such an autoequivalence. 
\end{proof}

In Theorem~\ref{main-theorem} we show that $\dcA_0 \subset \dcA$ and $\cAd_0 \subset \cAd$ 
are Lefschetz centers as soon as $\cA$ is strong and moderate. 
Thus the conclusion of Proposition~\ref{proposition-left-equals-right-HPD} can be rephrased as saying 
there is an equivalence of Lefschetz categories $\dcA \simeq \cAd$. 

\begin{remark}
The assumption that $\cAd$ admits a relative Serre functor $\rS_{\cAd/S}$ over $S$ 
holds for instance if $\cAd$ is smooth and proper over $S$, by Lemma~\ref{lemma-serre-functor-exists}. 
This in turn holds if $\cA$ is smooth and proper over $S$, by Lemma~\ref{lemma-bHC-smooth-proper} below. 
\end{remark}

\subsection{Relation to classical projective duality} 
\label{subsection-CPD}

Recall that in Definition~\ref{definition-critical-loci} we introduced 
the notion of the critical locus of a linear category. 
\begin{definition} 
\label{definition-CPD}
Let $\cA$ be a $\bP$-linear category which is smooth and proper over $S$. 
The \emph{classical projective dual} $\CPD(\cA) \subset \bPv$ is the set 
\begin{equation*}
\CPD(\cA) = \Crit_{\bPv}(\bH(\cA)) \subset \bPv. 
\end{equation*}
\end{definition}

\begin{remark} 
Definition~\ref{definition-CPD} can also be made without the 
assumption that $\cA$ is smooth and proper over $S$. 
However, these hypotheses guarantee that our definition recovers the classical notion 
in the geometric case (Corollary~\ref{corollary-HPD-CPD} below). 
\end{remark}

\begin{lemma}
\label{lemma-crit-singular-section}
Assume $S = \Spec(k)$ where $k$ is an algebraically closed field. 
Let $X$ be an integral scheme of finite type over $k$ equipped 
with a closed immersion $X \to \bP$, such that $X$ is not contained 
in any hyperplane in $\bP$. 
Then 
\begin{equation*}
\Crit_{\bPv}(\bH(\Perf(X))) = \{ \, H \in \bPv \mid X \times_{\bP} H \text{ is singular} \, \} . 
\end{equation*}
\end{lemma}

\begin{proof}
Since $X$ is not contained in any hyperplane in $\bP$, 
the universal hyperplane section $X \times_{\bP} \bH$ of $X$ 
is flat over $\bPv$. 
So since $\bH(\Perf(X)) \simeq \Perf(X \times_{\bP} \bH)$, the 
result follows from Remark~\ref{remark-crit-geometric}. 
\end{proof}

Recall that if $X$ as in Lemma~\ref{lemma-crit-singular-section} 
is \emph{in addition smooth}, then its \emph{classical projective dual} 
is given by
\begin{equation*}
X^{\svee} = \{ \, H \in \bPv \mid X \times_{\bP} H \text{ is singular} \, \} . 
\end{equation*}
Hence we have: 

\begin{corollary}
\label{corollary-HPD-CPD} 
Assume $S = \Spec(k)$ where $k$ is an algebraically closed field. 
Let $X$ be a smooth integral scheme of finite type over $k$ equipped 
with a closed immersion $X \to \bP$, such that $X$ is not contained 
in any hyperplane in $\bP$. 
Then 
\begin{equation*}
\CPD(\Perf(X)) = X^{\svee} \subset \bPv. 
\end{equation*}
\end{corollary}

Classical projective typically does not preserve smoothness of a variety. 
An interesting feature of HPD is that it does: 
\begin{lemma}
\label{lemma-bHC-smooth-proper} 
Let $\cA$ be a Lefschetz category over $\bP$, which is smooth and proper over $S$. 
Then the universal hyperplane section category $\bH(\cA)$ and 
the HPD categories $\dcA$ and $\cAd$ are all 
smooth and proper over $S$. 
\end{lemma}

\begin{proof}
The claim for $\bH(\cA)$ follows from Lemma~\ref{lemma-base-change-smooth-proper}, 
since the morphism $\bH \to \bP$ is smooth and proper. 
Then the claim for the HPD categories follows from 
Lemma~\ref{lemma-smooth-proper-sod}\eqref{smooth-proper-sod}. 
\end{proof}

\begin{proposition}
\label{proposition-CPD-HPD}
Let $\cA$ be a Lefschetz category over $\bP$, 
which is smooth and proper over~$S$. 
Then 
\begin{equation*}
\CPD(\cA) = \Crit_{\bPv}(\cAd) = \Crit_{\bPv}(\dcA) \subset \bPv. 
\end{equation*} 
\end{proposition}

\begin{proof}
The equality $\Crit_{\bPv}(\cAd) = \Crit_{\bPv}(\dcA)$ holds because 
by Lemma~\ref{lemma-HPD-admissible} there is a $\bPv$-linear equivalence $\cAd \simeq \dcA$. 
By Lemmas~\ref{lemma-smooth-proper-sod}\eqref{smooth-proper-sod} and~\ref{lemma-bHC-smooth-proper} 
the components of the semiorthogonal decomposition~\eqref{Ad} are admissible. 
Hence by Lemma~\ref{lemma-crit-sod} we have
\begin{equation*}
\CPD(\cA) = \Crit_{\bPv}(\cAd) \cup \bigcup_{i=1}^{m} \Crit_{\bPv} {\left( \cA_i(iH) \sotimes \Perf(\bPv) \right)} . 
\end{equation*}
But Lemma~\ref{lemma-smooth-proper-sod} also implies that the 
components $\cA_i(iH)$ of the Lefschetz decomposition of $\cA$ 
are smooth and proper over $S$. 
Hence by Lemma~\ref{lemma-smoothness-properness-base-change} their base changes  
\begin{equation*}
\cA_i(iH) \sotimes \Perf(\bPv)
\end{equation*} 
are smooth and proper over $\bPv$. 
Hence 
\begin{equation*}
\Crit_{\bPv}{\left( \cA_i(iH) \sotimes \Perf(\bPv) \right)} = \varnothing
\end{equation*} 
by Lemma~\ref{lemma-crit-smooth}. 
\end{proof}


\section{The main theorem of HPD}
\label{section-main-theorem}

Let $\cA$ be a moderate Lefschetz category over $\bP$. 
In \S\ref{section-HPD-category} we constructed 
$\bPv$-linear right and left HPD categories $\cAd$ and $\dcA$, 
which are equipped with natural Lefschetz chains. 
In this section, we show that if $\cA$ is right strong then this  
gives $\cAd$ the structure of a left strong Lefschetz category over $\bPv$, 
and if $\cA$ is left strong this gives $\dcA$ the structure of a right 
strong Lefschetz category over $\bPv$. 
Further, we prove that taking the HPD once more recovers $\cA$, i.e. 
we prove there are Lefschetz equivalences 
${^\hpd}(\cAd) \simeq \cA$ and $(\dcA)^{\hpd} \simeq \cA$. 

In fact, we will deduce these statements from a significantly stronger result
(Theorem~\ref{key-theorem}), 
which describes the series of categories $\Ku_r(\cA)$ 
defined by~\eqref{Cr} in terms of 
an analogous series of categories associated to $\cAd$. 

Our arguments generalize those of~\cite[\S6]{kuznetsov-HPD}, by 
placing them in the framework set up in the previous sections. 

\subsection{The main theorem and its corollaries}
Throughout this section, 
we work in the setup where $\cA$ is assumed right 
strong, and then we work with the right HPD category $\cAd$. 
The reader may check that all of the analogous results hold with 
``right'' replaced by ``left''. 

\begin{setup} 
\label{setup-HPD}
\mbox{}
\begin{enumerate}
\item 
\label{setup-A} 
$\cA$ is a right strong, moderate Lefschetz category over $\bP$ of length $m$. 
\item 
\label{setup-D}
$\cB$ is a $\bPv$-linear category and 
$\phi \colon \cB \to \cAd$ is a $\bPv$-linear functor such that: 
\begin{enumerate}
\item 
\label{setup-fully-faithful}
$\phi$ is fully faithful and admits  
left and right adjoints $\phi^*$ and $\phi^!$. 
\item 
\label{setup-Bj}
the categories $\cB_j = \phi^* \cAd_j$, $j \leq 0$, where $\cAd_j$ is as in~\eqref{Adj}, 
form a strong left Lefschetz chain 
\begin{equation*}
\cB_{1-n} \subset \dots \subset \cB_{-1} \subset \cB_0
\end{equation*} 
in $\cB$ with respect to $- \otimes \cO_{\bPv}(H')$,
\item 
\label{setup-image-A0} 
the image of the composition $p_* \circ \gamma \circ \phi \colon  \cB \to \cA$ is 
$\cA_0$, where $\gamma \colon \cAd \to \bH(\cA)$ is the inclusion functor. 
\end{enumerate}
\end{enumerate}
\end{setup}

\begin{remark}
\label{remark-HPD-criteria}
The category $\cB = \cAd$ satisfies the assumptions with $\phi = \id_{\cAd}$. 
Indeed, \eqref{setup-fully-faithful} is automatic, 
\eqref{setup-Bj} holds by Proposition~\ref{proposition-Ad-lefschetz-sequence}, 
and \eqref{setup-image-A0} holds by Proposition~\ref{lemma-im-ker-p-gamma}. 
In fact, we will see below that in Setup~\ref{setup-HPD}, the functor $\phi$ is 
automatically an equivalence. 
This gives useful criteria for checking that a category $\cB$ is equivalent to the
HPD category $\cAd$. 
\end{remark}

\makeatletter
\newcommand{\mylabel}[2]{#2\def\@currentlabel{#2}\label{#1}}
\makeatother

\begin{remark}
\label{remark-2c} 
Consider the following criterion: 
\begin{enumerate}[leftmargin=2.25pc]
\item[\mylabel{2cn}{(2c$'$)}] 
the functor $\phi^* \circ \gamma^* \circ p^* \colon \cA \to \cB$ is fully faithful on the right twisted components 
$\fatw_i$, $i \geq 0$, of $\cA$, and their images give a semiorthogonal decomposition 
\begin{equation*}
\cB_j = \llangle 
\phi^* \gamma^*p^*(\fatw_0), \phi^* \gamma^*p^*(\fatw_1), \dots, \phi^* \gamma^*p^*(\fatw_{N-2+j})  
\rrangle 
\end{equation*} 
for all $j \leq 0$. 
\end{enumerate} 
In Lemma~\ref{lemma-2c}, we will see that conditions \eqref{setup-A}, \eqref{setup-fully-faithful}, \eqref{setup-Bj}, and~\ref{2cn} 
together imply \eqref{setup-image-A0}. 
This is useful because the criterion~\ref{2cn} is sometimes more convenient to check than \eqref{setup-image-A0}, 
see for instance the proof of \cite[Theorem 9.5]{joins}. 
\end{remark}

Let $\bGv_s = \G(s,V^{\svee})$ be the Grassmannian of rank $s$ 
subbundles of $V^{\svee}$, and let $\bLv_s$ be the corresponding projectivized 
universal family. 
Then Lemma~\ref{lemma-left-linear-section} (with $\bP$ replaced with the dual $\bPv$) 
applies to $\cB$, and we get $\bGv_s$-linear categories $\Ku_s(\cB)$ 
for $0 \leq s \leq N$, 
characterized by semiorthogonal decompositions 
\begin{equation}
\label{Ds}
\bLv_{s}(\cB) = \llangle 
\cB_{1-n}((r-n)H') \sotimes \Perf(\bGv_s),
\dots , \cB_{-r}(-H')\sotimes \Perf(\bGv_s), 
\Ku'_s(\cB) 
\rrangle 
\end{equation}
where $r = N - s$. 
There is an identification $\bG_r \cong \bGv_s$, by which we regard 
$\Ku_s(\cB)$ as a $\bG_r$-linear category. 
We have canonical morphisms 
\begin{equation*}
\vcenter{\xymatrix{
& \ar[dl]_{q_s} \bLv_s  \ar[dr]^{g_s} & \\ 
\bPv & & \bG_r . 
}} 
\end{equation*}

We aim to prove that there is a $\bG_r$-linear equivalence $\Ku'_s(\cB) \simeq \Ku_r(\cA)$. 
By definition we have $\Ku_r(\cA) \subset \bL_r(\cA)$ and $\Ku'_s(\cB) \subset \bLv_s(\cB)$. 
The desired equivalence will be induced by a functor $\bLv_s(\cB) \to \bL_r(\cA)$, 
which can be described using the kernel formalism of \S\ref{section-kernels} as follows. 
Consider the $\bPv$-linear composition $\gamma \circ \phi \colon \cB \to \bH(\cA)$, 
where recall $\bH = \bL_{N-1}$ and $\gamma \colon \cAd \to \bH(\cA)$ is the inclusion. 
By Proposition~\ref{proposition-kernels}, there is a kernel 
\begin{equation*}
\cE \in \FM{\cB}{\bPv}{\cA}{\bP}{\bPv}{\bH}{\bPv}
\end{equation*}
such that 
\begin{equation}
\label{Phigammaphi} 
\Phi_{\cE} = \gamma \circ \phi. 
\end{equation}
Let 
\begin{equation*}
\zeta_r \colon \bL_r \times_{\bG_r} \bLv_{s}   \to \bH 
\end{equation*}
be the natural morphism. 
Set 
\begin{equation*}
\cE_r = \zeta_r^*\cE \in \FM{\cB}{\bPv}{\cA}{\bP}{\bLv_{s}}{\bL_r}{\bG_r}
\end{equation*}
and define 
\begin{equation*}
\Phi_r = \Phi_{\cE_r} \colon \bLv_s(\cB) \to \bL_r(\cA) 
\end{equation*}
to be the associated $\bG_r$-linear functor. 
This is the functor we are after. 
Note that $\Phi_{\cE} = \Phi_{N-1}$ under the identifications $\bPv = \bLv_1$, $\bH = \bL_{N-1}$. 

The key result of this section is the following. 
\begin{theorem}
\label{key-theorem}
In Setup~\ref{setup-HPD}, for all nonnegative integers $r$ and $s$ such that $r+s = N$, 
there is a $\bG_r$-linear equivalence 
\begin{equation*}
\phi_r \colon \Ku'_s(\cB) \xrightarrow{\sim} \Ku_r(\cA) 
\end{equation*}
induced by the restriction of $\Phi_r$ to $\Ku'_s(\cB)$. 
\end{theorem}

The proof of Theorem~\ref{key-theorem} will be given below. 
Here we derive some consequences. 

\begin{corollary}
\label{corollary-Ad-B}
In Setup~\ref{setup-HPD}, the functor $\phi \colon \cB \to \cAd$ is an equivalence. 
\end{corollary}

\begin{proof}
By construction, the functor $\phi \colon \cB \to \cAd$ coincides with the equivalence  
$\phi_{N-1}$ of Theorem~\ref{key-theorem} 
(note that $\cK'_1(\cB) = \bLv_1(\cB) = \cB$). 
\end{proof}

\begin{corollary}
\label{corollary-B-lc-full}
In Setup~\ref{setup-HPD}, 
the Lefschetz chain  
$\cB_{1-n} \subset \dots \subset \cB_{-1} \subset \cB_0$ in $\cB$ is full.
\end{corollary}

\begin{proof}
The defining semiorthogonal decomposition~\eqref{Ds} 
of $\Ku'_s(\cB)$ for $s=N$ can be written as 
\begin{equation*}
\cB = \llangle \cB_{1-n}(-nH'), \dots, \cB_{-1}(-2H'), \cB_0(-H'), \Ku'_N(\cB) \rrangle. 
\end{equation*}
Twisting by $\cO(H')$, we get 
\begin{equation*}
\cB = \llangle \cB_{1-n}((1-n)H'), \dots, \cB_{-1}(-H'), \cB_0, \Ku'_N(\cB)(H') \rrangle. 
\end{equation*}
Hence the Lefschetz chain  
$\cB_{1-n} \subset \dots \subset \cB_{-1} \subset \cB_0$ in $\cB$ is full if and only if $\Ku'_N(\cB) \simeq 0$. 
But by Theorem~\ref{key-theorem} we have $\Ku'_N(\cB) \simeq \Ku_0(\cA)$, 
and $\Ku_0(\cA) \simeq 0$ by Remark~\ref{Lr-special-values}.  
\end{proof}

By combining the above results in the case $\cB = \cAd$, 
we obtain the main theorem of HPD. 
\begin{theorem}
\label{main-theorem}
Let $\cA$ be a right strong, moderate Lefschetz category over $\bP$. 
Then: 
\begin{enumerate}
\item \label{Ad-lc} 
$\cAd$ is a left strong, moderate Lefschetz category over $\bPv$, with 
components $\cAd_j \subset \cAd$ given by~\eqref{Adj} and length given by 
\begin{equation*}
\length(\cAd) = N - \# \{ \, i \geq 0 \mid \cA_i = \cA_0 \, \}. 
\end{equation*} 

\item \label{Ad-A-linear-section-2}
Let $L \subset V$ be a subbundle and 
let $L^{\perp} = \ker(V^{\svee} \to L^{\svee})$ 
be its orthogonal.  
Set 
\begin{equation*}
r = \rank(L), \quad s = \rank(L^{\perp}), \quad m = \length(\cA), \quad n = \length(\cAd) .
\end{equation*} 
Then there are semiorthogonal decompositions 
\begin{align*}
\cA \otimes_{\Perf(\bP)} \Perf(\bP(L)) & = \llangle \Ku_L(\cA), 
\cA_s(H) , 
\dots, 
\cA_{m-1}((m-s)H)  \rrangle , \\
\cAd \otimes_{\Perf(\bPv)} \Perf(\bP(L^{\perp})) & = 
\llangle 
\cAd_{1-n}((r-n)H'), \dots, \cAd_{-r}(-H'), 
\Ku'_{L^{\perp}}(\cAd) 
\rrangle , 
\end{align*}
and an $S$-linear equivalence $\Ku_L(\cA) \simeq \Ku'_{L^{\perp}}(\cAd)$. 
\end{enumerate}
\end{theorem}

\begin{proof}
The categories $\cAd_j$ form a strong left Lefschetz chain in $\cAd$ by Proposition~\ref{proposition-Ad-lefschetz-sequence}. 
Hence they are the components of a left strong Lefschetz structure over 
$\bPv$ by Corollary~\ref{corollary-B-lc-full} combined with Lemma~\ref{lemma-lc-from-ld}. 
By construction, the length of $\cAd$ with this Lefschetz structure is as stated. 
This proves \eqref{Ad-lc}. 
For part~\eqref{Ad-A-linear-section-2}, note that there are semiorthogonal decompositions of the 
claimed form by Lemma~\ref{lemma-CL} and Lemma~\ref{lemma-left-linear-section}\eqref{left-linear-section-L}. 
The equivalence $\Ku_L(\cA) \simeq \Ku'_{L^{\perp}}(\cAd)$ is the base change 
of the equivalence $\Ku_r(\cA) \simeq \Ku'_s(\cAd)$ given by  Theorem~\ref{key-theorem}
along the morphism $S \to \bG_r$ classifying $L \subset V$. 
\end{proof}

\begin{remark}
\label{remark-geometric-linear-section}
Specializing Theorem~\ref{main-theorem} to the case where $S$ is the spectrum of a field and 
$\cA$ and $\cAd$ are geometric, i.e. of the form $\cA = \Perf(X)$ and $\cAd = \Perf(Y)$ for a $\bP$-scheme 
$X$ and a $\bPv$-scheme $Y$, we recover Theorem~\ref{theorem-kuznetsov-HPD} from~\S\ref{section-intro}.  
In fact, even in this case our result is more general than Theorem~\ref{theorem-kuznetsov-HPD} in 
several respects.
First, we do not impose any smooth and properness assumptions. 
Second, we do not need any transversality assumption about the fiber products 
$X \times_{\bP} \bP(L)$ and \mbox{$Y \times_{\bPv} \bP(L^{\perp})$}, 
which are taken in the derived sense, according to our conventions from \S\ref{subsection-DAG}. 
This addresses the question of \cite[Remark 6.26]{kuznetsov-HPD}, i.e. 
removes the ``admissibility'' assumption on $L \subset V$ from \cite[Theorem 6.3]{kuznetsov-HPD}. 
Finally, we note that by applying Theorem~\ref{theorem-dual-Db} and Proposition~\ref{proposition-sod-Db} 
we obtain a version of Theorem~\ref{main-theorem} for bounded derived categories of coherent sheaves, 
recovering the result of Remark~\ref{remark-Db-HPD}. 
\end{remark}

Using the above, we can prove that HPD is indeed a duality. 

\begin{theorem}
Let $\cA$ be a right strong, moderate Lefschetz category over $\bP$. 
Then the functor $\phi_{1} \colon {^\hpd}(\cAd) \xrightarrow{\sim} \cA$ 
is an equivalence of Lefschetz categories over $\bP$. 
\end{theorem}

\begin{proof}
First note that $\cK'_{N-1}(\cAd) = {^\hpd}(\cAd)$ by 
definition and $\cK_1(\cA) = \cA$, so that $\phi_1$ is indeed a $\bP$-linear 
equivalence between the stated categories. 
It remains to check that this equivalence takes ${^\hpd}(\cAd)_0$ to $\cA_0$. 
For this, let us describe ${^\hpd}(\cAd)_0 \subset {^\hpd}(\cAd)$ more explicitly. 
The universal hyperplane $\bH$ in $\bP$ is simultaneously the universal hyperplane 
in $\bPv$, via the projection $f \colon \bH \to \bPv$. 
The category ${^\hpd}(\cAd)$ is defined by the $\bP$-linear semiorthogonal decomposition 
\begin{equation}
\label{HcAd}
\bH(\cAd) =  
 \llangle  
\iota^*(\cAd_{1-n}((1-n)H') \sotimes \Perf(\bP))), 
\dots, 
\iota^*(\cAd_{-1}(-H') \sotimes \Perf(\bP)),
{^\hpd}(\cAd) 
\rrangle , 
\end{equation} 
where $\bH(\cAd) = \cAd \otimes_{\Perf(\bPv)} \Perf(\bH)$, see~\eqref{dA}. 
Further, by definition 
\begin{equation*}
{^\hpd}(\cAd)_0 = \lambda^! f^*(\cAd_0), 
\end{equation*}
where $\lambda \colon {^\hpd}(\cAd)  \to \bH(\cAd)$ is the inclusion functor. 
By Propositions~\ref{proposition-LvsD-generation} and \ref{proposition-Phi-splitting} 
(with \mbox{$r =1$} and $s = N-1$) proved below, we see that the functor $\Phi_1 \colon \bH(\cAd) \to \cA$ 
kills the categories to the left of ${^\hpd}(\cAd)$ in~\eqref{HcAd}. 
It follows that  
\begin{equation}
\label{phi1Phi1A}
\phi_1({^\hpd}(\cAd)_0) = \Phi_1(\lambda\lambda^! f^*(\cAd_0)) = \Phi_1(f^*(\cAd_0)) . 
\end{equation}
We claim there is an equivalence of functors 
\begin{equation*}
\Phi_1 \circ f^* \simeq p_* \circ \gamma \colon \cAd \to \cA . 
\end{equation*} 
Indeed, there is a tautological commutative diagram 
\begin{equation*}
\xymatrix{
\bP \times_{\bP} \bH \ar[r]^-{\zeta_1}_-{\sim} \ar[d]_{\id \times f} & \bH \times_{\bPv} \bPv \ar[dl]^{p \times \id} \\
\bP \times \bPv
}
\end{equation*}
Hence there is an equivalence of kernels $(\id \times f)_*(\cE_1) \simeq (p \times \id)_*(\cE)$, 
where recall $\cE_1 = \zeta_1^*(\cE)$ is the kernel for the functor $\Phi_1$. 
So the above claim follows from Lemma~\ref{lemma-pushforward-kernel}, since $\cE$ is the kernel for 
the functor $\gamma$. 
Combining the claim with \eqref{phi1Phi1A} and \eqref{impgamma}, 
we conclude $\phi_1$ takes ${^\hpd}(\cAd)_0$ to $\cA_0$.  
\end{proof}

\subsection{Notation}
Throughout, $r$ and $s$ will denote nonnegative integers such that $r+s = N$.  
We will prove Theorem~\ref{key-theorem} by an induction argument. 
For this, we need to relate the functors
\begin{align*}
\Phi_r & \colon \bLv_s(\cB) \to \bL_r(\cA) ,  \\ 
\Phi_{r-1} & \colon \bLv_{s+1}(\cB) \to \bL_{r-1}(\cA) . 
\end{align*}
Here and below, when we consider the functor $\Phi_{r-1}$ or the 
categories $\bLv_{s+1}(\cB)$ and $\bL_{r-1}(\cA)$, we implicitly assume $r \geq 1$. 

To this end, we introduce some auxiliary spaces. 
Let $\Fl{r-1}{r}$ denote the variety of flags of subspaces $L_{r-1} \subset L_r \subset V$ 
where $\dim L_i = i$. This variety comes with forgetful maps 
\begin{equation}
\label{Fl}
\vcenter{
\xymatrix{ 
& \Fl{r-1}{r} \ar[dl]_{\pi_{r-1}} \ar[dr]^{\pi_r}  & \\
\bG_{r-1} & & \bG_r
} 
}
\end{equation}
Further, define 
\begin{align*}
\bLp_{r-1} & = \bL_{r-1} \times_{\bG_{r-1}} \Fl{r-1}{r}  , \\
\bLm_{r} & = \bL_r \times_{\bG_r} \Fl{r-1}{r}  .
\end{align*}
These spaces fit into a commutative diagram 
\begin{equation}
\vcenter{
\xymatrix{
\Fl{r-1}{r} \ar[d]_{\pi_{r-1}} & \bLp_{r-1} \ar[l]_{ \fp_{r-1}} \ar[r]^{a_r} \ar[d]^{\pip_{r-1}} & \bLm_{r} 
\ar[d]_{\pim_r} \ar[r]^{\fm_r} & \Fl{r-1}{r} \ar[d]^{\pi_r} \\ 
\bG_{r-1} & \bL_{r-1} \ar[l]_{f_{r-1}} & \bL_{r} \ar[r]^{f_r} & \bG_r \\ 
}
}
\end{equation}
where the squares are cartesian and $a_r$ is a closed embedding. 
We also introduce the notation 
\begin{align*}
\alpha_{r} = \pim_{r} \circ a_r & \colon \bLp_{r-1} \to \bL_{r} , \\ 
\pplus_{r-1} = p_{r-1} \circ \pip_{r-1} & \colon \bLp_{r-1} \to \bP , \\
\pminus_r  = p_r \circ \pim_r & \colon \bLm_r \to \bP, 
\end{align*}
where recall $p_r \colon \bL_r \to \bP$ denotes the projection. 

Dually, let $\Flv{s}{s+1}$ denote the the variety of flags of subspaces $M_{s} \subset M_{s+1} \subset V^{\svee}$, 
where $\dim M_i = i$. This variety comes with forgetful maps 
\begin{equation}
\label{Flv}
\vcenter{
\xymatrix{ 
& \Flv{s}{s+1} \ar[dl]_{\piv_{s+1}} \ar[dr]^{\piv_{s}}  & \\
\bGv_{s+1} & & \bGv_{s}
} }
\end{equation}
As above, we define 
\begin{align*}
\bLvm_{s+1} & = \bLv_{s+1} \times_{\bGv_{s+1}} \Flv{s}{s+1} , \\
\bLvp_{s} & = \bLv_{s} \times_{\bGv_{s}} \Flv{s}{s+1} .
\end{align*}
Under the canonical isomorphisms $\Flv{s}{s+1} \cong \Fl{r-1}{r}$, $\bGv_{s+1} \cong \bG_{r-1}$, 
and $\bGv_{s} \cong \bG_{r}$, the diagram~\eqref{Flv} is identified with~\eqref{Fl}. 
Hence the above spaces fit into a commutative diagram 
\begin{equation}
\vcenter{
\xymatrix{
\Fl{r-1}{r} \ar[d]_{\pi_{r-1}} & \bLvm_{s+1} \ar[l]_{\gm_{s+1}} \ar[d]^{\pivm_{s+1}} & \bLvp_{s} \ar[l]_{b_{s+1}}
\ar[d]_{\pivp_s} \ar[r]^{\gp_s} & \Fl{r-1}{r} \ar[d]^{\pi_r} \\ 
\bG_{r-1} & \bLv_{s+1} \ar[l]_{g_{s+1}} & \bLv_{s} \ar[r]^{g_s} & \bG_r
}
}
\end{equation}
As above, we also set 
\begin{align*}
\beta_{s+1} = \pivm_{s+1} \circ b_{s+1} & \colon \bLvp_s \to \bLv_{s+1} , \\
\qminus_{s+1} = q_{s+1} \circ \pivm_{s+1} & \colon \bLvm_{s+1} \to \bPv , \\
\qplus_s  = q_s \circ \pivp_s & \colon \bLvp_s \to \bPv, 
\end{align*}
where recall $q_s \colon \bLv_s \to \bPv$ denotes the projection. 

We have morphisms 
\begin{align*}
\pim_r \times \pivp_s & \colon \bLm_r \times_{\Fl{r-1}{r}} \bLvp_s \to \bL_r \times_{\bG_r} \bLv_s , \\
\pip_{r-1} \times \pivm_{s+1} & \colon \bLp_{r-1} \times_{\Fl{r-1}{r}} \bLvm_{s+1} \to \bL_{r-1} \times_{\bG_{r-1}} \bLv_{s+1}. 
\end{align*}
Pulling back the kernel $\cE_r$ for $\Phi_r$ along these morphisms, 
we obtain kernels 
\begin{align*}
\cEm_{r} & = (\pim_r \times \pivp_s)^*\cE_r \in 
\FM{\cB}{\bPv}{\cA}{\bP}{\bLvp_s}{\bLm_r}{\Fl{r-1}{r}}, \\ 
\cEp_{r-1} & = (\pip_{r-1} \times \pivm_{s+1})^*\cE_{r-1} \in 
\FM{\cB}{\bPv}{\cA}{\bP}{\bLvm_{s+1}}{\bLp_{r-1}}{\Fl{r-1}{r}}.
\end{align*}
We denote by 
\begin{align*}
\Phim_{r} & \colon \bLvp_s(\cB) \to \bLm_r(\cA) ,  \\
\Phip_{r-1} & \colon \bLvm_{s+1}(\cB) \to \bLp_{r-1}(\cA) ,
\end{align*}
the associated $\Fl{r-1}{r}$-linear functors. 

Finally, we denote by $\cU_r$ the rank $r$ tautological subbundle of $V \otimes \cO$ on $\bG_r$, 
and by $\cW_s$ the rank $s$ tautological subbundle of $V^{\svee} \otimes \cO$ on $\bGv_s$. 
By abuse of notation, we use the same symbol to denote the pullback of $\cU_r$ or $\cW_s$ to 
any space mapping to $\bG_r$ or $\bGv_s$. 
Note that under the isomorphism $\bG_r \cong \bGv_s$, the bundle $\cW_s$ corresponds to the 
orthogonal bundle $\cU_r^{\perp}$. 

\subsection{Geometric lemmas}
Here we gather some results describing the geometry of the spaces introduced above. 
The proofs are left to the reader. 

\begin{lemma}
\label{lemma-L-divisor}
\begin{enumerate}
\item The morphism $a_r \colon \bLp_{r-1} \to \bLm_{r}$ embeds 
$\bLp_{r-1}$ as a divisor, cut out by a section of the line bundle 
$(\cU_r/\cU_{r-1})(H)$. 
\item 
\label{lemma-L-divisor-b}
The morphism $b_{s+1} \colon \bLvp_{s} \to \bLvm_{s+1}$ 
embeds $\bLvp_{s}$ as a divisor, cut out by a section of the line 
bundle 
$(\cU_r/\cU_{r-1})^{\svee}(H')$. 
\end{enumerate}
\end{lemma}

\begin{lemma}
\label{lemma-zetar-id}
\begin{enumerate}
\item 
\label{lemma-zetar-id-1}
The morphism 
$(\zeta_r, \pr_2) \colon \bL_r \times_{\bG_r} \bLv_s \to \bH \times_{\bPv} \bLv_s $
is a closed immersion, with image cut out by a regular section of the vector bundle 
$(\cW_s/\cO(-H'))^{\svee}(H)$. 

\item 
\label{lemma-zetar-id-2}
The morphism 
$(\pr_1, \zeta_r) \colon \bL_r \times_{\bG_r} \bLv_s \to \bL_r \times_{\bP} \bH$ 
is a closed immersion, with image cut out by a regular section of the vector bundle 
$(\cU_r/\cO(-H))^{\svee}(H')$. 
\end{enumerate}
\end{lemma}

\begin{lemma}
\label{lemma-pminusr-fmr}
The morphism 
$(\pminus_r , \fm_r) \colon \bLm_r \to \bP \stimes \Fl{r-1}{r}$   
is a closed immersion, with image cut out by a regular section of the vector bundle 
$((V \otimes \cO)/\cU_r) (H)$. 
\end{lemma}

\begin{lemma}
\label{lemma-xi}
The morphism 
$\beta_{s+1} = \pivm_{s+1} \circ b_{s+1} \colon \bLvp_s \to \bLv_{s+1}$ 
is the projectivization of the vector bundle 
$(\cW_{s+1}/\cO(-H'))^{\svee}$. 
\end{lemma}

\subsection{Relations between the functors}
We aim here to relate the various kernel functors introduced above. 
The main statement we are after is Proposition~\ref{proposition-Phi-relations}. 
We start with some preparations. 

\begin{lemma}
\label{Phi-adjoints}
The functors 
\begin{align*}
\Phi_r & \colon \bLv_s(\cB) \to \bL_r(\cA) ,  &  \Phi_{r-1} & \colon \bLv_{s+1}(\cB) \to \bL_{r-1}(\cA)  , \\ 
\Phim_{r} & \colon \bLvp_s(\cB) \to \bLm_r(\cA) , & \Phip_{r-1} & \colon \bLvm_{s+1}(\cB) \to \bLp_{r-1}(\cA) , 
\end{align*}
admit left adjoints $\Phi_r^*, (\Phim_r)^*, \Phi_{r-1}^*, (\Phip_{r-1})^*$, 
and right adjoints $\Phi_r^!,  (\Phim_r)^!, \Phi_{r-1}^!, (\Phip_{r-1})^!$. 
\end{lemma}

\begin{proof}
We show the claim for $\Phi_r$. The arguments for the other functors are similar and left to the reader. 
The proof consists of two steps: first we write down a functor which is easily seen to admit 
adjoints, and then we prove this functor is equivalent to $\Phi_r$. 

Consider the universal space of linear sections $\bLv_s(\bH) = \bH \times_{\bPv} \bLv_s$ 
of $\bH$ with respect to the morphism $\bH \to \bPv$. 
We define a space $\bLL_r(\bH)$ by the fiber product diagram 
\begin{equation*}
\xymatrix{
& \ar[dl]_{a} \bLL_r(\bH) \ar[dr]^{b} & \\ 
\ar[dr] \bL_r & & \bLv_s(\bH) \ar[dl] \\ 
& \bG_r & 
}
\end{equation*}
By base changing the $\bP$-linear category $\cA$ along the top of this diagram, 
we obtain a diagram of $\bG_r$-linear functors 
\begin{equation*}
\xymatrix{
& \ar[dl]_{a_*} \bLL_r(\bH(\cA)) & \\ 
\bL_r(\cA) & & \bLv_s(\bH(\cA)) \ar[ul]_{b^*} 
}
\end{equation*} 
We denote by 
\begin{equation*}
\bLv_s(\Phi_\cE) \colon \bLv_s(\cB) \to \bLv_s(\bH(\cA)) 
\end{equation*}
the $\bLv_s$-linear functor induced by $\Phi_{\cE} \colon \cB \to \bH(\cA)$ on the linear section categories. 
Set 
\begin{equation*}
\Psi_r = a_* \circ b^* \circ \bLv_s(\Phi_\cE) \colon \bLv_s(\cB) \to \bL_r(\cA).  
\end{equation*} 
The functor $\Psi_r$ is $\bG_r$-linear, since it is a composition of such functors. 
The functors $a_*$ and $b^*$ in this composition admit left and right adjoints by 
Remark~\ref{remark:good-morphism}. 
Further, recall \eqref{Phigammaphi} that $\Phi_{\cE} = \gamma \circ \phi$ by definition. 
The functor $\phi$ admits adjoints by Setup~\ref{setup-HPD}\eqref{setup-fully-faithful}, 
and so does $\gamma$ by Lemma~\ref{lemma-HPD-admissible}. 
It follows that $\Phi_{\cE}$ and hence also $\bLv_s(\Phi_\cE)$ admits adjoints. 
In conclusion, $\Psi_r$ admits left and right adjoints, being a composition of functors with 
this property. 

Now we prove that $\Psi_r$ is equivalent to $\Phi_r$. 
Consider the commutative diagram 
\begin{equation*}
\xymatrix{
\bLL_r(\bH) \times_{\bLv_s} \bLv_s \ar[d]_-{a \times \id} \ar[r]^-{b \times \id} & \bLv_s(\bH) \times_{\bLv_s} \bLv_s \ar[r]^-{\psi} & 
\bH \times_{\bPv} \bPv \\ 
\bL_r \times_{\bG_r} \bLv_s \ar[urr]_-{\zeta_r} 
}
\end{equation*}
where $\psi$ is the evident morphism. 
We have written the spaces in the above diagram with redundant fiber products because 
we want to think of them as morphisms between kernel spaces. 
Recall that by definition $\cE_r = \zeta_r^*\cE$ is the kernel for the functor $\Phi_r$. 
It is easy to see that $\psi^*\cE$ is a kernel for the functor $\bLv_s(\Phi_\cE)$. 
Then by Lemma~\ref{lemma-pullback-kernel-T-equal} 
we find $(b \times \id)^*\psi^*\cE$ is a kernel for the functor 
$b^* \circ \bLv_s(\Phi_\cE) \colon \bLv_s(\cB) \to \bLL_r(\bH(\cA))$, 
and then further by Lemma~\ref{lemma-pushforward-kernel} we find that 
$(a \times \id)_*(b \times \id)^*\psi^*\cE$ is a kernel for $\Psi_r$. 
Hence to prove $\Psi_r \simeq \Phi_r$, it suffices to show 
\begin{equation*}
(a \times \id)_* \circ (b \times \id)^* \circ \psi^* \simeq \zeta_r^* 
\colon \Perf(\bH \times_{\bPv} \bPv) \to \Perf(\bL_r \times_{\bG_r} \bLv_s). 
\end{equation*}
By the above diagram, the left side can be rewritten as 
\begin{equation*}
(a \times \id)_* \circ (b \times \id)^* \circ \psi^* \simeq 
(a \times \id)_* \circ (a \times \id)^* \circ \zeta_r^*. 
\end{equation*} 
But $a \times \id \colon \bLL_r(\bH) \times_{\bLv_s} \bLv_s \to \bL_r \times_{\bG_r} \bLv_s$ 
is a projective bundle, being a base change of the projective bundle $\bH \to \bPv$, 
and hence $(a \times \id)_* \circ (a \times \id)^* \simeq \id$. 
This finishes the proof that $\Psi_r \simeq \Phi_r$. 
Since we already showed that $\Psi_r$ admits left and right adjoints, we are done. 
\end{proof}

\begin{lemma}
\label{lemma-base-change-Phi}
There are equivalences of functors 
\begin{align*}
\Phim_{r} \circ (\pivp_{s})^* & \simeq (\pim_r)^* \circ \Phi_r  ,  & 
\Phip_{r-1} \circ (\pivm_{s+1})^* & \simeq (\pip_{r-1})^* \circ  \Phi_{r-1}, \\
(\pim_{r})_* \circ \Phim_{r}   & \simeq \Phi_{r} \circ (\pivp_{s})_* , & 
(\pip_{r-1})_* \circ \Phip_{r-1} & \simeq \Phi_{r-1} \circ (\pivm_{s+1})_* , \\ 
 (\pivp_{s})_*  \circ (\Phim_{r})^*   & \simeq (\Phi_{r})^* \circ (\pim_{r})_*,  & 
(\pivm_{s+1})_* \circ (\Phip_{r-1})^* & \simeq \Phi_{r-1}^* \circ (\pip_{r-1})_* , \\  
 (\pivp_{s})_!  \circ (\Phim_{r})^*   & \simeq (\Phi_{r})^* \circ (\pim_{r})_!,  & 
(\pivm_{s+1})_! \circ (\Phip_{r-1})^* & \simeq \Phi_{r-1}^* \circ (\pip_{r-1})_! , 
\end{align*}
where $(\pivp_{s})_!$ denotes the left adjoint of $(\pivp_{s})^{*}$, 
and similarly for $(\pim_{r})_!, (\pivm_{s+1})_!, (\pip_{r-1})_!$. 
\end{lemma}

\begin{proof}
The first two rows follow from Lemma~\ref{lemma-pullback-kernel} and Remark~\ref{remark-pullback-kernel}. 
By adjunction the third row is equivalent to the assertion 
\begin{align*}
\Phim_{r} \circ  (\pivp_{s})^!  & \simeq (\pim_{r})^! \circ \Phi_{r} ,  & 
\Phip_{r-1} \circ (\pivm_{s+1})^!  & \simeq (\pip_{r-1})^! \circ \Phi_{r-1} . 
\end{align*}
We have 
\begin{align*}
(\pim_r)^! &   \simeq ( - \otimes \omega_{\Fl{r-1}{r}/\bG_{r}}) \circ (\pim_{r})^* , \\
(\pivp_{s})^! & \simeq ( - \otimes \omega_{\Fl{r-1}{r}/\bG_{r}}) \circ (\pivp_{s})^* , \\
(\pip_{r-1})^! & \simeq ( - \otimes \omega_{\Fl{r-1}{r}/\bG_{r-1}}) \circ (\pip_{r-1})^*  , \\
(\pivm_{s+1})^! & \simeq (- \otimes \omega_{\Fl{r-1}{r}/\bG_{r-1}})  \circ (\pivm_{s+1})^* . 
\end{align*}
Hence by $\Fl{r-1}{r}$-linearity the assertion reduces to the first row. 
Similarly, we have 
\begin{align*}
(\pim_r)_! &   \simeq (\pim_{r})_* \circ ( - \otimes \omega_{\Fl{r-1}{r}/\bG_{r}}) , \\
(\pivp_{s})_! & \simeq  (\pivp_{s})_* \circ ( - \otimes \omega_{\Fl{r-1}{r}/\bG_{r}}) , \\
(\pip_{r-1})_! & \simeq (\pip_{r-1})_* \circ ( - \otimes \omega_{\Fl{r-1}{r}/\bG_{r-1}})  , \\
(\pivm_{s+1})_! & \simeq (\pivm_{s+1})_* \circ (- \otimes \omega_{\Fl{r-1}{r}/\bG_{r-1}}) .
\end{align*}
Hence by $\Fl{r-1}{r}$-linearity the fourth row follows from the third. 
\end{proof}

We define 
\begin{align*}
\bMm_{r} & = \bLm_{r} \times_{\Fl{r-1}{r}} \bLvp_s   , \\ 
\bMp_{r-1} & = \bLp_{r-1} \times_{\Fl{r-1}{r}}  \bLvm_{s+1}  . 
\end{align*}
Then by definition $\Phim_r$ is defined by an $\bMm_{r}$-kernel and 
$\Phip_{r-1}$ by an $\bMp_{r-1}$-kernel. 
These spaces fit into a cartesian commutative diagram 
\begin{equation}
\label{LrLs-square}
\vcenter{
\xymatrix{
\bLp_{r-1} \times_{\Fl{r-1}{r}} \bLvp_s \ar[rr]^{\id \times b_{s+1}} \ar[d]_{a_r \times \id} & &  
\bMp_{r-1} \ar[d]^{a_r \times \id}  \\ 
\bMm_{r} \ar[rr]^{\id \times b_{s+1}} &&  \bLm_r \times_{\Fl{r-1}{r}} \bLvm_{s+1}
}}
\end{equation}
By Lemma \ref{lemma-L-divisor} all of the morphisms in this diagram are divisorial embeddings.  

The projections $\pminus_r \colon \bLm_r \to \bPv$ and  $\qminus_{s+1} \colon \bLvm_{s+1} \to \bPv$ 
induce a morphism 
\begin{equation}
\label{nur}
\nu_r = \pminus_r \times \qminus_{s+1} \colon  \bLm_r \times_{\Fl{r-1}{r}} \bLvm_{s+1}  \to \bP \stimes \bPv . 
\end{equation}
Let $\tbM_r$ be defined by the fiber product diagram 
\begin{equation}
\label{tbMr-square}
\vcenter{
\xymatrix{
\tbM_r \ar[rr]^{\iota_r} \ar[d]_{\tzeta_r} && \bLm_r \times_{\Fl{r-1}{r}} \bLvm_{s+1}  \ar[d]^{\nu_r} \\ 
\bH \ar[rr]^{\iota} && \bP \stimes \bPv 
}}
\end{equation} 
The compositions  
\begin{align*}
\zetam_r & \colon \bMm_r \xrightarrow{ \pim_r \times \pivp_s} \bL_r \times_{\bG_r} \bLv_s \xrightarrow{\, \zeta_r \,} \bH, \\
\zetap_{r-1} & \colon \bMp_{r-1} \xrightarrow{\pip_{r-1} \times \pivm_{s+1}} \bL_{r-1} \times_{\bG_{r-1}} \bLv_{s+1} \xrightarrow{\, \zeta_{r-1} \,} \bH , 
\end{align*}
factor through closed embeddings 
\begin{align}
\label{deltam} \deltam_{r} \colon & \bMm_{r}  \hookrightarrow \tbM_r \\ 
\label{deltap} \deltap_{r-1} \colon & \bMp_{r-1}  \hookrightarrow \tbM_r. 
\end{align}

\begin{lemma}
\label{lemma-M-union}
We have 
\begin{equation*}
\tbM_r = \bMm_{r} \cup \bMp_{r-1},
\end{equation*} 
where the right side is the 
scheme-theoretic union inside $\bLm_r \times_{\Fl{r-1}{r}} \bLvm_{s+1}$.  
\end{lemma}

\begin{proof}
Note that by definition $\tbM_r$ is cut out in $\bLm_r \times_{\Fl{r-1}{r}} \bLvm_{s+1}$ 
by a section of the line bundle $\cO(H+H')$. 
On the other hand, by Lemma~\ref{lemma-L-divisor} the union $\bMm_{r} \cup \bMp_{r-1}$ is 
cut out in $\bLm_r \times_{\Fl{r-1}{r}} \bLvm_{s+1}$ by a section of the line bundle 
$(\cU_r/\cU_{r-1})(H) \otimes (\cU_r/\cU_{r-1})^{\svee}(H') \cong \cO(H + H')$. 
These two sections of $\cO(H+H')$ coincide, cf. \cite[Lemma 6.13]{kuznetsov-HPD}. 
\end{proof} 

Let $\tcE_{r} = \tzeta_r^*\cE$ 
be the $\tbM_r$-kernel obtained by pulling back the $\bH$-kernel $\cE$. 
We denote by $D_r$ the Cartier divisor on $\bLvm_{s+1}$ corresponding to the 
line bundle $(\cU_r/\cU_{r-1})^{\svee}(H')$.
By Lemma~\ref{lemma-L-divisor} the scheme $\bMm_{r}$ is cut out in 
$\bLm_r \times_{\Fl{r-1}{r}} \bLvm_{s+1}$ by a section of $\cO(D_r)$. 

\begin{lemma}
\label{lemma-kernel-relation}
\begin{enumerate}
\item 
\label{kernel-triangle}
There is an exact triangle 
\begin{equation*}
(a_r \times \id)_*\cEp_{r-1} \otimes \cO(-D_r) \to \iota_{r*} \tcE_r \to 
(\id \times b_{s+1})_*\cEm_r
\end{equation*}
of $\bLm_r \times_{\Fl{r-1}{r}} \bLvm_{s+1}$-kernels. 

\item 
\label{kernel-equivalence-1}
There is an equivalence 
\begin{equation*}
(a_r \times \id)^*\cEm_r \simeq (\id \times b_{s+1})^* \cEp_{r-1}
\end{equation*}
of $\bLp_{r-1} \times_{\Fl{r-1}{r}} \bLvp_s$-kernels. 

\item 
\label{kernel-equivalence-2}
There is an equivalence 
\begin{equation*}
\iota_{r*} \tcE_r  \simeq  \nu_r^* \iota_* \cE
\end{equation*}
of $\bLm_r \times_{\Fl{r-1}{r}} \bLvm_{s+1}$-kernels. 
\end{enumerate}
\end{lemma}

\begin{proof}
It follows from Lemma~\ref{lemma-M-union} that there is an exact sequence 
\begin{equation*}
0 \to (\deltap_{r-1})_*\cO_{\bMp_{r-1}}(-D_r) \to \cO_{\tbM_r} \to (\deltam_r)_* \cO_{\bMm_r} \to 0
\end{equation*}
of sheaves on $\tbM_r$, where $\deltap_{r-1}$ and $\deltam_r$ are the embeddings~\eqref{deltap} and~\eqref{deltam}. 
Now~\eqref{kernel-triangle} follows by kernel formalism. 
More precisely, tensoring the above exact sequence by $\tcE_r$ gives an exact triangle of kernels by 
Lemma~\ref{lemma-functor-triangle-to-kernel}. 
But by definition $(\deltap_{r-1})^*(\tcE_r) \simeq \cEp_{r-1}$ and $(\deltam_r)^*(\tcE_r) \simeq \cEm_r$, 
so by the projection formula for kernels (Lemma~\ref{lemma-projection-formula}), the resulting exact triangle can be written 
\begin{equation*}
(\deltap_{r-1})_*\cEp_{r-1} \otimes \cO(-D_r) \to \tcE_r \to (\deltam_r)_*\cEm_r 
\end{equation*}
Now the result follows by pushing forward via $\iota_r \colon \tbM_r \to \bLm_r \times_{\Fl{r-1}{r}} \bLvm_{s+1}$ 
and using the projection formula once more. 

It follows from the definitions that both kernels appearing in~\eqref{kernel-equivalence-1} 
are equivalent to the pullback of $\tcE_r$ along the natural map 
$\bLp_r \times_{\Fl{r-1}{r}} \bLvp_{s+1} \to \tbM_r$, hence~\eqref{kernel-equivalence-1} 
holds. 

Finally,~\eqref{kernel-equivalence-2} follows from the cartesian square~\eqref{tbMr-square} 
and base change. 
\end{proof}

\begin{proposition}
\label{proposition-Phi-relations}
\begin{enumerate}
\item \label{functor-triangle-1}
There is an exact triangle  
\begin{align*}
\Phim_{r} \circ b_{s+1}^! \to (a_r)_* \circ \Phip_{r-1} \to \Phi_{\iota_{r*} \tcE_r \otimes \cO(D_r)}
\end{align*}
of functors $\bLvm_{s+1}(\cB) \to \bLm_r(\cA)$, where 
$b_{s+1}^! \colon \bLvm_{s+1}(\cB) \to \bLvp_{s}(\cB)$ 
is the functor induced by the right adjoint to 
$(b_{s+1})_* \colon \Perf(\bLvp_{s}) \to \Perf(\bLvm_{s+1})$

\item 
\label{functor-triangle-2}
There is an exact triangle 
\begin{align*}
(\Phi_{\iota_{r*} \tcE_r \otimes \cO(D_r)})^* \to (\Phip_{r-1})^* \circ a_r^* \to (b_{s+1})_* \circ (\Phim_r)^*
\end{align*}
of functors $\bLm_r(\cA) \to \bLvm_{s+1}(\cB)$. 

\item 
\label{Phim-Phip-relation}
There is an equivalence 
\begin{equation*}
a_r^* \circ \Phim_r \simeq \Phip_{r-1} \circ (b_{s+1})_*  
\end{equation*}
of functors $\bLvp_{s}(\cB) \to \bLp_{r-1}(\cA)$. 
\end{enumerate}
\end{proposition}

\begin{proof}
It follows from Lemma~\ref{lemma-L-divisor}\eqref{lemma-L-divisor-b} that 
$b_{s+1}^! = b_{s+1}^* \circ (- \otimes \cO(D_r))[-1]$. 
Now~\eqref{functor-triangle-1} follows by rotating the triangle of 
kernels from Lemma~\ref{lemma-kernel-relation}\eqref{kernel-triangle}, 
twisting by $\cO(D_r)$, passing to the associated kernel functors,  
and using Lemmas~\ref{lemma-twist-kernel} and~\ref{lemma-pushforward-kernel}. 
By Lemmas~\ref{lemma-adjoints-triangle} and~\ref{Phi-adjoints}, 
\eqref{functor-triangle-2} follows from~\eqref{functor-triangle-1} by passing to left adjoints. 
Finally, \eqref{Phim-Phip-relation} follows from 
Lemma~\ref{lemma-kernel-relation}\eqref{kernel-equivalence-1} by passing to the 
associated kernel functors and using Lemma~\ref{lemma-pullback-kernel}. 
\end{proof}

\subsection{Semiorthogonal sequences in $\ker \Phi_r^*$ and $\ker \Phi_r$} 
Our next goal is to show that the semiorthogonal sequence to the right of $\Ku_r(\cA)$ in 
the decomposition~\eqref{Cr} is contained in $\ker \Phi_r^*$, and the semiorthogonal 
sequence to the left of $\Ku'_s(\cB)$ in~\eqref{Ds} is contained $\ker \Phi_r$. 
Later we will combine this result with the fact that the $\Phi_r$ are left splitting 
(Proposition~\ref{proposition-Phi-splitting}) to reduce Theorem~\ref{key-theorem} 
to a statement about the generation of $\bL_r(\cA)$ and $\bLv_s(\cB)$ by 
certain semiorthogonal sequences. 

\begin{lemma}
\label{lemma-resolution-pr-Phir} 
For $0 \leq i \leq s-1$ there are kernels 
\begin{equation*}
\cF_{r,i}, \, \cK_{r,i}  \in \FM{\cB}{\bPv}{\cA}{\bP}{\bLv_s}{\bP}{S} 
\end{equation*}
such that: 
\begin{enumerate}
\item $\Phi_{\cF_{r,0}} \simeq p_{r*} \circ \Phi_r$. 
\item For all $i$ there is an exact triangle 
\begin{equation*}
\cF_{r,i+1} \to \cK_{r,i} \to \cF_{r,i} 
\end{equation*}
where for $i = s-1$ we set $\cF_{r,s} = 0$. 
\item 
\label{Phi-Ki}
There is an equivalence 
\begin{equation*}
\Phi_{\cK_{r,i}} \simeq  
 (- \otimes \cO(-iH)) \circ p_* \circ \Phi_\cE \circ q_{s*} \circ (- \otimes \wedge^i(\cW_s/\cO(-H')))  . 
\end{equation*} 
\end{enumerate}
\end{lemma}

\begin{proof}
By definition the functor $\Phi_r$ is given by the kernel 
\begin{equation*}
\cE_r = \zeta^*_r \cE \in 
\FM{\cB}{\bPv}{\cA}{\bP}{\bLv_s}{\bL_r}{\bG_r} . 
\end{equation*}
Hence by Lemma~\ref{lemma-pushforward-kernel} the functor $p_{r*} \circ \Phi_r$ 
is given by the pushforward of this kernel along 
\begin{equation*}
p_r \times \id \colon \bL_r \times_{\bG_r} \bLv_s \to \bP \stimes \bLv_s, 
\end{equation*}
i.e. by the kernel $(p_r \times \id)_*\cE_r$. 
To understand this kernel, we consider the commutative diagram
\begin{equation}
\label{sos-kerPhi-diagram1}
\vcenter{
\xymatrix{
\bL_r \times_{\bG_r} \bLv_s \ar[d]_{(\zeta_r, \pr_2)} \ar[dr]^{p_r \times \id} & \\
\bH \times_{\bPv} \bLv_s \ar[d]_{\pr_1} \ar[r]^{p \times \id} & \bP \stimes \bLv_s \ar[d]^{\id \times q_s} \\
\bH \ar[r]^{\iota} & \bP \stimes \bPv
}}
\end{equation}
where the square is cartesian. 
We have 
\begin{align*}
(p_r \times \id)_* \cE_r & \simeq (p \times \id)_*(\zeta_r, \pr_2)_*(\zeta_r, \pr_2)^* \pr_1^*\cE  \\
& \simeq (p \times \id)_*(\pr_1^*\cE \otimes (\zeta_r, \pr_2)_*\cO) . 
\end{align*}
It follows from Lemma~\ref{lemma-zetar-id}\eqref{lemma-zetar-id-1}  
that there is a resolution of $(\zeta_r, \pr_2)_*\cO$
of the form 
\begin{equation}
\label{resolution-zeta-id-O}
0 \to  \wedge^{s-1}(\cW_s/\cO(-H'))(-(s-1)H) \to \cdots \to (\cW_s/\cO(-H'))(-H) \to \cO \to (\zeta_r, \pr_2)_*\cO \to 0. 
\end{equation}
Let 
\begin{equation*}
\cK_{r,i} = (p \times \id)_*(\pr_1^*\cE \otimes  \wedge^i(\cW_s/\cO(-H'))(-iH)) 
\end{equation*}
denote the kernel corresponding to the $i$-th term in this resolution. 
Then 
\begin{align*}
\cK_{r,i}  & \simeq ((p \times \id)_*\pr_1^*\cE) \otimes  \wedge^i(\cW_s/\cO(-H'))(-iH)    \\ 
& \simeq ((\id \times q_s)^* \iota_*\cE) \otimes  \wedge^i(\cW_s/\cO(-H'))(-iH) 
\end{align*}
where the first line holds since $\wedge^i(\cW_s/\cO(-H'))$
is pulled back from $\bLv_s$ and $\cO(-iH)$ from $\bP$, 
and the second since the square in~\eqref{sos-kerPhi-diagram1} is cartesian. 
The kernel formalism then shows that $\cK_i$ satisfies part~\eqref{Phi-Ki} of the lemma, 
and the rest follows by splitting the resolution~\eqref{resolution-zeta-id-O} into 
short exact sequences. 
\end{proof}

\begin{proposition}
\label{proposition-sos-kerPhiadjoint}
There is an inclusion 
\begin{equation*}
\llangle \cA_s(H) \sotimes \Perf(\bG_r), 
\dots, 
\cA_{m-1}((m-s)H) \sotimes \Perf(\bG_r)  \rrangle \subset  \ker \Phi_r^* .
\end{equation*}
\end{proposition}

\begin{proof}
By adjunction, the assertion is equivalent to 
\begin{equation*}
\im \Phi_r \subset 
\llangle \cA_s(H) \sotimes \Perf(\bG_r), 
\dots, 
\cA_{m-1}((m-s)H) \sotimes \Perf(\bG_r)  \rrangle^{\perp} ,  
\end{equation*}
where the orthogonal is taken inside $\bL_r(\cA)$. 
By $\bG_r$-linearity of $\Phi_r$, this is equivalent to 
\begin{equation*}
\im(p_{r*} \circ \Phi_r) \subset 
\llangle \cA_s(H) , 
\dots, 
\cA_{m-1}((m-s)H) \rrangle^{\perp} 
\end{equation*}
where the orthogonal is taken inside $\cA$.  
The functor $\Phi_{\cK_{r,i}}$ from Lemma~\ref{lemma-resolution-pr-Phir} 
satisfies
\begin{equation}
\label{claim-Phi-Ki}
\im \Phi_{\cK_{r,i}} \subset \cA_0(-iH). 
\end{equation}
Indeed, using the description of this functor from Lemma~\ref{lemma-resolution-pr-Phir}\eqref{Phi-Ki}, 
the claim follows from Lemma~\ref{lemma-im-ker-p-gamma}. 
Hence we find 
\begin{equation*}
\im(p_{r*} \circ \Phi_r) \subset 
\llangle 
\cA_0(-(s-1)H), \cA_0(-(s-2))\dots, \cA_{0} 
\rrangle, 
\end{equation*}
so we are done by Lemma~\ref{lemma-A0-sequence} below. 
\end{proof}

\begin{lemma}
\label{lemma-A0-sequence}
For any $i \geq 0$ the following subcategories of $\cA$ coincide: 
\begin{align*}
& \llangle 
\cA_0, \cA_0(H), \dots, \cA_0((i-1)H)
\rrangle , \\
& \llangle \cA_0, \cA_1(H), \dots, \cA_{i-1}((i-1)H) 
\rrangle , \\
& \llangle \cA_i(iH), \cA_{i+1}((i+1)H), \dots, \cA_{m-1}((m-1)H) 
\rrangle^{\perp} .
\end{align*}
\end{lemma}

\begin{proof}
The semiorthogonal decomposition 
\begin{equation*}
\cA = \llangle \cA_0, \cA_1(H), \dots, \cA_{m-1}((m-1)H) \rrangle
\end{equation*}
implies the last two categories coincide, and 
the first category contains the second and is contained in the third. 
\end{proof}

\begin{proposition}
\label{proposition-sos-kerPhi}
There is an inclusion 
\begin{equation*}
\llangle 
\cB_{1-n}((r-n)H') \sotimes \Perf(\bGv_s), 
\dots , \cB_{-r}(-H')\sotimes \Perf(\bGv_s)  
\rrangle 
\subset \ker \Phi_r. 
\end{equation*}
\end{proposition}

\begin{proof}
Analogous to Proposition~\ref{proposition-sos-kerPhiadjoint} and 
left to the reader. 
\end{proof}

\subsection{Technical results about the functor $\Phi_{\iota_{r*} \tcE_r \otimes \cO(D_r)}$} 
Here we prove two related results about the (left adjoint of the) 
functor $\Phi_{\iota_{r*} \tcE_r \otimes \cO(D_r)}$ appearing in Proposition~\ref{proposition-Phi-relations}, 
which will be needed later in our induction arguments. 
Namely, we describe the image of a certain subcategory under 
$(\Phi_{\iota_{r*} \tcE_r \otimes \cO(D_r)})^*$ (Proposition~\ref{proposition-im-Phi-tE}) 
and show that certain composite functors 
involving $(\Phi_{\iota_{r*} \tcE_r \otimes \cO(D_r)})^*$ vanish (Proposition~\ref{proposition-Phi-compositions}). 

\begin{proposition}
\label{proposition-im-Phi-tE}
The image  
\begin{equation*}
(\Phi_{\iota_{r*} \tcE_r \otimes \cO(D_r)})^*
\left ( \left ( \llangle \cA_{s}(H), \dots, \cA_{m-1}((m-s)H) \rrangle 
\otimes \Perf(\Fl{r-1}{r}) \right )^{\perp} \right )  
\end{equation*}
lies inside the subcategory 
\begin{equation*}
\cB_{1-r}(-H') \sotimes \Perf(\Fl{r-1}{r}) \subset \bLvm_{s+1}(\cB) . 
\end{equation*}
\end{proposition}

\begin{proof}
By Lemma~\ref{lemma-twist-kernel} we have 
\begin{equation*}
(\Phi_{\iota_{r*} \tcE_r \otimes \cO(D_r)})^* \simeq (- \otimes \cO(-D_r)) \circ (\Phi_{\iota_{r*} \tcE_r})^*  . 
\end{equation*}
Since by definition $\cO(D_r) = (\cU_r/\cU_{r-1})^{\svee}(H')$, 
it therefore suffices to show 
\begin{equation*}
(\Phi_{\iota_{r*} \tcE_r })^*
\left ( \left ( \llangle \cA_{s}(H), \dots, \cA_{m-1}((m-s)H) \rrangle 
\otimes \Perf(\Fl{r-1}{r}) \right )^{\perp} \right ) 
\subset \cB_{1-r} \sotimes \Perf(\Fl{r-1}{r}) .
\end{equation*}
For this, we consider the diagram 
\begin{equation*}
\xymatrix{
& \bLm_r \times_{\Fl{r-1}{r}} \bLvm_{s+1} \ar[d]_{\id \times j} \ar[r] & \bLvm_{s+1} \ar[d]^j \\
\bLm_r \ar[d]_{\pminus_r} & \ar[l] \bLm_r \times_{\Fl{r-1}{r}} (\Fl{r-1}{r} \stimes \bPv) \ar[d]_{\pminus_r \times \pr_2} \ar[r] & \Fl{r-1}{r} \stimes \bPv \\ 
\bP & \ar[l] \bP \stimes \bPv & 
}
\end{equation*}
where $j \colon \bLvm_{s+1} \to \Fl{r-1}{r} \stimes \bPv$ is the embedding and the 
squares are cartesian. 
By definition the morphism 
$\nu_r  \colon  \bLm_r \times_{\Fl{r-1}{r}} \bLvm_{s+1}  \to \bP \stimes \bPv$ 
defined by~\eqref{nur} satisfies $\nu_r = (\pminus_r \times \pr_2) \circ (\id \times j)$. 
Hence by Lemma~\ref{lemma-kernel-relation}\eqref{kernel-equivalence-2} 
there is an equivalence
\begin{equation*}
\iota_{r*} \tcE_r  \simeq  \nu_r^* \iota_* \cE \simeq (\id \times j)^* (\pminus_r \times \pr_2)^* \iota_* \cE. 
\end{equation*}
It follows from Lemma~\ref{lemma-pullback-kernel} that 
\begin{equation*}
\Phi_{\iota_{r*} \tcE_r} \simeq \Phi_{(\pminus_r \times \pr_2)^* \iota_* \cE} \circ j_* 
\end{equation*}
and hence 
\begin{equation*}
(\Phi_{\iota_{r*} \tcE_r})^* \simeq j^* \circ  (\Phi_{(\pminus_r \times \pr_2)^* \iota_* \cE})^*
\end{equation*}
To prove the lemma, it thus suffices to show the image 
\begin{equation*}
(\Phi_{(\pminus_r \times \pr_2)^* \iota_* \cE})^*
\left ( \left ( \llangle \cA_{s}(H), \dots, \cA_{m-1}((m-s)H) \rrangle 
\otimes \Perf(\Fl{r-1}{r}) \right )^{\perp} \right ) 
\end{equation*}
is contained in 
\begin{equation*}
\cB_{1-r} \sotimes \Perf(\Fl{r-1}{r})  \subset \cB \otimes_{\Perf(\bPv)} \Perf(\Fl{r-1}{r} \stimes \bPv). 
\end{equation*}
Since the image in question is $\Fl{r-1}{r}$-linear, it suffices to show 
\begin{equation}
\label{im-Phi-tE-reduction-1}
\pr_{2*}(\Phi_{(\pminus_r \times \pr_2)^* \iota_* \cE})^*
\left ( \left ( \llangle \cA_{s}(H), \dots, \cA_{m-1}((m-s)H) \rrangle 
\otimes \Perf(\Fl{r-1}{r}) \right )^{\perp} \right ) \subset \cB_{1-r} 
\end{equation}
as subcategories of $\cB$. 
By the above diagram and Lemma~\ref{lemma-pullback-kernel}, there is an equivalence 
\begin{equation*}
\Phi_{(\pminus_r \times \pr_2)^* \iota_* \cE} \circ \pr_2^* \simeq (\pminus_r)^* \circ \Phi_{\iota_*\cE}  .
\end{equation*}
Hence, if $(\pr_2)_{!}$ and $(\pminus_r)_!$ denote the left adjoints of 
the pullback functors of the morphisms $\pr_2 \colon \Fl{r-1}{r} \stimes \bPv \to \bPv$ and 
$\pminus_r \colon \bLm_r \to \bP$, then we have 
\begin{equation}
\label{pr2-Phi-pminusr}
(\pr_2)_! \circ (\Phi_{(\pminus_r \times \pr_2)^* \iota_* \cE})^* \simeq (\Phi_{\iota_*\cE})^* \circ (\pminus_r)_!. 
\end{equation}
It is easy to see
\begin{align}
\label{adjoint-pr2}
(\pr_2)_! & \simeq \pr_{2*} \circ (- \otimes \omega_{\Fl{r-1}{r}/S})  \\
\label{adjoint-pminusr}
(\pminus_r)_! & \simeq (\pminus_r)_* \circ (- \otimes \cO(sH) \otimes \det(\cU_r)^{-1} \otimes 
 \omega_{\Fl{r-1}{r}/S}[-s] ) 
\end{align}
where we have used Lemma~\ref{lemma-pminusr-fmr} in deriving the second equivalence. 
Combining \eqref{pr2-Phi-pminusr}, \eqref{adjoint-pr2}, \eqref{adjoint-pminusr}, 
and using that $\omega_{\Fl{r-1}{r}/S}$ and $\det(\cU_r)^{-1}$ are pulled back from 
$\Fl{r-1}{r}$, 
the claim~\eqref{im-Phi-tE-reduction-1} reduces to showing  
\begin{equation*}
(\Phi_{\iota_*\cE})^*\left( \llangle \cA_{s}((s+1)H), \dots, \cA_{m-1}(mH) \rrangle^\perp \right) 
\subset \cB_{1-r}. 
\end{equation*}
But $(\Phi_{\iota_*\cE})^* \simeq \phi^* \circ \gamma^* \circ p^*$ and 
\begin{equation*}
\llangle \cA_{s}((s+1)H), \dots, \cA_{m-1}(mH) \rrangle^\perp = 
\llangle \cA_{0}(H), \dots, \cA_{s-1}(sH) \rrangle, 
\end{equation*}
so the desired inclusion follows from Lemma~\ref{lemma-pgamma-adjoint} below and the definition 
of $\cB_{1-r}$. 
\end{proof}

\begin{lemma}
\label{lemma-pgamma-adjoint}
For any $i \geq 0$ the functor $\gamma^* \circ p^* \colon \cA \to \cAd$ satisfies
\begin{equation*}
\gamma^*p^* \left( 
\llangle \cA_0(H), \dots, \cA_{i-1}(i) \rrangle \right) \subset \cAd_{i+1-N}. 
\end{equation*}
\end{lemma}

\begin{proof}
By Lemma~\ref{lemma-im-ker-p-gamma} the functor $\gamma^*\circ p^*$ kills 
${^\perp}\cA_0$. 
Hence
$\gamma^*p^*(C) \simeq \gamma^*p^*\alpha_0^*(C)$ for any \mbox{$C \in \cA$}, 
where recall $\alpha_0 \colon \cA_0 \to \cA$ is the inclusion. 
But it follows immediately from the definitions that 
\begin{equation*}
\alpha_0^*\left( \llangle \cA_0(H), \dots, \cA_{i-1}(iH) \rrangle \right) = 
\llangle \fatw_0, \dots, \fatw_{i-1} \rrangle,  
\end{equation*}
and then the result follows from the definition of $\cAd_{i+1-N}$. 
\end{proof}

We will need the following observation. 
\begin{lemma}
\label{lemma-im-Phim-Phip}
There are inclusions 
\begin{align*}
& \im \Phim_r \subset 
\left( \llangle \cA_s(H), 
\dots, 
\cA_{m-1}((m-s)H) \rrangle \sotimes \Perf(\Fl{r-1}{r})  \right)^{\perp} \\ 
& \llangle 
\cB_{1-n}((r-n-1)H') , 
\dots , \cB_{1-r}(-H') \rrangle \sotimes \Perf(\Fl{r-1}{r}) 
\subset \ker \Phip_{r-1}. 
\end{align*}
\end{lemma}

\begin{proof}
By Lemma~\ref{lemma-base-change-Phi} we have 
\begin{align*}
(\pim_{r})_* \circ \Phim_r & \simeq \Phi_r \circ (\pivp_s)_* , \\ 
\Phip_{r-1} \circ (\pivm_{s+1})^* & \simeq (\pip_{r-1})^* \circ \Phi_{r-1}, 
\end{align*}
from which the result follows using $\Fl{r-1}{r}$-linearity and 
Propositions~\ref{proposition-sos-kerPhiadjoint} and~\ref{proposition-sos-kerPhi}. 
\end{proof}

\begin{proposition}
\label{proposition-Phi-compositions}
We have 
\begin{equation}
\label{Phi-composition-vanish-1}
\Phip_{r-1} \circ (\Phi_{\iota_{r*} \tcE_r \otimes \cO(D_r)})^* \circ \Phim_r \simeq 0 . 
\end{equation}
Moreover, we have the stronger vanishings 
\begin{align}
\label{Phi-composition-vanish-2}
(\Phi_{\iota_{r*} \tcE_r \otimes \cO(D_r)})^* \circ \Phim_r & \simeq 0  \, \,  \text{ if } r-1 \geq n , \\ 
\label{Phi-composition-vanish-3}
\Phip_{r-1} \circ  (\Phi_{\iota_{r*} \tcE_r \otimes \cO(D_r)})^* & \simeq 0  \, \, \text{ if } r \leq N-m . 
\end{align} 
\end{proposition}

\begin{proof}
Follows from Lemmas \ref{lemma-im-Phim-Phip} and \ref{proposition-im-Phi-tE} and 
the observations that $\cB_{1-r} = 0$ if $r-1 \geq n$ and 
\begin{equation*}
\left( \llangle \cA_s(H), 
\dots, 
\cA_{m-1}((m-s)H) \rrangle \sotimes \Perf(\Fl{r-1}{r})  \right)^{\perp} = \bLm_r(\cA)
\end{equation*} 
if $s \geq m$ or equivalently $r \leq N - m$. 
\end{proof}

\subsection{The functors $\Phi_r$ are left splitting} 

\begin{proposition}
\label{proposition-Phi-splitting}
The functor $\Phi_r \colon \bLv_s(\cB) \to \bL_r(\cA)$ is left splitting for any $r$, 
and hence 
\begin{align*}
\bL_r(\cA)  & = \llangle \im \Phi_r, \ker \Phi_r^* \rrangle, \\
\bLv_s(\cB) & = \llangle \ker \Phi_r , \im \Phi_r^* \rrangle. 
\end{align*}
\end{proposition}

\begin{proof}
We argue by descending induction on $r$. 
For $r = N$ the result is trivial as $\bL_0(\cB) = 0$, and 
for $r = N-1$ the result holds since by assumption the 
functor $\Phi_{N-1}$ is fully faithful with admissible image. 
Now assume the result holds for $\Phi_{r}$. We will prove 
that $\Phi_{r-1}$ satisfies criterion~\ref{splitting-functor-2} 
of Theorem~\ref{theorem-splitting-functors}. 

Note that it follows from Lemma~\ref{lemma-xi} that the functor 
\begin{equation*}
\beta_{s+1}^* \colon \bLv_{s+1}(\cB) \to \bLvp_s(\cB) 
\end{equation*}
is fully faithful, and hence 
\begin{equation*}
(\pivm_{s+1})_*(b_{s+1})_*\beta_{s+1}^* \simeq (\beta_{s+1})_*\beta_{s+1}^* \simeq \id.
\end{equation*}
Thus applying Lemma~\ref{lemma-base-change-Phi} we find 
\begin{align}
\nonumber
\Phi_{r-1} \Phi_{r-1}^* \Phi_{r-1} 
& \simeq \Phi_{r-1} \Phi_{r-1}^* \Phi_{r-1} (\pivm_{s+1})_*(b_{s+1})_*\beta_{s+1}^*  \\
\label{splitting-induction-1}
& \simeq (\pip_{r-1})_*\Phip_{r-1}(\Phip_{r-1})^*\Phip_{r-1}(b_{s+1})_*\beta_{s+1}^*. 
\end{align}

Now we examine the inner term $\Phip_{r-1}(\Phip_{r-1})^*\Phip_{r-1}(b_{s+1})_*$. 
First from Proposition~\ref{proposition-Phi-relations}\eqref{Phim-Phip-relation} we 
deduce 
\begin{equation*}
\Phip_{r-1}(\Phip_{r-1})^*\Phip_{r-1}(b_{s+1})_* \simeq 
\Phip_{r-1}(\Phip_{r-1})^*a_r^*\Phim_r. 
\end{equation*}
By~\eqref{Phi-composition-vanish-1} if we compose the 
the triangle of Proposition~\ref{proposition-Phi-relations}\eqref{functor-triangle-2} 
on the left with $\Phip_{r-1}$ and on the right with $\Phim_r$, 
then the first term vanishes and we obtain an equivalence 
\begin{equation*}
\Phip_{r-1}(\Phip_{r-1})^*a_r^*\Phim_r \simeq 
\Phip_{r-1} (b_{s+1})_* (\Phim_r)^* \Phim_r. 
\end{equation*}
Applying Proposition~\ref{proposition-Phi-relations}\eqref{Phim-Phip-relation} again, 
we find
\begin{equation*}
\Phip_{r-1} (b_{s+1})_* (\Phim_r)^* \Phim_r \simeq 
a_r^*\Phim_r(\Phim_r)^* \Phim_r
\end{equation*}
Since by assumption $\Phi_r$ is left splitting, 
by Proposition~\ref{proposition-base-change-splitting-functor} so is 
$\Phim_r$, and hence 
\begin{equation*}
a_r^*\Phim_r(\Phim_r)^* \Phim_r \simeq a_r^*\Phim_r . 
\end{equation*}
Combining the above and applying Proposition~\ref{proposition-Phi-relations}\eqref{Phim-Phip-relation} 
one more time, all together we have shown
\begin{equation*}
\Phip_{r-1}(\Phip_{r-1})^*\Phip_{r-1}(b_{s+1})_* \simeq \Phip_{r-1}(b_{s+1})_*. 
\end{equation*}

Returning to~\eqref{splitting-induction-1}, we have thus shown 
\begin{equation*}
\Phi_{r-1} \Phi_{r-1}^* \Phi_{r-1}  \simeq 
(\pip_{r-1})_*\Phip_{r-1}(b_{s+1})_*\beta_{s+1}^*.
\end{equation*}
So using Lemma~\ref{lemma-base-change-Phi} we conclude 
\begin{equation*}
\Phi_{r-1} \Phi_{r-1}^* \Phi_{r-1} \simeq \Phi_{r-1} (\pivm_{s+1})_*(b_{s+1})_*\beta_{s+1}^* \simeq 
\Phi_{r-1} . \qedhere
\end{equation*}
\end{proof}

\subsection{Proof of Theorem~\ref{key-theorem}}
By combining Proposition~\ref{proposition-Phi-splitting}, 
Theorem~\ref{theorem-splitting-functors}, and Propositions~\ref{proposition-sos-kerPhiadjoint} and~\ref{proposition-sos-kerPhi}, 
we obtain semiorthogonal sequences 
\begin{align*}
\bL_r(\cA) & \supset
\llangle \im \Phi_r, 
\cA_s(H) \sotimes \Perf(\bG_r), 
\dots, 
\cA_{m-1}((m-s)H) \sotimes \Perf(\bG_r)  \rrangle , 
\\ 
\bLv_{s}(\cB) & \supset  
\llangle 
\cB_{1-n}((r-n)H') \sotimes \Perf(\bGv_s), 
\dots , \cB_{-r}(-H')\sotimes \Perf(\bGv_s)  , \im \Phi_r^* 
\rrangle , 
\end{align*}
such that the functors $\Phi_r$ and $\Phi_r^*$ induce mutually inverse equivalences $\im \Phi_r^* \simeq \im \Phi_r$. 
Hence to prove Theorem~\ref{key-theorem}, it remains to prove that the 
above semiorthogonal sequences generate $\bL_r(\cA)$ and $\bLv_s(\cB)$. 
We do this in Propositions~\ref{proposition-LvsD-generation} and~\ref{proposition-LrC-generation} 
via an inductive argument. 
We start with $\bLv_s(\cB)$, where 
the base case of the induction holds by our assumptions on $\cB$.

\begin{proposition}
\label{proposition-LvsD-generation}
We have 
\begin{equation*}
\label{LvsD-generation}
\bLv_{s}(\cB) =  
\llangle 
\cB_{1-n}((r-n)H') \sotimes \Perf(\bGv_s), 
\dots , \cB_{-r}(-H')\sotimes \Perf(\bGv_s)  , \im \Phi_r^* 
\rrangle  . 
\end{equation*} 
\end{proposition}

\begin{proof}
If $r \geq n$ the right side 
is simply $\im \Phi_r^*$. Since $\Phi_r$ is left splitting by Proposition~\ref{proposition-Phi-splitting}, 
the result amounts to the claim that $\ker \Phi_r = 0$, i.e. that $\Phi_r$ is fully faithful. 
This can be proved by descending induction in $r$. 
The base case $r = N-1$ holds by Setup~\ref{setup-HPD}\eqref{setup-fully-faithful}. 
We leave it to the reader to check that the same induction argument from Proposition~\ref{proposition-Phi-splitting}, but 
using~\eqref{Phi-composition-vanish-2} in place of~\eqref{Phi-composition-vanish-1} 
and Proposition~\ref{proposition-base-change-ff} in place of Proposition~\ref{proposition-base-change-splitting-functor}, 
shows that $\Phi_r^*\Phi_r \simeq \id$, i.e. $\Phi_r$ is fully faithful, for all $r \geq n$. 

For the case $r \leq n$, we again use descending induction. 
Assume the result holds for some $r \leq n$. 
To prove the result for $r-1$, suppose 
$D \in \bLv_{s+1}(\cB)$ lies in 
\begin{equation}
\label{assumption-D}
\llangle 
\cB_{1-n}((r-n-1)H') \sotimes \Perf(\bGv_s), 
\dots , \cB_{1-r}(-H')\sotimes \Perf(\bGv_s)  , \im \Phi_{r-1}^* 
\rrangle^{\perp} .
\end{equation}
We must show $D \simeq 0$. 
For this, it is enough to show $b_{s+1}^!(\pivm_{s+1})^*D \simeq 0$. 
Indeed, the functor $b_{s+1}^! \circ (\pivm_{s+1})^*$ differs from $\beta_{s+1}^* = b_{s+1}^* \circ (\pivm_{s+1})^*$
by the twist by a line bundle by Lemma~\ref{lemma-L-divisor}\eqref{lemma-L-divisor-b}, 
and the functor $\beta_{s+1}^*$ is fully faithful by Lemma~\ref{lemma-xi}. 
By base change the induction hypothesis implies
\begin{multline}
\label{goal-D}
\bLvm_s(\cB) =  \langle 
\cB_{1-n}((r-n)H') \sotimes \Perf(\Fl{r-1}{r}), \dots \\ 
\dots , \cB_{-r}(-H')\sotimes \Perf(\Fl{r-1}{r})  , \im \, (\Phim_r)^* 
 \rangle .
\end{multline} 
Hence it suffices to prove $b_{s+1}^!(\pivm_{s+1})^*D$ lies in the right orthogonal 
to the right side. 
 
First we claim 
\begin{multline}
\label{generation-pi-pullback-D}
(\pivm_{s+1})^*D \in  \langle 
\cB_{1-n}((r-n-1)H') \sotimes \Perf(\Fl{r-1}{r}),  \dots \\
\dots , \cB_{1-r}(-H')\sotimes \Perf(\Fl{r-1}{r})  , \im \, (\Phip_{r-1})^* 
 \rangle^{\perp} .
\end{multline}
If $(\pivm_{s+1})_!$ denotes the left adjoint of $(\pivm_{s+1})^*$,  
then by Lemma~\ref{lemma-base-change-Phi} we have 
\begin{equation*}
(\pivm_{s+1})_! \circ (\Phip_{r-1})^*  \simeq \Phi_{r-1}^* \circ (\pip_{r-1})_!. 
\end{equation*}
Now the claim~\eqref{generation-pi-pullback-D} follows by adjunction and 
our assumption that $D$ lies in~\eqref{assumption-D}. 

It follows from~\eqref{generation-pi-pullback-D} and Proposition~\ref{proposition-im-Phi-tE} 
that $(\pivm_{s+1})^*D$ lies in the right 
orthogonal to the images of the subcategory
\begin{equation*}
\left ( \llangle \cA_{s}(H), \dots, \cA_{m-1}((m-s)H) \rrangle 
\otimes \Perf(\Fl{r-1}{r}) \right )^{\perp} 
\end{equation*}
under the first two terms of the triangle from 
Proposition~\ref{proposition-Phi-relations}\eqref{functor-triangle-2}, and hence 
$(\pivm_{s+1})^*D$ lies in the right orthogonal to 
\begin{equation*}
(b_{s+1})_* (\Phim_r)^* \left( 
\left ( \llangle \cA_{s}(H), \dots, \cA_{m-1}((m-s)H) \rrangle 
\otimes \Perf(\Fl{r-1}{r}) \right )^{\perp} 
\right ). 
\end{equation*}
It then follows from Lemma~\ref{lemma-im-Phim-Phip} that in fact 
$(\pivm_{s+1})^*D$ lies in the right orthogonal to the entire image 
of the functor $(b_{s+1})_* (\Phim_r)^*$. 
Hence 
\begin{equation}
\label{generation-1}
b_{s+1}^!(\pivm_{s+1})^*D \in (\im \, (\Phim_r)^*)^{\perp} . 
\end{equation}

Next note that by Lemma~\ref{lemma-L-divisor}\eqref{lemma-L-divisor-b} 
there is an exact sequence of sheaves on $\bLvm_{s+1}$
\begin{equation*}
0 \to (\cU_r/\cU_{r-1})(-H') \to \cO_{\bLvm_{s+1}} \to (b_{s+1})_*\cO_{\bLvp_{s}} \to 0 . 
\end{equation*}
This implies that the image of the functor $(b_{s+1})_* \colon \bLvp_s(\cB) \to \bLvm_{s+1}(\cB)$ 
applied to 
\begin{equation*}
\llangle 
\cB_{1-n}((r-n)H') \sotimes \Perf(\Fl{r-1}{r}),  
\dots , \cB_{-r}(-H')\sotimes \Perf(\Fl{r-1}{r}) \rrangle 
\end{equation*}
is contained in 
\begin{equation*}
\llangle 
\cB_{1-n}((r-n-1)H') \sotimes \Perf(\Fl{r-1}{r}), 
\dots , \cB_{1-r}(-H')\sotimes \Perf(\Fl{r-1}{r}) \rrangle . 
\end{equation*}
By~\eqref{generation-pi-pullback-D} it thus follows that $b_{s+1}^!(\pivm_{s+1})^*D$ 
lies in 
\begin{equation*}
\llangle 
\cB_{1-n}((r-n)H') \sotimes \Perf(\Fl{r-1}{r}), \dots \\ 
\dots , \cB_{-r}(-H')\sotimes \Perf(\Fl{r-1}{r}) \rrangle^{\perp}. 
\end{equation*}
Together with~\eqref{generation-1}, this shows $b_{s+1}^!(\pivm_{s+1})^*D$ lies 
in the right orthogonal to~\eqref{goal-D}, as required. 
\end{proof}

Up to now, we have not used the condition \eqref{setup-image-A0} of Setup~\ref{setup-HPD}. 
We pause here to prove the claim of Remark~\ref{remark-2c} that this condition can be 
replaced with~\ref{2cn}. 

\begin{lemma}
\label{lemma-2c} 
The conditions \eqref{setup-A}, \eqref{setup-fully-faithful}, \eqref{setup-Bj} of Setup~\ref{setup-HPD} together 
with the condition \textup{(2c$'$)} of Remark~\ref{remark-2c} imply \eqref{setup-image-A0}. 
\end{lemma}

\begin{proof}
Assume \eqref{setup-A}, \eqref{setup-fully-faithful}, \eqref{setup-Bj}, and~\ref{2cn} hold. 
We show the claim of Corollary~\ref{corollary-Ad-B} holds. 
Then by Lemma~\ref{lemma-im-ker-p-gamma}, 
the condition \eqref{setup-image-A0} follows. 

We may freely use everything proved so far, since we only invoked \eqref{setup-A}, \eqref{setup-fully-faithful}, 
and \eqref{setup-Bj} up to now. 
First note that $\bGv_N = S$, $\bLv_N(\cB) = \cB$, and $\bL_0(\cA) = 0$. 
Hence the claim of Corollary~\ref{corollary-B-lc-full} holds by Proposition~\ref{LvsD-generation}. 

To prove the claim of Corollary~\ref{corollary-Ad-B}, note that by~\eqref{setup-fully-faithful} 
the functor $\phi \colon \cB \to \cAd$ is fully faithful and admits a left adjoint $\phi^*$. 
Thus there is a semiorthogonal decomposition $\cAd = \llangle \im \phi, \ker \phi^* \rrangle$, 
so it suffices to show $\ker \phi^* = 0$. 
By the above, we may apply Corollary~\ref{corollary-B-lc-full} in the case $\cB = \cAd$ to conclude 
that the categories $\cAd_j \subset \cAd$ defined by~\eqref{Adj} form a full Lefschetz chain. 
By definition, the primitive components of this Lefschetz chain are given by the 
categories $\fad_j = \gamma^*p^*(\fatw_{N-2+j})$, $2-N \leq j \leq 0$. 
On the other hand, by~\eqref{setup-Bj},~\ref{2cn}, and the claim of Corollary~\ref{corollary-B-lc-full}, the 
category $\cB$ has a full Lefschetz chain with primitive components given by the images of the 
$\fad_j$ under the functor $\phi^*$. 
It follows that the functor $\phi^*$ has trivial kernel if and only if its restriction to $\fad_j$ has 
trivial kernel for all~$2-N \leq j \leq 0$. 
But by Lemma~\ref{lemma-lls-Cd}\eqref{A0-fully-faithful} combined with~\ref{2cn}, 
the restriction of $\phi^*$ to $\fad_j$ is fully faithful. 
\end{proof} 

The following gives the base case for the induction in the proof of generation of $\bL_r(\cA)$.  

\begin{lemma}
We have $\im \Phi_1 = \bL_1(\cA)$. 
\end{lemma}

\begin{proof}
By definition $\bL_1 = \bP$ and $p_1 \colon \bL_1 \to \bP$ is the identity morphism.  
In particular, $\Phi_1$ coincides with the functor $p_{1*} \circ \Phi_1$, 
which we will analyze using Lemma~\ref{lemma-resolution-pr-Phir}. 
Note that $q_{N-1} \colon \bLv_{N-1} \to \bPv$ is isomorphic to the 
projective bundle $\bP(\cS) \to \bPv$, where $\cS$ is the vector bundle 
on $\bPv$ defined by the short exact sequence 
\begin{equation*}
0 \to \cS \to V \otimes \cO \to \cO(H') \to 0. 
\end{equation*}
Under this identification, the bundle $\cW_{N-1}/\cO(-H')$ on $\bLv_{N-1}$ corresponds to 
$\Omega_{\bP(\cS)/\bPv}(H)$. 
By a computation on the projective bundle $\bP(\cS) \to \bPv$, we thus obtain 
for $0 \leq i, t \leq N-2$ an isomorphism 
\begin{equation*}
(q_{N-1})_*(\wedge^i(\cW_{N-1}/\cO(-H'))(-tH)) \simeq 
\begin{cases}
\cO_{\bPv}[-i] & \text{ if } i = t , \\
0 & \text{ else.}
\end{cases}
\end{equation*}
It follows that for $0 \leq i \leq N-2$ the functor $\Phi_{\cK_{1,i}}$ from Lemma~\ref{lemma-resolution-pr-Phir} satisfies 
\begin{equation*}
\Phi_{\cK_{1,i}} \circ  (-\otimes \cO(-tH)) \circ q_{N-1}^* \simeq 
\begin{cases}
[-i] \circ (- \otimes \cO(-iH)) \circ p_* \circ \Phi_\cE  & \text{ if } i = t, \\
0 & \text{ else},  
\end{cases}
\end{equation*}
and hence 
\begin{equation*}
\Phi_{1} \circ (-\otimes \cO(-iH)) \circ q_{N-1}^*  
\simeq (- \otimes \cO(-iH)) \circ p_* \circ \Phi_\cE. 
\end{equation*}
Since $\Phi_{\cE} = \gamma \circ \phi$ it follows from Setup~\ref{setup-HPD}\eqref{setup-image-A0}  
that $\im \Phi_1$ contains the categories $\cA_{0}(-iH)$ for $0 \leq i \leq N-2$. 
The stable subcategory of $\cA$ generated by these categories is all of $\cA$ 
by the assumption that $\cA$ is a moderate Lefschetz category. 
But $\im \Phi_1 \subset \cA$ is a stable subcategory since $\Phi_1$ is left splitting, so we 
conclude $\im \Phi_1 = \cA$. 
\end{proof} 

\begin{proposition}
\label{proposition-LrC-generation}
We have 
\begin{equation*}
\bL_{r}(\cA) =  
\llangle \im \Phi_r, 
\cA_s(H) \sotimes \Perf(\bG_r), 
\dots, 
\cA_{m-1}((m-s)H) \sotimes \Perf(\bG_r)  \rrangle  . 
\end{equation*} 
\end{proposition}

\begin{proof}
Analogous to Proposition~\ref{proposition-LvsD-generation} and left to the reader. 
\end{proof}


\newpage
\addtocontents{toc}{\vspace{\normalbaselineskip}}

\providecommand{\bysame}{\leavevmode\hbox to3em{\hrulefill}\thinspace}
\providecommand{\MR}{\relax\ifhmode\unskip\space\fi MR }
\providecommand{\MRhref}[2]{%
  \href{http://www.ams.org/mathscinet-getitem?mr=#1}{#2}
}
\providecommand{\href}[2]{#2}


\end{document}